\documentclass{amsart}

\usepackage{graphicx}

\usepackage[english]{babel}
\usepackage{hyperref}
\usepackage{subfig}
\usepackage[countmax]{subfloat}
\usepackage{amsmath}
\usepackage{amssymb}
\usepackage{mathtools}
\mathtoolsset{showonlyrefs}
\usepackage{float}
\newtheorem{thm}{Theorem}[section]

\newtheorem{lem}[thm]{Lemma}
\newtheorem{prop}[thm]{Proposition}
\theoremstyle{definition}
\newtheorem{defn}[thm]{Definition}
\theoremstyle{remark}
\newtheorem{rem}[thm]{Remark}
\newtheorem{ex}[thm]{Example}

\numberwithin{equation}{section}

\usepackage[usenames,dvipsnames]{xcolor}

\begin{document}
	
	\title{Asymptotic lines and parabolic points of plane fields in $\mathbb{R}^3$}%
	
	\author{Douglas H. da Cruz}
	\author{Ronaldo A. Garcia}%
	
	
	\begin{abstract}
		In this paper are studied the simplest qualitative properties of asymptotic lines
		of a plane field in Euclidean space. These lines are the integral curves of the null directions
		of the normal curvature of the plane field, on the closure of the hyperbolic
		region, where the Gaussian curvature  is negative.
	When the plane field is completely integrable, these curves coincides with the classical asymptotic lines on  surfaces.
	\end{abstract}
	\maketitle
	\section{Introduction}

		
	The classical work about plane fields is \cite{MR1510041} and the general reference for the purposes of this work is the book \cite{MR1749926}.
	The normal curvature of a plane field can be introduced as in \cite{Euler1760}.
	This and other concepts of the differential geometry of surfaces in $\mathbb{R}^3$ were naturally extended
	for plane fields in $\mathbb{R}^3$ and three manifolds.
	
	The main results of this work are the Theorems \ref{t1}, \ref{t2}, \ref{thm:node-focus},  \ref{t3}, \ref{aci} and \ref{genpar}. 
	
	Motivated by the \cite[Theorem 1]{MR1634428}, \cite{MR1725206}, the Theorems  \ref{t1}, \ref{t2}, \ref{thm:node-focus}, \ref{t3}, \ref{aci} concerns about
	the simplest qualitative properties of asymptotic lines of plane fields in $\mathbb{R}^3$ near a regular parabolic surface and are proved in the section \ref{par}.
	
	In the section \ref{pre}, we give the definitions of a plane field, normal curvature, asymptotic line, parabolic point, Gaussian curvature, mean curvature and others objects that
	will be necessary in the subsequent sections. Furthermore, some preliminaries results are presented.

	Theorem \ref{t1} establishes the behaviour of the asymptotic lines when the asymptotic direction is not tangent to the surface of parabolic points. 
	
	The regular curve $\varphi$ of special parabolic points is characterized by the property that the asymptotic direction is tangent to 
	the parabolic surface. Lemma \ref{lemma:s1} show that the curve $\varphi$ is related with the curve $\widetilde{\varphi}$ of singular 
	points of the Lie-Cartan vector field $\mathcal{X}$, defined in Section \ref{subsection:lie}. In Proposition \ref{prop0} we show  that the integral curves of $\mathcal{X}$ are projected onto
	the asymptotic lines.
	
	Theorem \ref{t2} establishes the behaviour of the asymptotic lines near $\varphi$
	when the singular points of $\widetilde{\varphi}$ are of type saddle, node or focus.
	
	Theorem \ref{thm:node-focus} concerns with the case where there is a transition
	of  node-focus type at a point of $\varphi$.

	In Theorem \ref{t3} is analyzed the case where there is a transition
	of the type saddle-node at a point $r$ of $\varphi$ and Lemma \ref{lemma:spct} shows that this transition occur only if the tangent line of the curve $\varphi$ at $r$ is the asymptotic direction.
	
	In all the above cases, the associated eigenvalues do not cross the 
	imaginary axis. 
	
In 	Theorem  \ref{aci} it is analyzed  the case where, at a point $r$ of $\varphi$, the pair of complex eigenvalues
	crosses the imaginary axis.  This case only occurs if the tangent line of $\varphi$ at $r$ and the asymptotic direction at $r$
	generate the tangent plane of the parabolic surface as shown in Lemma \ref{lemma:spct}.  
	
	In the section \ref{genp},   Theorem \ref{genpar} shows that parabolic points studied in the section \ref{par} are generic in the topological sense.

	\begin{figure}
		\centering
		\subfloat{\includegraphics[trim=0 0 0 0,clip,width=.6\textwidth]{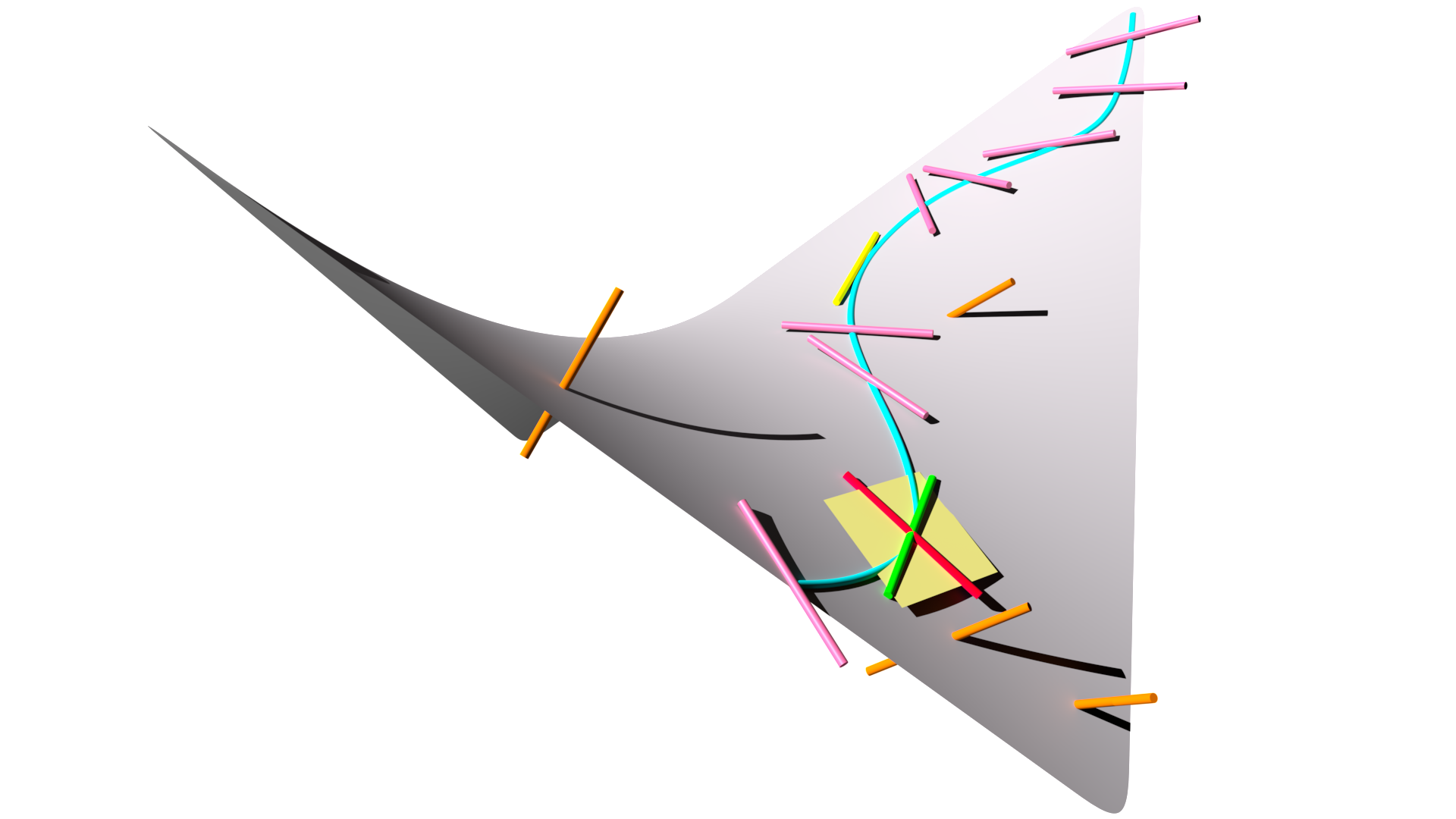}}
		\caption{A curve $\varphi$ of special parabolic points (blue curve) on the
			parabolic surface. A parabolic point is special if the asymptotic direction
			at it is tangent to the parabolic surface. Some asymptotic directions on the parabolic surface are the orange, pink,
			yellow and red lines. The orange asymptotic directions are the asymptotic directions that are not tangent to the
		parabolic surface. The pink, yellow and red asymptotic directions are the asymptotic directions on the curve of special parabolics points.
		 The red asymptotic direction is a more special case, it is the asymptotic direction that together with the green tangent line of $\varphi$ generate the pastel yellow  tangent plane of the parabolic surface which also belongs to the plane field.
		The most special case of this work is the
	yellow asymptotic direction, which is the asymptotic direction that is tangent to the curve of special parabolic points.
 }
		\label{fig:par}
	\end{figure}
	
	For related works about local and global properties of implicit differential equations see for example \cite{MR3192164}, \cite{MR800916},  \cite{MR1297761}, \cite{MR1725206}  and \cite{MR3570306}. For global aspects of extrinsic geometry of non integrable plane fields in dimension three see \cite{MR1749926}, \cite{MR753131} and \cite{gomes2020}.

A connection  of the geometry of plane fields with sub-Riemmanian geometry is presented in 
\cite[Chapter 17]{MR3971262}. See also \cite{MR4371078} and \cite{MR1344763}.

	\section{Geometry of plane fields}
	\label{pre}
	In this paper, the Euclidean space $\mathbb{R}^3$ is endowed with the Euclidean norm $|\cdot|=\langle \cdot,\cdot\rangle^{\frac{1}{2}}$.
	Let $[\cdot,\cdot,\cdot]$ denote the triple product in $\mathbb{R}^3$.
	
	\subsection{Plane field in \texorpdfstring{$\mathbb{R}^3$}{R3}}
	Let $\xi:\mathbb{R}^3\rightarrow\mathbb{R}^3$, $\xi(x,y,z)=(a(x,y,z),b(x,y,z),c(x,y,z))$, be a vector field of class $C^{k}$, where $k\geq 3$.
	
	A point $r=(x,y,z)$ is called a singular point of $\xi$ if $\xi(r)=0$, otherwise it is a regular point.

	A plane field $\Delta$ in $\mathbb{R}^3$, orthogonal to the vector field $\xi$, is defined by the equation $\langle \xi,dr\rangle=0$, where $dr=(dx,dy,dz)$.
	See the Figure \ref{pf}.

	If $r=(x,y,z)$ is a singular point of $\xi$, then the plane of $\Delta$ at $r$ is not defined.
	
	\begin{defn}
		A regular  curve $\gamma\colon I\rightarrow \mathbb{R}^3$ is an integral curve of a      plane field $\Delta$ if $\gamma^\prime(t)$ is orthogonal to $\xi(\gamma(t))$ for every $t\in I$, that is, $\gamma^\prime(t)$ is contained in the plane of $\Delta$ at $\gamma(t)$.  
	\end{defn}
	
	An integral curve of $\Delta$ is also called a  Legendre curve of $\Delta$.
	\begin{defn}
		The curl vector field of $\xi$ will be denoted by $curl(\xi)$, i.e, $curl(\xi)=\left(c_{y}-b_{z},a_{z}-c_{x},b_{x}-a_{y}\right)$ in cartesian coordinates $(x,y,z)$.
	\end{defn}
	\begin{thm}[{\cite[Jacobi Theorem, p. 2]{MR1749926}}]\label{jacobi}
		There exist  a family of surfaces orthogonal to $\xi$ if, and only if, $\langle\xi,curl(\xi)\rangle\equiv0$.
	\end{thm}
	\begin{defn}
		A plane field $\Delta$ is said to be completely integrable if $\langle\xi,curl(\xi)\rangle\equiv0$. A surface of the
		family of surfaces orthogonal to $\xi$ is called  an integral surface.
	\end{defn}
	\begin{rem}
		Set the 1-form $\eta=\langle \xi,dr\rangle$. Then $\eta\wedge d\eta=\langle\xi,curl(\xi)\rangle dx\wedge dy\wedge dz$.
		Theorem \ref{jacobi} is a special case of the Frobenius Integrability Theorem for differential forms.
	\end{rem}
	\begin{figure}
		\captionsetup[subfigure]{width=.3\linewidth}
		\centering
		\subfloat{\includegraphics[width=.6\textwidth]{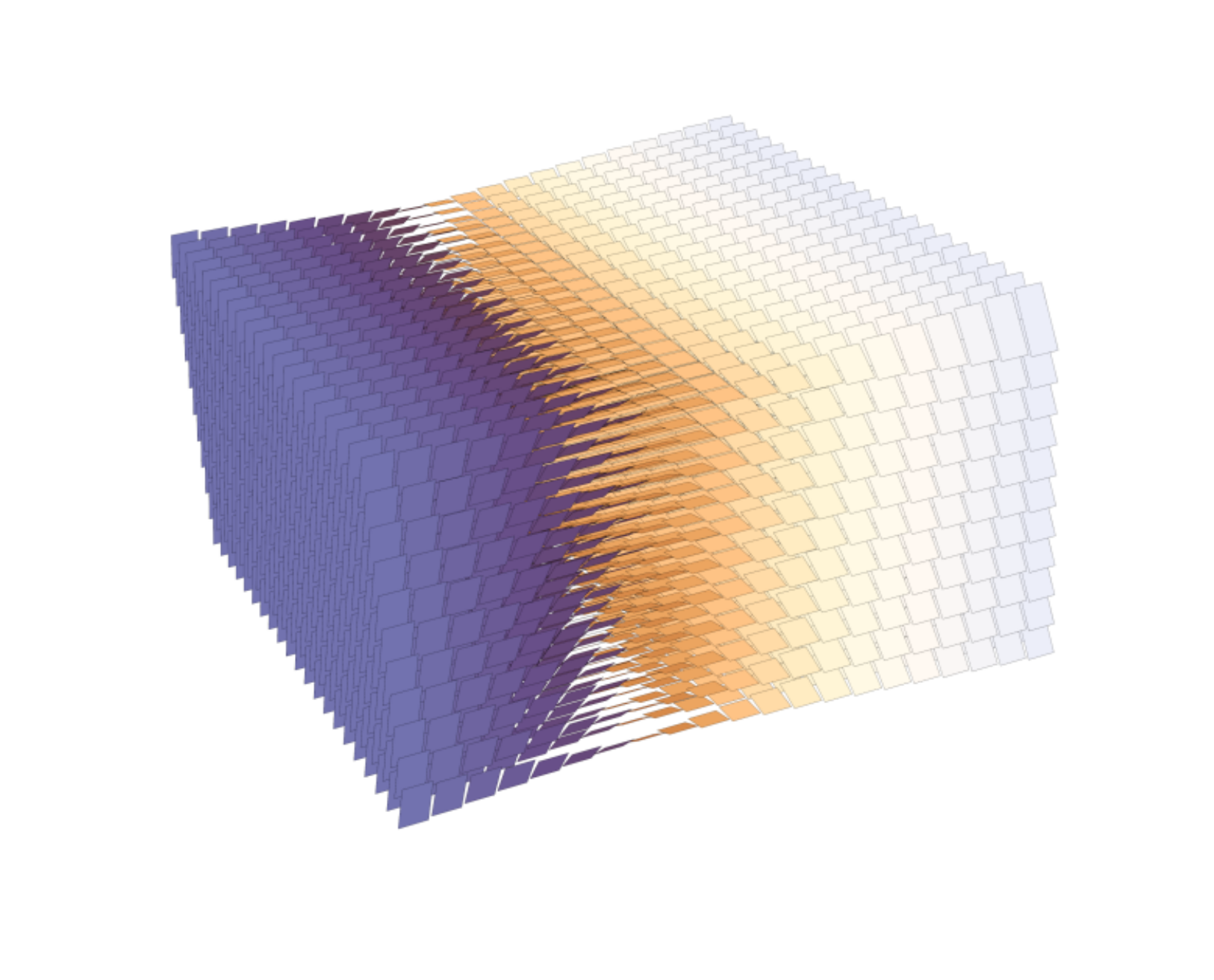}\label{pf1}}
		\caption{Plane field orthogonal to the vector field $\xi=(-y,0,1)$. In this case, the equation $\langle \xi,dr\rangle=0$ becomes $dz-ydx=0$.}
		\label{pf}
	\end{figure}
	\subsection{Normal curvature of a plane field}
	\begin{defn}[{\cite[p. 8]{MR1749926}}]
		The normal curvature $k_{n}$ of a plane field, in the direction $dr$, is defined by
		\begin{equation}\label{nc}
		k_{n}=\frac{\langle d^2r,\xi \rangle}{\langle dr,dr\rangle}=-\frac{\langle dr, d\xi \rangle}{\langle dr,dr\rangle}.
		\end{equation}
	\end{defn}
	
	This definition agrees with
	the classical one given by L. Euler in  \cite{Euler1760}. The geometrical
	interpretation of $k_{n}$ is given by Proposition \ref{kng}.
	
	\begin{prop}[{\cite[Proposition 1.2]{alacyr}}]
		Let $\Delta$ be a plane field. Then in each plane of $\Delta$
		there exists two orthogonal directions at which the normal curvature $k_{n}$
		attains the extreme values, minimal and maximal.
	\end{prop}
	\begin{defn}
		The minimal (resp. maximal) principal curvature will be denoted by $k_{1}$ (resp. $k_{2}$).
		The principal direction associated to $k_{1}$ (resp. $k_{2}$) will be denoted by $\mathcal{P}_{1}$ (resp. $\mathcal{P}_{2}$).
	\end{defn}

	The Euler curvature formula holds for planes fields, see Theorem \ref{ecf}:
	\begin{equation}
	k_{n}=k_{1}cos^2(\theta)+k_{2}sin^2(\theta),
	\end{equation}
	
	\noindent where  $\theta$ is the angle between $dr$
	and the principal direction associated to $k_{1}$. 
	
	\subsection{Geodesic curvature and geodesic torsion of a plane field}
	
	The geodesic curvature $k_{g}$ of the plane field, in the direction $dr$, is defined by
	\begin{equation}\label{gc}
	k_{g}=\frac{[\xi,dr,d^2r]}{\langle dr,dr\rangle^{\frac{3}{2}}}.
	\end{equation}
	See  \cite[p. 14]{MR1749926}.
	
	The geodesic torsion $\tau_{g}$ of the plane field, in the direction $dr$, is defined by
	\begin{equation}\label{gt}
	\tau_{g}=\frac{[dr,\xi,d\xi]}{\langle dr,dr\rangle}.
	\end{equation} See {\cite[p. 50]{MR1749926}}
	
	The geodesic torsion formula for planes fields is given by Theorem \ref{proptau}.
	
\subsection{Gaussian curvature and Mean curvature of a plane field}	

	\begin{defn}[{\cite[p. 11]{MR1749926}}]
		\label{defHK}
		The Gaussian curvature $\mathcal{K}_{G}$ of a plane field is defined by $\mathcal{K}_{G}=k_{1}k_{2}$, where $k_{1}$ and $k_{2}$
		are the principal curvatures of the plane field.
		
		The Mean curvature $\mathcal{H}_{M}$ of a plane field  is defined by  $\mathcal{H}_{M}=\frac{k_{1}+k_{2}}{2}$.
		
	A point $p$ is called, respectively, elliptic, parabolic or hyperbolic
when $\mathcal{K}_{G}(p)>0$, $\mathcal{K}_{G}(p)=0$ or $\mathcal{K}_{G}(p)<0$.

The set of hyperbolic points (resp. elliptic points) will be denoted by $\mathbb{H}$ (resp. $\mathbb{E}$) and is called a hyperbolic region of $\Delta$ (resp. elliptic region of $\Delta$). The set of parabolic points will be denoted by $\mathbb{P}$.
	\end{defn}

	\begin{rem}
		When the plane field is completely integrable, the normal, geodesic curvature and  geodesic torsion above coincide with that of curves on surfaces.  The definition of the geodesic curvature (resp. geodesic torsion) of a surface can be found in \cite[p. 271]{MR1891533} and \cite[p. 542]{MR2253203} (resp. \cite[p. 545]{MR2253203}).
	\end{rem}
	\subsection{Asymptotic lines and parabolic points of a plane field}

The directions where $k_{n}=0$ are called asymptotic directions of
	the plane field
	and therefore are defined by the following implicit differential equations:
	\begin{equation}\label{eqla}
	\begin{split}
	\langle\xi,dr\rangle&=adx+bdy+cdz=0,\\
	\langle d\xi,dr\rangle&=a_{x}dx^2+(a_{y}+b_{x})dxdy+b_{y}dy^2\\
	&+(a_{z}+c_{x})dxdz+(b_{z}+c_{y})dydz+c_{z}dz^2=0.
	\end{split}
	\end{equation}
		A solution $dr=(dx,dy,dz)$ of \eqref{eqla} is an asymptotic direction. A curve $\gamma$ in $\mathbb{R}^3$ is an asymptotic line of the plane field if $\gamma$ is an integral curve of \eqref{eqla}.
		
		The system \eqref{eqla} will be called    the implicit differential equation  of the asymptotic lines of the plane field.
	
	
The line fields of asymptotic directions will be denoted by
$\ell_{1}$ and $\ell_{2}$. They are called asymptotic line fields. 
	
		
	\begin{figure}
		\captionsetup[subfigure]{width=.3\linewidth}
		\centering
		\subfloat{\includegraphics[width=.4\textwidth]{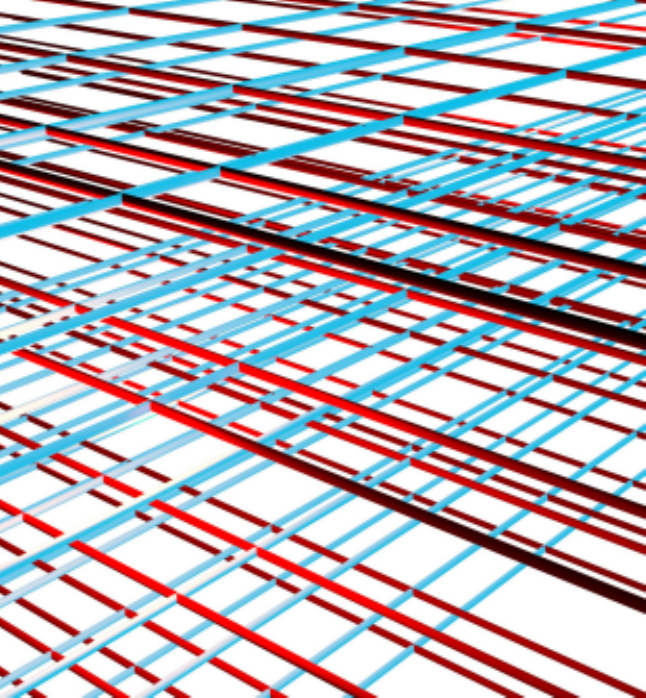}\label{hip1}}
		\caption{The two foliations of asymptotic lines in the hyperbolic region, one coloured with red and the other with blue.}
		\label{hip}
	\end{figure}
	
		As a consequence of \cite[Panov. Corollary of Theorem 4]{MR1725930}, a closed asymptotic line without parabolic points of a surface  
	in $\mathbb{R}^3$ cannot
	have a convex projection in any plane. See also \cite{douglas2020}.  However, for plane fields  we have no restrictions.
	\begin{ex}
		The circle in $\mathbb{R}^3$ given by $x^2+y^2=1$, $z=0$, is an asymptotic line without parabolic points of the plane field $\Delta$ orthogonal to the vector field
		$\xi=(a,b,c)$, where $a=x^2yz+y^3z-x^2y-y^3+xz-2yz+y$, $b=x^3-x^3z-xy^2z+xy^2+2xz+yz-x$ and $c=-x^2-y^2$. The plane field $\Delta$ is not completely integrable.\end{ex}
		
		The ordered pair $\{\ell_{1},\ell_{2}\}$ is well defined
in the hyperbolic region $\mathbb{H}$, where these directions are real, see
the Proposition \ref{fam}.

The asymptotic foliations of $\Delta$ are the integral foliations 
$\mathcal{A}_{1}$ of
$\ell_{1}$ and $\mathcal{A}_{2}$ of
$\ell_{2}$; they fill out the hyperbolic region $\mathbb{H}$, 
 see
the Proposition \ref{fam}.


	\begin{prop}
		If a straight line is a integral curve of a plane field $\Delta$, then it is also an asymptotic line of $\Delta$.
	\end{prop}
	\begin{proof}
		Let $\ell(t)=p_0+tv$ be a parametrization of the straight line with $v\in \Delta$. Since $\langle\xi(\ell(t)),v\rangle=0$ it follows that
		$\langle d\xi(\ell(t))v,v\rangle  
		=0$. Then the straight line $\ell$ is an asymptotic line of $\Delta$.
	\end{proof}
	\begin{prop}\label{prop3}
		Given a plane field $\Delta$, let $h:\mathbb{R}^3\rightarrow \mathbb{R}$ be a signal defined smooth function. Then a curve $\gamma$ is an asymptotic line of $\Delta$ if, and only if,
		$\gamma$ is an asymptotic line of the plane field $\widetilde{\Delta}$ orthogonal to the vector field $\widetilde{\xi}=h\xi$.
	\end{prop}
	\begin{proof}
		The implicit differential equations of the asymptotic lines of $\widetilde{\Delta}$ are given by
		\begin{equation}
		\langle\widetilde{\xi},dr\rangle=h\langle\xi,dr\rangle=0, \ \
		\langle d\widetilde{\xi}(dr),dr\rangle=dh(dr)\langle\xi,dr\rangle+h \langle d\xi(dr),dr\rangle=0.
		\end{equation}
		Then $\gamma$ is an asymptotic line of $\Delta$ if, and only if,
		$\gamma$  is an asymptotic line of the plane field $\widetilde{\Delta}$.
	\end{proof}
	\begin{prop}\label{laC}
		Let $\gamma$ be a integral curve, parameterized by arc length $s$, with nonvanishing curvature, of a plane field $\Delta$ and let $\{T,N,B=T\wedge N\}$ be the Frenet orthonormal frame
		associated to $\gamma$ with the Frenet equations $T'=kN$, $N'=-kT+\tau B$, $B'=-\tau N$, where $k$ is the curvature of $\gamma$ and $\tau$ is the torsion of $\gamma$.  Then $k_{n}^2+k_{g}^2=k^2$ and the following conditions are equivalent:
		\begin{itemize}
			\item $\gamma$  is an asymptotic line of $\Delta$.
			\item $k_{g}=\pm k$, where $k_{g}(s)$ is the geodesic curvature \ref{gc} of $\Delta$ evaluated at the direction of $T(s)$.
			\item for every $s$, the osculating plane of $\gamma$ at $\gamma(s)$ is the plane of $\Delta$ at $\gamma(s)$.
		\end{itemize}
		Furthermore, if $\gamma$  is an asymptotic line of $\Delta$, then $\tau_{g}=\tau$, where $\tau_{g}(s)$ is the geodesic torsion \ref{gt} of $\Delta$ evaluated at the direction of $T(s)$.
	\end{prop}
	\begin{proof}
		At $\gamma$, $\xi(s)=cos(\theta(s))N(s)+sin(\theta(s))B(s).$
		Then, at $\gamma$, we have that
		$d\xi(\gamma')=-kcos(\theta)T-(\theta'+\tau)sin(\theta)N+(\theta'+\tau)cos(\theta)B$.
		It follows that the normal \ref{nc} and geodesic curvatures \ref{gc}, and geodesic torsion \ref{gt} of $\Delta$, evaluated at the direction $T$, are given by
		$k_{n}=kcos(\theta)$, $k_{g}=ksin(\theta)$, $\tau_{g}=\tau+\theta'$ and so $k_{n}^2+k_{g}^2=k^2$. If $\gamma$  is an asymptotic line, which means that $k_{n}=0$, then $cos(\theta)=0$ and so $sin(\theta)=\pm1$, which gives $k_{g}=\pm k$ (or, by the formula $k_{n}^2+k_{g}^2=k^2$, $k_{g}=\pm k$). If $k_{g}=\pm k$, then $cos(\theta)=0$ and so $\xi=\pm B$. It follows that, for every $s$, the osculating plane of $\gamma$ at $\gamma(s)$ is the plane of $\Delta$ at $\gamma(s)$.
		If for every $s$, the osculating plane of $\gamma$ at $\gamma(s)$ is the plane of $\Delta$ at $\gamma(s)$, then $\xi(s)$ is parallel to $B(s)$
		and then $cos(\theta(s))=0$ for every $s$. It follows that $k_{n}=0$.
		
		If $\gamma$ is asymptotic line, then $\theta$ is constant, and it follows that $\tau_{g}=\tau$.
	\end{proof}
	\begin{rem}
		The Proposition \ref{laC} is the version for plane fields of the equivalent results for surfaces in $\mathbb{R}^3$ that are stated, for example, in \cite[p. 196]{MR532832}.
	\end{rem}
	\begin{lem}\label{propeqlaV}
		Let $\Delta$ be a plane field such that $r=(x,y,z)$ is not a singular point of $\xi=(a,b,c)$. Then we can choose a coordinate system such that $c(r)\neq 0$. In this
		coordinate system, the implicit differential equations of the asymptotic lines, in a neighbourhood of $r$,
		becomes
		\begin{equation}
		\begin{split}\label{eqlaV}
		dz=-\left(\frac{a}{c}\right)dx-\left(\frac{b}{c}\right)dy, \quad edx^2+2fdxdy+gdy^2=0,
		\end{split}
		\end{equation}
		\noindent where
		\begin{equation}
		\begin{split}
		e&=a_{x}-\frac{(a_{z}+c_{x})a}{c}+\frac{a^2c_{z}}{c^2}, \ \ g=b_{y}-\frac{(b_{z}+c_{y})b}{c}+\frac{b^2c_{z}}{c^2},\\
		f&=\frac{a_{y}+b_{x}}{2}-\frac{(a_{z}+c_{x})b}{2c}-\frac{(b_{z}+c_{y})a}{2c}+\frac{abc_{z}}{c^2}.
		\end{split}
		\end{equation}
		Furthermore, in this neighbourhood, the parabolic set of $\Delta$ is given by $eg-f^2=0$.
	\end{lem}
	\begin{proof}
		Since $\xi(r)\neq(0,0,0)$, we can choose a coordinate system such that $c(r)\neq0$.
		In a neighbourhood of $r$, the equation $\langle\xi,dr\rangle=0$ of \eqref{eqla} can be solved for $dz$ and then
		we get the first equation of \eqref{eqlaV}.
		Replace this $dz$ on the equation $\langle d\xi,dr\rangle=0$ of \eqref{eqla} to get the second equation of \eqref{eqlaV}.
		
		By \eqref{eqlaV}, at a point $(x,y,z)$, all directions are asymptotic directions if, and only if, $e(x,y,z)=f(x,y,z)=g(x,y,z)=0$
		and, at a point $(x,y,z)$, the asymptotic directions coincides if, and only if, $e(x,y,z)g(x,y,z)-(f(x,y,z))^2=0$.
	\end{proof}
	\begin{defn}
		Let $\Delta$ be a plane field satisfying the assumptions of the Lemma \ref{propeqlaV}. Set $\mathcal{K}=eg-f^2$ and  
		$\mathcal{H}=-\left(\frac{e+g}{2}\right)$.
	\end{defn}

	\begin{prop}[{\cite[p. 11]{MR1749926}}]\label{gaussyu}
		Let $\Delta$ be a plane field satisfying the assumptions of the Lemma \ref{propeqlaV}.
		\begin{itemize}
			\item $\mathcal{K}(0,0,0)=\mathcal{K}_{G}(0,0,0)$ and $\mathcal{H}(0,0,0)=\mathcal{H}_{M}(0,0,0)$,
			\item $\mathcal{K}=0$ if, and only if, $\mathcal{K}_{G}=0$ and $\mathcal{H}=0$ if, and only if, $\mathcal{H}_{M}=0$.
		\end{itemize}
	\end{prop}
	
	\begin{prop}\label{fam}
	 The asympotic directions are well defined in the hyperbolic region $\mathbb{H}$, the
directions are real. 
	The asymptotic foliations  fill out the hyperbolic region $\mathbb{H}$.
	Locally, the asymptotic lines in $\mathbb{H}$ are as show in the Figure \ref{hip}.
	
	\end{prop}
	\begin{proof}
		Let $r=(x,y,z)$ be a point of a open subset of $\mathbb{H}$. By Lemma \ref{propeqlaV}, the implicit differential equations of the asymptotic lines, in a neighbourhood of $r$,
		are given by \eqref{eqlaV}. Since $r\in\mathbb{H}$, in a neighbourhood $\Lambda$ of $r$, the functions $e$, $f$, $g$ does not vanishes simultaneously. Without loss of generality, suppose
		that $e$ does not vanishes at $\Lambda$. Then we can solve the equation $edx^2+2fdxdy+gdy^2=0$ for $dy$ to get
		\begin{equation}
		dy=\left(\frac{-f\pm\sqrt{-\mathcal{K}}}{e}\right)dx.
		\end{equation}
		By \eqref{eqlaV},
		\begin{equation}
		dz=\left[-\left(\frac{a}{c}\right)-\left(\frac{-bf\pm b\sqrt{-\mathcal{K}}}{ce}\right)\right]dx.
		\end{equation}
		This defines the following two vector fields $Z_{+}(x,y,z)$ and $Z_{-}(x,y,z)$ in $\Lambda$:
		\begin{equation}
		Z_{\pm}=\left(1,\frac{dy}{dx},\frac{dz}{dx}\right)=\left(1,\left(\frac{-f\pm\sqrt{-\mathcal{K}}}{e}\right),-\left(\frac{a}{c}\right)-\left(\frac{-bf\pm b\sqrt{-\mathcal{K}}}{ce}\right)\right).
		\end{equation}
		Since $\mathcal{K}$ never vanishes, then, at each point of $\Lambda$, the integral curves of $Z_{+}$ and $Z_{-}$ are transversal.
	\end{proof}
	\begin{rem}
		We can apply the Tubular Flow Theorem (see, for example, \cite[Tubular Flow Theorem, p. 40]{MR669541}) for one of the vector fields $Z_{\pm}$, but we cannot apply it directly for both vector fields simultaneously.
	\end{rem}
	\begin{rem}
		The Proposition \ref{fam} is the version for asymptotic lines of plane fields of the \cite[Theorem 2.2, p. 166]{zbMATH00043670}.
	\end{rem}
	\begin{lem}\label{propetaV}
		Let $\Delta$ be a plane field satisfying the assumptions of Lemma \ref{propeqlaV}. If $(0,0,0)\in\mathbb{H}\cup\mathbb{P}$, then
		in a neighbourhood of $(0,0,0)$, $\xi=(a,b,c)$ where
		\begin{equation}\label{etaV}
		\begin{split}
		a&=a_{2}y+a_{3}z+(a_{11}x^2+a_{12}xy+a_{13}xz+a_{22}y^2+a_{23}yz+a_{33}z^2)\\
		&+(a_{111}x^3+a_{112}x^2y+a_{113}x^2z+a_{122}xy^2+a_{123}xyz+a_{133}xz^2+a_{222}y^3\\
		&+a_{223}y^2z+a_{233}yz^2+a_{333}z^3)+\sum_{i+j+k=4}a_{ij}^{k}x^iy^jz^k+\mathcal{O}^5(x,y,z),\\
		b&=b_{1}x+b_{2}y+b_{3}z+(b_{11}x^2+b_{12}xy+b_{13}xz+b_{22}y^2+b_{23}yz+b_{33}z^2)\\
		&+(b_{111}x^3+b_{112}x^2y+b_{113}x^2z+b_{122}xy^2+b_{123}xyz+b_{133}xz^2+b_{222}y^3\\
		&+b_{223}y^2z+b_{233}yz^2+b_{333}z^3)+\sum_{i+j+k=4}b_{ij}^{k}x^iy^jz^k+\mathcal{O}^5(x,y,z),\\
		c&=1+(c_{11}x^2+c_{12}xy+c_{13}xz+c_{22}y^2+c_{23}yz+c_{33}z^2)\\
		&+(c_{111}x^3+c_{112}x^2y+c_{113}x^2z+c_{122}xy^2+c_{123}xyz+c_{133}xz^2+c_{222}y^3\\
		&+c_{223}y^2z+c_{233}yz^2+c_{333}z^3)+\sum_{i+j+k=4}c_{ij}^{k}x^iy^jz^k+\mathcal{O}^5(x,y,z).
		\end{split}
		\end{equation}
		Furthermore, the implicit differential equation of the
		asymptotic lines evaluated at $(0,0,0)$ are given by
		\begin{equation}\label{eqlaVTh}
		dz=0, \ \ (a_{2}+b_{1})dxdy+b_{2}dy^2=0.
		\end{equation}
		\noindent $\mathcal{K}(0,0,0)=-(a_{2}+b_{1})^2$ and
		\begin{equation}
		\nabla \mathcal{K} (0,0,0)=(2b_{2}a_{11},(a_{3}b_{1}+a_{12})b_{2},[a_{13}-(a_{3})^2]b_{2}).
		\end{equation}
	\end{lem}
	\begin{proof}
		By the Proposition \ref{prop3} we can suppose that $\xi$ is unitary. By Lemma \ref{propeqlaV}, we can choose a coordinate system in $\mathbb{R}^3$ such that
		$c(0,0,0)\neq0$. Moreover, without loss of generality, this coordinate system can be taken such that $\xi(0,0,0)=(0,0,1)$.
		Then $c_{x}(0,0,0)=c_{y}(0,0,0)=c_{z}(0,0,0)=0$. By \ref{propeqlaV}, the implicit differential
		equation of the asymptotic lines evaluated at $(0,0,0)$ becomes
		\begin{equation}\label{eqlaMC}
		dz=0, \ \ a_{x}(0,0,0)dx^2+\left(\frac{a_{y}(0,0,0)+b_{x}(0,0,0)}{2}\right)dxdy+b_{y}(0,0,0)dy^2=0.
		\end{equation}
		There is a rotation in $\mathbb{R}^3$ around the $z$ axis that makes $(dx,0,0)$ be one asymptotic direction at $(0,0,0)$, that is,
		$a_{x}(0,0,0)=0$. It follows that
		$\xi=(a,b,c)$, where
		$a$, $b$ and $c$ are given by \eqref{etaV} and
		that the implicit differential equations of the
		asymptotic lines \eqref{eqlaV} evaluated at $(0,0,0)$ are given by \eqref{eqlaVTh}.
		
		Then, $e(0,0,0)=0$, $f(0,0,0)=a_{2}+b_{1}$, $g(0,0,0)=b_{2}$ and $\mathcal{K}(0,0,0)=-(a_{2}+b_{1})^2$.
	\end{proof}
	
	\begin{rem}
	In \cite{MR1749926} there is defined the following Mean curvature
	of first kind:
	 $\mathcal{H}_{1}=-\left(\frac{a_{x}+b_{y}}{2}\right)$, related
	 with the divergence of the vector field $\xi$. It follows
	 that $\mathcal{H}\neq\mathcal{H}_{1}$, but 
	 $\mathcal{H}(0,0,0)=\mathcal{H}_{1}(0,0,0)$.
	\end{rem}
	
	\begin{prop}\label{prop10}
		If $(0,0,0)\in\mathbb{P}$, then $a_{2}=-b_{1}$. If $b_{2}\neq0$, then at $(0,0,0)$ the two asymptotic directions coincides with the asymptotic direction
		$\mathcal{A}=\left(dx,0,0\right)$, $dx\neq0$, where, without loss of generality,
		we can assume $dx=1$. If $b_{2}=0$, then at $(0,0,0)$, all directions in the plane of $\Delta$ at $(0,0,0)$ are asymptotic directions.
	\end{prop}
	\begin{proof}
		By Lemma \ref{propetaV}, $\xi=(a,b,c)$ where
		$a$, $b$ and $c$ are given by \eqref{etaV}. The implicit differential equations of the
		asymptotic lines \eqref{eqlaV} evaluated at $(0,0,0)$ are given by \eqref{eqlaVTh}. Since $\mathcal{K}(0,0,0)=-(a_{2}+b_{1})^2$, it follows
		that $a_{2}=-b_{1}$. The equations \eqref{eqlaVTh} becomes $dz=0$, $b_{2}dy^2=0$. If $b_{2}\neq0$, then
		the asymptotic direction at $(0,0,0)$ is given by $\mathcal{A}=(dx,0,0)$.  If $b_{2}=0$, then all the directions $(dx,dy,0)$  are
		asymptotic directions. Note that the plane of $\Delta$ at $(0,0,0)$ is given by $dz=0$.
	\end{proof}
	\subsection{Principal directions}
	\label{pdu}
	The principal directions of a plane field $\Delta$ are defined by the following system of implicit differential equations, see \cite{MR1749926} and \cite{gomes2020}:

	\begin{equation}\label{eqlc}
	\langle\xi,dr\rangle=0, \ \ 2[d\xi,dr,\xi]+\langle curl(\xi),\xi\rangle\langle dr,dr\rangle=0.
	\end{equation}
	\begin{defn}
		The equations \eqref{eqlc} are said to be the implicit differential equations of the principal curvature lines of $\Delta$.
	\end{defn}
	\begin{defn}
		A solution $dr=(dx,dy,dz)$ of \eqref{eqlc} is called of principal direction of $\Delta$. A curve $\gamma$ in $\mathbb{R}^3$ is a principal curvature line of $\Delta$ if $\gamma$ is a integral curve of \eqref{eqlc}.
	\end{defn}
	\begin{rem}
		If $(x,y,z)$ is a singular point of $\xi$, then the principal directions are not defined in $(x,y,z)$.
	\end{rem}
	\begin{prop}\label{eqlc2}
		The second equation of \eqref{eqlc} is equivalent to $2\tau_{g}+\langle curl(\xi),\xi\rangle=0$. Furthermore, let $\tau_{g,i}$, $i=1,2$,
		be the geodesic torsion evaluated at the principal direction associated to the principal curvature $k_{i}$. Then $\tau_{g,1}=\tau_{g,2}$.
	\end{prop}
	\begin{proof}
		By the definition of geodesic torsion \ref{gt}, it is clear that the second equation of \eqref{eqlc} is equivalent to $2\tau_{g}+\langle curl(\xi),\xi\rangle=0$. It follows from the two equations $2\tau_{g,i}+\langle curl(\xi),\xi\rangle=0$, $i=1,2$, that $\tau_{g,1}-\tau_{g,2}=0$.
	\end{proof}
	\begin{thm}[Euler curvature formula for a plane field]\label{ecf}
		Let $\Delta$ be a plane field, then
		\begin{equation}\label{ecfor}
		k_{n}=k_{1}cos^2(\theta)+k_{2}sin^2(\theta),
		\end{equation}
	\end{thm}
	\noindent where $\theta$ is the angle between the direction $dr$ and the principal direction
	associated with $k_{1}$.
	\begin{proof}
		The proof given in \cite{Euler1760} holds for plane fields. For different proofs, see \cite[p. 12]{MR1749926} and the Section \ref{darbouxframe}.
	\end{proof}
	\begin{prop}\cite[Section 1.2]{MR1749926}
		The equation \eqref{eqlc} defines two principal directions at every point that is not a partially umbilic point.
		At a partially umbilic point, all directions are principal directions.
	\end{prop}
	\begin{prop}
		If $(0,0,0)\in\mathbb{P}$, then one asymptotic direction at $(0,0,0)$  is a principal direction.
	\end{prop}
	\begin{proof}
		In $(0,0,0)$ the implicit differential equations of the principal curvatures lines of the plane field $\Delta$ are given by
		$b_{2}dxdy=0$ and $dz=0$ and then one principal direction in $(0,0,0)$ is given by
		$(dx,0,0)=\mathcal{A}$. If $b_{2}=0$, then all directions in $\Delta$ are principal directions and then $(0,0,0)$ is a partially umbilic point of $\Delta$.
	\end{proof}
	\subsection{Geometrical interpretation of the normal curvature of a plane field}
	\begin{prop}\label{kng}
	The normal curvature $k_n$ of a plane field $\Delta$ evaluated at a point $P$ 
and in the direction $dr$ is 
the curvature $k$, evaluated at $P$, of a plane curve $\gamma$,
which is the integral curve
of the line field $\ell$ defined by the intersection of the plane
$\mathcal{N}$ generated by $dr$ and $\xi(P)$ with the plane field $\Delta$,
see the Figure  \ref{cneuler}.
	
	\end{prop}	
	\begin{proof}
		There is no loss of generality in assuming that $\langle dr,dr\rangle=1$ and that $\xi$ is unitary. Let $\zeta^{\bot}$ be the projection of $\xi$ onto the plane $\mathcal{N}$:
		\begin{equation}\label{igcn}
		\zeta^{\bot}(x,y,z)=\xi(x,y,z)-\langle \xi(x,y,z),(\xi(P)\wedge dr)\rangle (\xi(P)\wedge dr).
		\end{equation}
		Let $\zeta$ be the vector field at $\mathcal{N}$ orthogonal to $\zeta^{\bot}$, that is, $\langle \zeta^{\bot},\zeta \rangle\equiv0$ and $\langle (\xi(P)\wedge dr),\zeta\rangle\equiv0$.
		By equation \eqref{igcn}, we have that $\langle \xi,\zeta \rangle\equiv0$.
		It follows that $\zeta$ defines the straight line field $\ell$ and $\zeta(P)$ is parallel to $dr$. Now, $\gamma$ is a curve such that $\gamma(0)=P$, $\gamma'(t)=\zeta(\gamma(t))$ and $\gamma'(0)=dr$.
		It follows that the normal vector of $\gamma$ at $P$ is $\xi(P)$ and, at $\gamma$, $\langle \xi,\gamma'\rangle\equiv0$.
		Then the curvature $k$ of $\gamma$ at $P$ is given by $k=\langle \xi(P),\gamma''(0)\rangle$. Since $\langle \xi(\gamma),\gamma'\rangle\equiv0$, then, at $\gamma$,
		$\langle \xi,\gamma''\rangle=-\langle d\xi(\gamma'),\gamma'\rangle$. Then, at $P$,
		\begin{equation}
		k=-\langle d\xi(\gamma'(0)),\gamma'(0)\rangle=-\langle d\xi(dr),dr\rangle=k_{n},
		\end{equation}
		\noindent where $k_{n}$ is the normal curvature defined by equation  \eqref{nc}, evaluated at $P$ in the direction $dr$.
	\end{proof}
	\begin{figure}
		\captionsetup[subfigure]{width=.3\linewidth}
		\centering
		\subfloat[][Draw the plane of $\Delta$ at a point $P$ and draw the vector $\xi(P)$ orthogonal to it.]{\includegraphics[width=.3\linewidth]{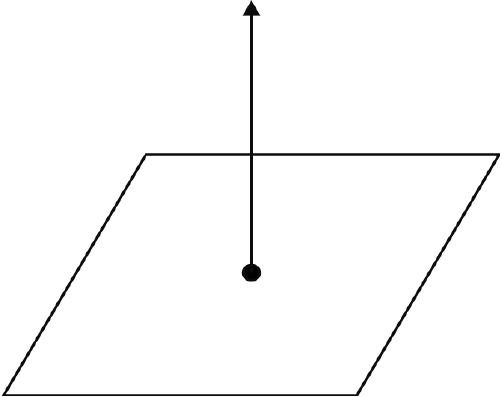}\label{cn1}}
		\qquad \qquad
		\subfloat[][Draw a direction $dr$ in the plane.]{\includegraphics[width=.3\linewidth]{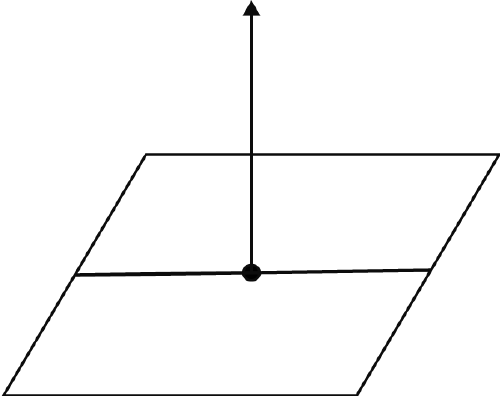}\label{cn2}}
		\\
		\subfloat[][Draw the plane $\mathcal{N}$ generated by $dr$ and $\xi(P)$.]{\includegraphics[width=.3\linewidth]{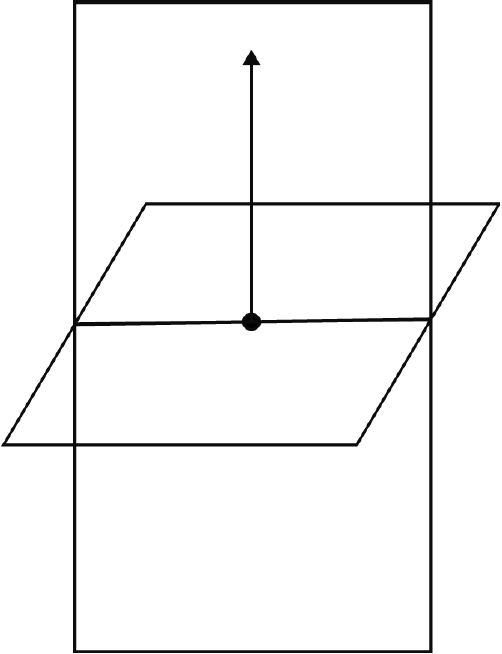}\label{cn3}}
		\qquad \qquad
		\subfloat[][Let $\ell$ be the   line field defined by the intersection of $\mathcal{N}$ with $\Delta$, see Figure \ref{int}. ]{\includegraphics[width=.3\linewidth]{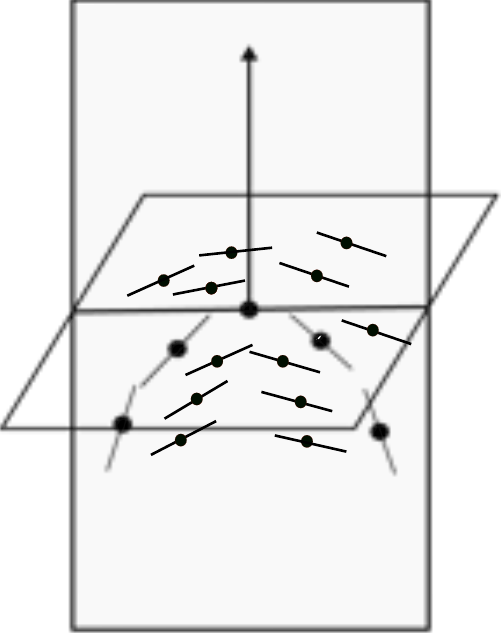}\label{cn4}}
		\\
		\subfloat[][Let $\gamma$ be the integral plane curve of $\ell$ such that $\gamma(0)=P$ and $\gamma'(0)=dr$. See Figure \ref{int}]{\includegraphics[width=.3\linewidth]{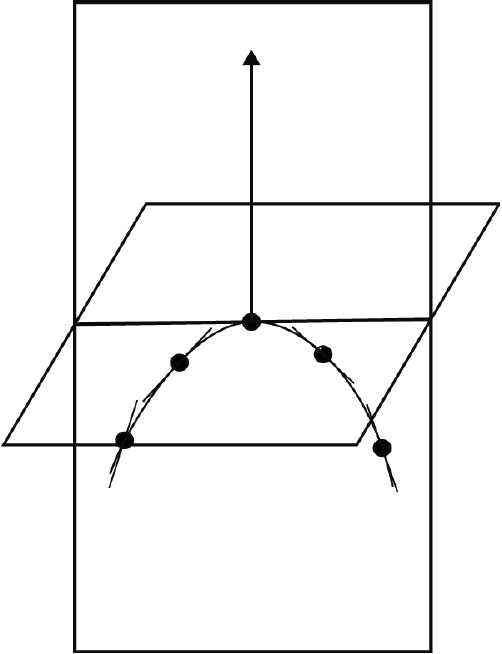}\label{cn5}}
		\qquad \qquad
		\subfloat[][The normal curvature $k_{n}$ of $\Delta$, at the direction $dr$, is the curvature of $\gamma$ at $P$.]{\includegraphics[width=.3\linewidth]{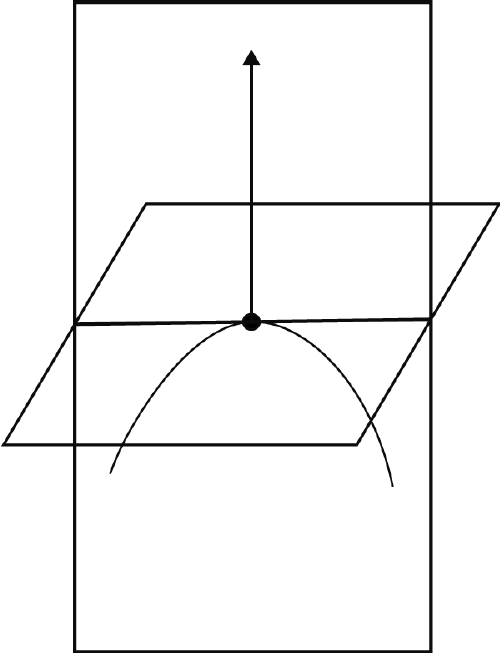}\label{cn6}}
		\caption{Geometrical steps to define the normal curvature $k_{n}$ of a plane field $\Delta$, following \cite{Euler1760}.}
		\label{cneuler}
	\end{figure}
	
	\begin{figure}
		\captionsetup[subfigure]{width=.3\linewidth}
		\centering
		\subfloat{\includegraphics[width=.4\textwidth]{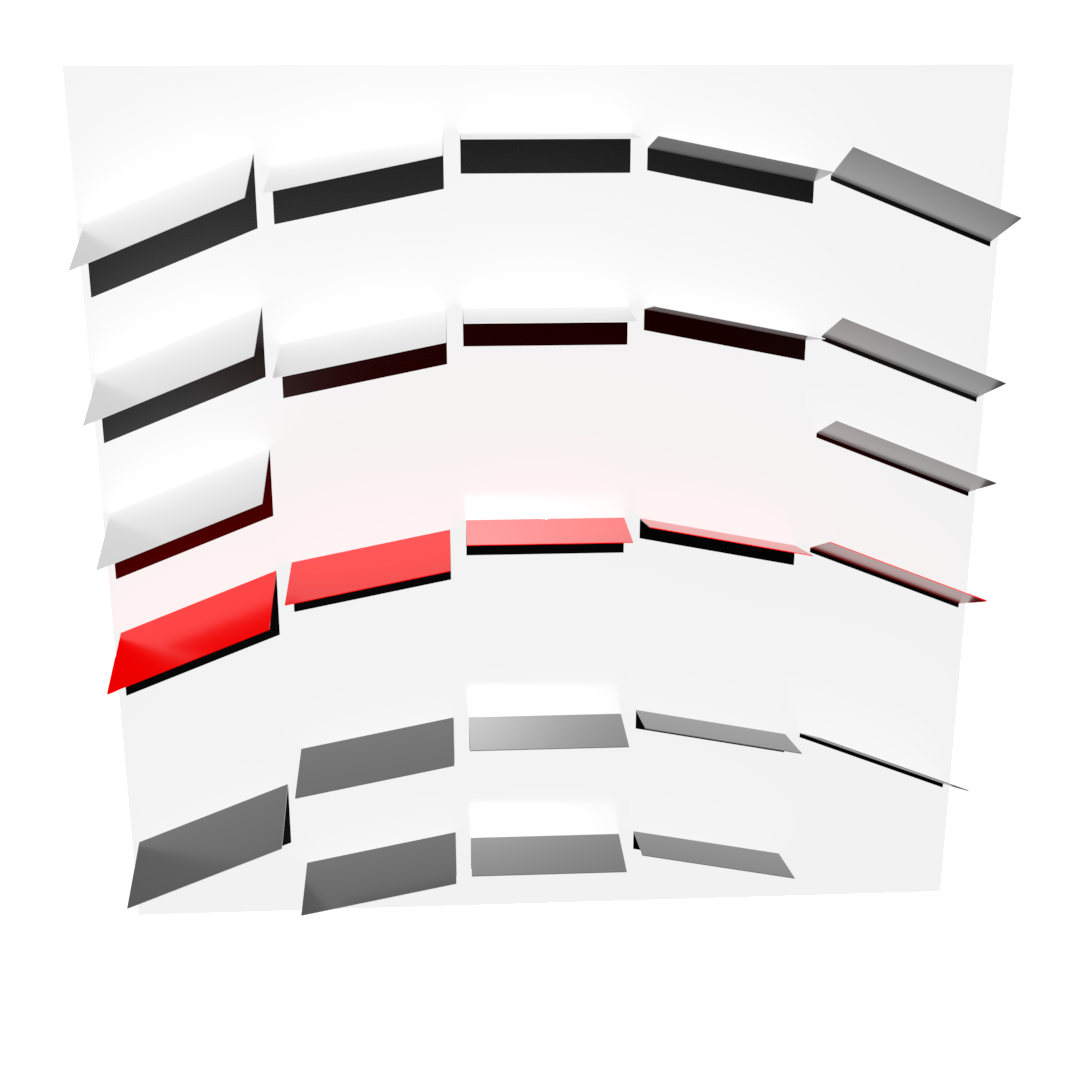}\label{int2}}
		\caption{Intersection of $\mathcal{N}$ with $\Delta$.
		The intersection of $\mathcal{N}$ with the red planes are the tangent
		lines of the curve $\gamma$.}
		\label{int}
	\end{figure}
	\subsection{Darboux frame}\label{darbouxframe}
	Let $\gamma:I\rightarrow \mathbb{R}^3$ be a integral curve of a plane field $\Delta$, parametrized by arc lenght $s$.
	The Darboux frame $\{X(s),Y(s),(X\wedge Y)(s)=\xi(s)\}$ associated to $\gamma$, where $\gamma'(s)=X(s)$, $\langle X(s),Y(s)\rangle=0$,
	$\langle Y(s),Y(s)\rangle=1$, $\langle (X\wedge Y)(s),(X\wedge Y)(s)\rangle=1$,
	is defined by the equations:
	\begin{equation}\label{darboux}
	\begin{split}
	\left.\nabla_{X}X\right|_{\gamma(s)}&=k_{g}(s)Y(s)+k_{n}(s)(X\wedge Y)(s),\\
	\left.\nabla_{X}Y\right|_{\gamma(s)}&=-k_{g}(s)X(s)+\tau_{g}(s)(X\wedge Y)(s),\\
	\left.\nabla_{X}(X\wedge Y)\right|_{\gamma(s)}&=-k_{n}(s)X(s)-\tau_{g}(s)Y(s),
	\end{split}
	\end{equation}
	\noindent where $k_{n}$, $k_{g}$ and $\tau_{g}$ are the normal curvature \eqref{nc}, geodesic curvature \eqref{gc} and geodesic torsion \eqref{gt}
	of the plane field $\Delta$ in the direction $X$:
	\begin{equation}
	\begin{split}
	k_{n}(s)&=\langle \left.\nabla_{X}X\right|_{\gamma(s)},(X\wedge Y)(s) \rangle=-\langle X(s),\left.\nabla_{X}(X\wedge Y)\right|_{\gamma(s)} \rangle,\\
	k_{g}(s)&=((X\wedge Y)(s),X(s),\left.\nabla_{X}X\right|_{\gamma(s)})\\
	&=\langle Y(s), \left.\nabla_{X}X\right|_{\gamma(s)}\rangle=-\langle\left.\nabla_{X}Y\right|_{\gamma(s)},X(s)\rangle,\\
	\tau_{g}(s)&=(X(s),(X\wedge Y)(s),\left.\nabla_{X}(X\wedge Y)\right|_{\gamma(s)})\\
	&=-\langle Y(s),\left.\nabla_{X}(X\wedge Y)\right|_{\gamma(s)}\rangle=\langle \left.\nabla_{X}Y\right|_{\gamma(s)},(X\wedge Y)(s)\rangle.
	\end{split}
	\end{equation}
	Also, at $\gamma(s)$, we have that
	\begin{equation}\label{darboux2}
	\begin{split}
	\left.\nabla_{Y}Y\right|_{\gamma(s)}&=-k_{g,Y}(s)X(s)+k_{n,Y}(s)(X\wedge Y)(s),\\
	\left.\nabla_{Y}X\right|_{\gamma(s)}&=k_{g,Y}(s)Y(s)-\tau_{g,Y}(s)(X\wedge Y)(s),\\
	\left.\nabla_{Y}(X\wedge Y)\right|_{\gamma(s)}&=\tau_{g,Y}(s)X(s)-k_{n,Y}(s) Y(s),
	\end{split}
	\end{equation}
	\noindent where $k_{n,Y}$, $k_{g,Y}$ and $\tau_{g,Y}$ are the normal curvature \eqref{nc}, geodesic curvature \eqref{gc} and geodesic torsion \eqref{gt}
	of the plane field $\Delta$, at $\gamma(s)$, in the direction $Y$.
	\begin{thm}\label{proptau}
		Let $\Delta$ be a plane field. Then
		\begin{equation}\label{tgf}
		\tau_{g}-\widetilde{\tau}_{g}=(k_{2}-k_{1})cos(\theta)sin(\theta),
		\end{equation}
		\noindent where $k_{1}$ and $k_{2}$ are the principal curvatures of $\Delta$, $\theta$ is the angle between
		the direction $dr$ and the principal direction associated to $k_{1}$ and $\widetilde{\tau}_{g}$
		is the geodesic torsion evaluated in the principal direction.
	\end{thm}
	\begin{rem}
		The version of the Theorem \ref{proptau} for surfaces in $\mathbb{R}^3$, see for example \cite[Proposition 2, p.192]{MR532832}, states that $\tau_{g}=(k_{2}-k_{1})cos(\theta)sin(\theta)$. This difference happens because, at surfaces, $\tau_{g,1}=\tau_{g,2}=0$, see for example \cite[p. 196]{MR532832}.
	\end{rem}
	\begin{proof}[Proof of the Theorems \ref{ecf} and \ref{proptau}]
		At a point $(x,y,z)$, let $\gamma$ be a integral curve of $\Delta$, parametrized by arc length $s$, such that $\gamma(0)=(x,y,z)$.
		Consider the Darboux frame $\{X(s),Y(s),(X\wedge Y)(s)=\xi(s)\}$ associated to $\gamma$, where $\gamma'(s)=X(s)$.
		At $\gamma$, $X=cos(\theta)X_{1}+sin(\theta)X_{2}$ and so $Y=-sin(\theta)X_{1}+cos(\theta)X_{2}$, where
		at a point $\gamma(s)$, $X_{1}(s)$ and $X_{2}(s)$ are the two principal directions at it. It follows that
		\begin{equation}\label{eFE}
		\begin{split}
		\left.\nabla_{X}(X\wedge Y)\right|_{\gamma(s)}&=cos(\theta) \left.\nabla_{X_{1}}(X\wedge Y)\right|_{\gamma(s)}+
		sin(\theta) \left.\nabla_{X_{2}}(X\wedge Y)\right|_{\gamma(s)}\\
		&=cos(\theta) \left.\nabla_{X_{1}}(X_{1}\wedge X_{2})\right|_{\gamma(s)}+
		sin(\theta) \left.\nabla_{X_{2}}(X_{1}\wedge X_{2})\right|_{\gamma(s)}.
		\end{split}
		\end{equation}
		By \eqref{darboux}, $\left.\nabla_{X_{1}}(X_{1}\wedge X_{2})\right|_{\gamma(s)}=-k_{1}X_{1}-\tau_{g,1}X_{2}$,
		where $k_{1}$ is the principal direction associated with $X_{1}$ and $\tau_{g,1}$ is the geodesic torsion evaluated in the direction $X_{1}$.
		By \eqref{darboux2}, $\left.\nabla_{X_{2}}(X_{1}\wedge X_{2})\right|_{\gamma(s)}=-k_{2}X_{2}+\tau_{g,2}X_{1}$,
		where $k_{2}$ is the principal direction associated with $X_{2}$ and $\tau_{g,2}$ is the geodesic torsion evaluated in the direction $X_{2}$.
		
		By Proposition \ref{eqlc2}, it follows that
		$\widetilde{\tau}_{g}=\tau_{g,1}=\tau_{g,2}$.
		
		Then $k_{n}(s)=-\langle X(s),\left.\nabla_{X}(X\wedge Y)\right|_{\gamma(s)} \rangle=k_{1}cos^2(\theta)+k_{2}sin^2(\theta)$,
		which prove the Theorem \ref{ecf} and
		$\tau_{g}(s)=-\langle Y(s),\left.\nabla_{X}(X\wedge Y)\right|_{\gamma(s)} \rangle=(k_{2}-k_{1})cos(\theta)sin(\theta)+\widetilde{\tau}_{g}$, which prove the Theorem \ref{proptau}.
	\end{proof}
	\begin{rem}
		The  Euler curvature formula \eqref{ecfor} and the formula \eqref{tgf} are equivalent to the formulas \cite[3 and 4, p. 105]{zbMATH02626886} respectively.
	\end{rem}
	\subsubsection{Triple orthogonal system of plane fields}
	With the notation above, we have that
	\begin{equation}\label{darboux3}
	\begin{split}
	\left.\nabla_{X\wedge Y}(X\wedge Y)\right|_{\gamma(s)}&=l_{1}(s)X(s)+l_{2}(s)Y(s)\\
	\left.\nabla_{X\wedge Y}X\right|_{\gamma(s)}&=l_{3}(s)Y(s)-l_{1}(s)(X\wedge Y)(s),\\
	\left.\nabla_{X\wedge Y}Y\right|_{\gamma(s)}&=-l_{3}(s)X(s)-l_{2}(s)(X\wedge Y)(s).\\
	\end{split}
	\end{equation}
	Let $\Delta_{X}$ and $\Delta_{Y}$ be plane fields such that $\{\Delta,\Delta_{X},\Delta_{Y}\}$ is a triple orthogonal system of plane fields
	and, for all $s$, the vector orthogonal to $\Delta_{X}$ (resp. $\Delta_{Y}$) at $\gamma(s)$ is $X(s)$ (resp. $Y(s)$).
	
	Let $k_{n}^{\Delta_{X}}$, $k_{g}^{\Delta_{X}}$, $\tau_{g}^{\Delta_{X}}$ (resp. $k_{n}^{\Delta_{Y}}$, $k_{g}^{\Delta_{Y}}$, $\tau_{g}^{\Delta_{Y}}$) be respectively
	the normal curvature, geodesic curvature and geodesic torsion of the plane field $\Delta_{X}$ (resp. $\Delta_{Y}$) evaluated at the direction $X\wedge Y$.
	It follows, directly from the definitions \ref{nc}, \ref{gc}, \ref{gt}, that $l_{1}=k_{n}^{\Delta_{X}}=k_{g}^{\Delta_{Y}}$ , $l_{2}=k_{n}^{\Delta_{Y}}=-k_{g}^{\Delta_{X}}$
	and $l_{3}=\tau_{g}^{\Delta_{X}}=\tau_{g}^{\Delta_{Y}}$.
	
	\subsection{The second fundamental form of a plane field}
	Let $\{X,Y\}$ be a local basis for the plane field $\Delta$, that is, locally, $X\wedge Y=\xi$.
	\begin{defn}[{\cite{MR512930}}]
		The second fundamental form $II$ of $\Delta$ is defined by
		\begin{equation}
		II(X,Y)=\frac{1}{2}\langle \nabla_{X}Y+\nabla_{Y}X,\xi \rangle.
		\end{equation}
	\end{defn}
	\begin{rem}[{\cite{MR512930}}]
		In the case that $\Delta$ is completely integrable, $II$ is the second fundamental form of the integral surfaces.
	\end{rem}
	\begin{prop}\label{gauss2}
		Let $X_{i}$, $i=1,2$, be the principal direction associated with $k_{i}$. Then $II(X_{1},X_{2})=0$ and
		\begin{equation}
		\frac{II(X_{1},X_{1})}{\langle X_{1},X_{1}\rangle}=k_{1}, \ \ \frac{II(X_{2},X_{2})}{\langle X_{2},X_{2}\rangle}=k_{2}.
		\end{equation}
	\end{prop}
	\begin{proof}
		The principal curvatures $k_{i}$, $i=1,2$, are  given by
		$k_{i}=\frac{\langle \nabla_{X_{i}}X_{i},\xi\rangle}{\langle X_{i},X_{i}\rangle}$.
		It follows that $\frac{II(X_{i},X_{i})}{\langle X_{i},X_{i}\rangle}=k_{i}$. The geodesic torsion evaluated at the direction $X_{1}$ (resp. $X_{2}$) is given by
		$\tau_{g,1}=\frac{\langle \nabla_{X_{1}}X_{2},\xi\rangle}{\langle X_{1},X_{1}\rangle}$
		\bigg(resp. $\tau_{g,2}=-\frac{\langle \nabla_{X_{2}}X_{1},\xi\rangle}{\langle X_{2},X_{2}\rangle}$\bigg).
		It follows that $\frac{II(X_{1},X_{2})}{\langle X_{1},X_{1}\rangle\langle X_{2},X_{2}\rangle}=\frac{\tau_{g,1}-\tau_{g,2}}{2}$. By Proposition \ref{eqlc2}, $\tau_{g,1}-\tau_{g,2}=0$ and then $II(X_{1},X_{2})=0$.
	\end{proof}
	\begin{rem}
		The \cite[Lemma 2.4, item 3]{MR3008238} asserts that $II(X_{1},X_{2})=0$ if $X_{1}$, $X_{2}$ are the principal directions and can be used to prove
		the Proposition \ref{gauss2}.
	\end{rem}
	\begin{thm}
		The following holds
		\begin{equation}\label{gaussK}
		\begin{split}
		\mathcal{K}_{G}&=\frac{II(X,X)II(Y,Y)-(II(X,Y))^2}{\langle X,X\rangle \langle Y,Y\rangle-\langle X,Y\rangle^2},\\
		\mathcal{H}_{M}&=\frac{II(X,X)\langle Y,Y\rangle-2II(X,Y)\langle X,Y\rangle+II(Y,Y)\langle X,X\rangle}{2(\langle X,X\rangle \langle Y,Y\rangle-\langle X,Y\rangle^2)},
		\end{split}
		\end{equation}
		\noindent where $\mathcal{K}_{G}$ (resp. $\mathcal{H}_{M}$) is the Gauss curvature (resp. Mean curvature) of $\Delta$.
	\end{thm}
	\begin{proof}
		We can write $X$ and $Y$ as
		\begin{equation}
		X=cos(\theta)X_{1}+sin(\theta)X_{2}, \ \ Y=cos(\alpha)X_{1}+sin(\alpha)X_{2},
		\end{equation}
		\noindent where $X_{i}$, $i=1,2$, is the principal direction associated to $k_{i}$.
		It follows that
		\begin{equation}
		\begin{split}
		\langle X,X\rangle&=\langle X_{1},X_{1}\rangle cos^2(\theta)+\langle X_{2},X_{2}\rangle sin^2(\theta), \\
		\langle Y,Y\rangle&=\langle X_{1},X_{1}\rangle cos^2(\alpha)+\langle X_{2},X_{2}\rangle sin^2(\alpha), \\
		\langle X,Y\rangle&=\langle X_{1},X_{1}\rangle cos(\theta)cos(\alpha)+\langle X_{2},X_{2}\rangle sin(\theta)sin(\alpha).
		\end{split}
		\end{equation}
		Then
		\begin{equation}
		\langle X,X\rangle \langle Y,Y\rangle-\langle X,Y\rangle^2=(cos(\theta)sin(\alpha)-cos(\alpha)sin(\theta))^2\langle X_{1},X_{1}\rangle \langle X_{2},X_{2}\rangle,
		\end{equation}
		\noindent and
		\begin{equation}
		\begin{split}
		II(X,X)&=II(X_{1},X_{1})cos^2(\theta)+2II(X_{1},X_{2})cos(\theta)sin(\theta)+II(X_{2},X_{2})sin^2(\theta),\\
		II(Y,Y)&=II(X_{1},X_{1})cos^2(\alpha)+2II(X_{1},X_{2})cos(\alpha)sin(\alpha)+II(X_{2},X_{2})sin^2(\alpha),\\
		II(X,Y)&=II(X_{1},X_{1})cos(\theta)cos(\alpha)+II(X_{1},X_{2})(cos(\theta)sin(\alpha)+cos(\alpha)sin(\theta))
		\\&+II(X_{2},X_{2})sin(\theta)sin(\alpha).
		\end{split}
		\end{equation}
		It follows that
		\begin{equation}
		\frac{II(X,X)II(Y,Y)-(II(X,Y))^2}{\langle X,X\rangle \langle Y,Y\rangle-\langle X,Y\rangle^2}=
		\frac{II(X_{1},X_{1})II(X_{2},X_{2})-(II(X_{1},X_{2}))^2}{\langle X_{1},X_{1}\rangle \langle X_{2},X_{2}\rangle},
		\end{equation}
		\noindent and
		\begin{equation}
		\begin{split}
		&\frac{II(X,X)\langle Y,Y\rangle-2II(X,Y)\langle X,Y\rangle+II(Y,Y)\langle X,X\rangle}{2(\langle X,X\rangle \langle Y,Y\rangle-\langle X,Y\rangle^2)}\\
		&=\frac{II(X_{1},X_{1})\langle X_{2},X_{2}\rangle+II(X_{2},X_{2})\langle X_{1},X_{1}\rangle}{2\langle X_{1},X_{1}\rangle \langle X_{2},X_{2}\rangle}.
		\end{split}
		\end{equation}
		By Proposition \ref{gauss2}, $\frac{II(X_{1},X_{1})}{\langle X_{1},X_{1}\rangle}=k_{1}$ and $\frac{II(X_{2},X_{2})}{\langle X_{2},X_{2}\rangle}=k_{2}$.
	\end{proof}
	\begin{rem}
		The Mean curvature $\mathcal{H}_{M}$ of a plane field as the expression given in \eqref{gaussK} is defined in \cite{MR512930}.
		The expression of $\mathcal{K}_{G}$ of \eqref{gaussK} is defined in \cite[Definition 2.5]{MR2443254}, where is called of sectional curvature of the plane field.
	\end{rem}

    \subsection{Tubular neighbourhood around an integral curve of the plane field}
    
     Let $\gamma$ be a integral curve of the plane field $\Delta$, parametrized by arc lenght $u$.
    
    Set $X(u)=\gamma'(u)$ and consider the Darboux frame defined in Section \ref{darbouxframe} with the equations \eqref{darboux}, \eqref{darboux2} and \eqref{darboux3}.
    
    By Theorems \ref{ecf} and \ref{proptau},
    \begin{equation}\label{eqcur1}
    \begin{split}
    k_{n}(u)&=k_{1}(u)cos^2(\theta(u))+k_{2}(u)sin^2(\theta(u)), \\ \tau_{g}(u)&=\widetilde{\tau}_{g}(u)+(k_{1}(u)-k_{2}(u))cos(\theta(u))sin(\theta(u)).
    \end{split}
    \end{equation}
    Since $\langle X(u),Y(u)\rangle=0$ for all $u$,
    then
    \begin{equation}\label{eqcur2}
    \begin{split}
    k_{n,Y}(u)&=k_{1}(u)sin^2(\theta(u))+k_{2}(u)cos^2(\theta(u)), \\ \tau_{g,Y}(u)&=\widetilde{\tau}_{g}(u)-(k_{1}(u)-k_{2}(u))cos(\theta(u))sin(\theta(u)).
    \end{split}
    \end{equation}
    Let $\alpha:\Lambda\subset \mathbb{R}^3\rightarrow \mathbb{R}^3$ defined by
    \begin{equation}
    \alpha(u,v,w)=\gamma(u)+vY(u)+w(X\wedge Y)(u),
    \end{equation}
    \noindent where $\Lambda$ is a open subset.
    
    \begin{defn}
    The map $\alpha$ is a tubular neighbourhood around the integral curve $\gamma$.
    \end{defn}

\begin{lem}
	In the tubular neighbourhood around an integral curve of the plane field,
	$\mathcal{K}=\mathcal{K}_{G}$ and $\mathcal{H}=\mathcal{H}_{M}$. 
\end{lem}

\begin{proof}
Since $d\alpha=\alpha_{u}du+\alpha_{v}dv+\alpha_{w}dw$, then
\begin{equation}
		d\alpha=[(1-vk_{g}-wk_{n})du]X+[dv-w\tau_{g}du]Y+[v\tau_{g}du+dw]X\wedge Y.
		\end{equation}
		Since $\xi(u,0,0)=(X\wedge Y)(u)$, at the  tubular neighbourhood $\alpha$, the Taylor expansion of $\xi$ is given by
		\begin{equation}\label{xic}
		\begin{split}
		&\xi(u,v,w)=(X\wedge Y)(u,0,0)+v\left.\nabla_{Y}(X\wedge Y)\right|_{\gamma(u)}+w\left.\nabla_{X\wedge Y}(X\wedge Y)\right|_{\gamma(u)}\\
		&+\Bigg(\frac{v^2m_{11}(u)}{2}+vwm_{21}(u)+\frac{w^2m_{31}(u)}{2}+\frac{n_{11}(u)v^3}{6}+\frac{n_{21}(u)v^2w}{2}
		+\frac{n_{31}(u)vw^2}{2}\\
		&+\frac{n_{41}(u)w^3}{6}+\mathcal{O}^4_{u}(v,w)\Bigg)X(u)
		+\Bigg(\frac{v^2m_{12}(u)}{2}+vwm_{22}(u)+\frac{w^2m_{32}(u)}{2}\\
		&+\frac{n_{12}(u)v^3}{6}+\frac{n_{22}(u)v^2w}{2}+\frac{n_{32}(u)vw^2}{2}+\frac{n_{42}(u)w^3}{6}+\mathcal{O}^4_{u}(v,w)\Bigg)Y(u)\\
		&+\Bigg(\frac{v^2m_{13}(u)}{2}+vwm_{23}(u)+\frac{w^2m_{33}(u)}{2}+\frac{n_{13}(u)v^3}{6}+\frac{n_{23}(u)v^2w}{2}
		+\frac{n_{33}(u)vw^2}{2}\\
		&+\frac{n_{43}(u)w^3}{6}+\mathcal{O}^4_{u}(v,w)\Bigg)(X\wedge Y)(u).
		\end{split}
		\end{equation}
		
		We have that $d\xi=\xi_{u}du+\xi_{v}dv+\xi_{w}dw$.
		The implicit differential equations  \eqref{eqla} of the asymptotic lines are given by $\langle\xi,d\alpha\rangle=0$ and $\langle d\xi,d\alpha\rangle=0$.
		We can solve the equation $\langle\xi,d\alpha\rangle=0$ for $dw$ and substitute it in $\langle d\xi,d\alpha\rangle=0$. Then the equations of the asymptotic lines
		becomes $dw=Adu+Bdv$ and $edu^2+2fdudv+gdv^2=0$, where
		
		\begin{equation}
		\begin{split}
		e(u,v,w)&=-k_{n}(u)+\mathcal{O}^1_{u}(v,w),
		\quad f(u,v,w)=-\frac{\tau_{g}(u)-\tau_{g,Y}(u)}{2}
		+\mathcal{O}^1_{u}(v,w), \\
		g(u,v,w)&=-k_{n,Y}(u)+\mathcal{O}^1_{u}(v,w).
		\end{split}
		\end{equation}
		
		It follows from equations \eqref{eqcur1} and \eqref{eqcur2} that
		\begin{equation}
		\mathcal{K}(u,0,0)=e(u,0,0)g(u,0,0)-(f(u,0,0))^2=
		k_{1}(u)k_{2}(u)=\mathcal{K}_{G}(u,0,0)
		\end{equation}
		\noindent and
		\begin{equation}
		\mathcal{H}(u,0,0)=-\frac{(e(u,0,0)+g(u,0,0))}{2}=
		\frac{k_{1}(u)+k_{2}(u)}{2}=\mathcal{H}_{M}(u,0,0).
		\end{equation}
\end{proof}

	\subsection{Lie-Cartan hypersurface and Lie-Cartan vector field}
	\label{subsection:lie}
	Let $\Delta$ be a plane field satisfying the assumptions of the Lemma \ref{propeqlaV}.
	Define $F:\mathbb{R}^4\rightarrow \mathbb{R}$ by
	$F(x,y,z,p)=e+2fp+gp^2$.
	\begin{defn}
		The set defined by the equation $F=0$ is called Lie-Cartan hypersurface, and will be denoted by $\mathbb{L}$.
		The subset of $\mathbb{L}$ defined by the equations $F=0$, $F_{p}=0$ is called criminant surface, and will be denoted by $\widetilde{\mathbb{P}}$.
	\end{defn}
	Let $\pi:\mathbb{R}^4\rightarrow \mathbb{R}^3$ be the map defined by $\pi(x,y,z,p)=(x,y,z)$.
	\begin{prop}\label{propLC1} Let $\Delta$ be a plane field satisfying the assumptions of the Lemma \ref{propeqlaV}. Then the following holds
		\begin{enumerate}
			\item $\pi(\widetilde{\mathbb{P}})\subset \mathbb{P}$.
			\item If $(0,0,0)$ is a parabolic point such that the two asymptotic directions coincides
			then $\pi(0,0,0,0)=(0,0,0)$ and  $\pi^{-1}(0,0,0)=\{(0,0,0,0)\}$.
		\end{enumerate}
	\end{prop}
	\begin{proof}
		Let us prove the first item.
		Suppose that $g(x,y,z)=0$. From $F_{p}=0$ we have that $f(x,y,z)=0$ and then from $F=0$ we have that $e(x,y,z)=0$.
		Then $(x,y,z)\in\mathbb{P}$. Now suppose that $g(x,y,z)\neq0$. From $F_{p}=0$ we will have that $p=-\frac{f(x,y,z)}{g(x,y,z)}$
		and with this $p$, the equation $F=0$ becomes $e(x,y,z)g(x,y,z)-(f(x,y,z))^2=0$ and so $(x,y,z)\in\mathbb{P}$.
		
		Now we will prove the second item. In a neighbourhood of $(0,0,0)$ we can assume that $\xi$ can be written in the form given by Lemma \ref{propetaV}. Suppose that $(0,0,0)$ is a parabolic point such that the two asymptotic directions coincides.
		Then by the Proposition \ref{prop10} we can suppose that $b_{2}\neq0$. Solving $F_{p}(0,0,0,p)=0$ for $p$
		we get $p=0$. It follows that $(0,0,0,0)\in\widetilde{\mathbb{P}}$ since $F(0,0,0,0)=0$.
	\end{proof}
	\begin{prop}\label{prop13}
		Let $\Delta$ be a plane field satisfying the assumptions of the Lemma \ref{propeqlaV}. If $(0,0,0)$ is a parabolic point where the two asymptotic directions coincides, then the Lie-Cartan hypersurface and the criminant surface are both regular in a neighbourhood
		of $(0,0,0,0)$.
	\end{prop}
	\begin{proof}
		In a neighbourhood of $(0,0,0)$ we can assume that $\xi$ can be written in the form given by Lemma \ref{propetaV}.
		After the calculations, we get that $b_{2}F_{x}(0,0,0,0)=(\mathcal{K})_{x}(0,0,0)$,
		$b_{2}F_{y}(0,0,0,0)=(\mathcal{K})_{y}(0,0,0)$ and $b_{2}F_{z}(0,0,0,)=(\mathcal{K})_{z}(0,0,0)$.
		As the parabolic set defined by $\mathcal{K}=0$ is
		a regular surface in a neighbourhood of $(0,0,0)$ then the Lie-Cartan hypersurface defined by $F=0$ is regular in a neighbourhood
		of $(0,0,0,0)$. Since $F_{p}(0,0,0,0)=0$ and $F_{pp}(0,0,0,0)=2b_{2}\neq0$, the gradient vectors $\nabla F(0,0,0,0)$
		and $\nabla F_{p}(0,0,0,0)$ are lineament independents and so the criminant, which is defined by $F=F_{p}=0$, is a regular surface in a neighbourhood of
		$(0,0,0,0)$.
	\end{proof}
	\begin{prop}\label{prop0}
		Let $\Delta$ be a plane field, orthogonal to a vector field $\xi$ of class $C^{k}$, $k\geq3$, satisfying the assumptions of the Lemma \ref{propeqlaV}. Then the equations $F=0$, $dy-pdx=0$, $dz-qdx=0$, $F_{x}dx+F_{y}dy+F_{z}dz+F_{p}dp=0$, where $q=-\frac{a}{c}-\frac{b}{c}p$,
		defines a  line field $\widetilde{\mathcal{A}}$ tangent to $\mathbb{L}$, which in a neighbourhood of $(0,0,0,0)$,
		is spanned by the following vector field of class $C^{k-2}$
		\begin{equation}
		\mathcal{X}=F_{p}\frac{\partial}{\partial x}+pF_{p}\frac{\partial}{\partial y}+qF_{p}\frac{\partial}{\partial z}-
		(F_{x}+pF_{y}+qF_{z})\frac{\partial}{\partial p},
		\end{equation}
		\noindent which can be written in the form $\mathcal{X}=(F_{p},pF_{p},qF_{p},-(F_{x}+pF_{y}+qF_{z}))$.
		
		Furthermore,
		if $\widetilde{\gamma}(t)=(x(t),y(t),z(t),p(t))$ is a integral curve of $\mathcal{X}$, then
		$\gamma(t)=\pi(\widetilde{\gamma}(t))=(x(t),y(t),z(t))$   is an asymptotic line of $\Delta$ and if $\gamma(t)=(x(t),y(t),z(t))$  is an asymptotic line of $\Delta$, then  $\widetilde{\gamma}(t)=(x(t),y(t),z(t),p(t))$
		is a integral curve of $\mathcal{X}$, where $p(t)=\left.\left(\frac{dy}{dt}\right)\right/\left(\frac{dx}{dt}\right)$.
	\end{prop}
	\begin{proof}
		For each $(x,y,z,p)$, the equations $dy-pdx=0$, $dz-qdx=0$, $F_{x}dx+F_{y}dy+F_{z}dz+F_{p}dp=0$ defines a straight line
		with coordinates $(dx,dy,dz,dp)$ and so varying $(x,y,z,p)$ we have a field of straight lines in $\mathbb{R}^3$.
		We will show that this field of straight lines is locally defined by a vector field $\mathcal{X}$.
		
		From the three equations above we have that
		$(F_{x}+pF_{y}+qF_{z})dx+F_{p}dp=0$,
		which the solution $(dx,dp)$ is given by $dp=-(F_{x}+pF_{y}+qF_{z})$ and $dx=F_{p}$. Therefore,  we have that
		$dy=pdx=pF_{p}$ and $dz=qdx=qF_{p}$. This defines locally the vector field $\mathcal{X}=(\dot{x},\dot{y},\dot{z},\dot{p})$
		where
		\begin{equation}
		\dot{x}=F_{p}, \ \ \dot{y}=pF_{p}, \ \ \dot{z}=qF_{p}, \ \ \dot{p}=-(F_{x}+pF_{y}+qF_{z}),
		\end{equation}
		\noindent which can be written in the following notation
		\begin{equation}
		\mathcal{X}=F_{p}\frac{\partial}{\partial x}+pF_{p}\frac{\partial}{\partial y}+qF_{p}\frac{\partial}{\partial z}-
		(F_{x}+pF_{y}+qF_{z})\frac{\partial}{\partial p}.
		\end{equation}
		Now, let $\widetilde{\gamma}$ be a integral curve of $\mathcal{X}$, $\widetilde{\gamma}(t)=(x(t),y(t),z(t),p(t)).$
		From $F(\widetilde{\gamma})=0$ we have that $e(x(t),y(t),z(t))+2f(x(t),y(t),z(t))p(t)+g(x(t),y(t),z(t))(p(t))^2=0$.
		But $\frac{dy}{dt}=p(t)\frac{dx}{dt}$ and so $\left(\frac{dy}{dt}\right)^2=(p(t))^2\left(\frac{dx}{dt}\right)^2.$
		Then multiplying the equation $F(\widetilde{\gamma})=0$ by $\left(\frac{dx}{dt}\right)^2$ we get the equation
		\begin{equation}
		e(\gamma(t))\left(\frac{dx}{dt}\right)^2+2f(\gamma(t))\left(\frac{dx}{dt}\right)\left(\frac{dy}{dt}\right)
		+g(\gamma(t))\left(\frac{dy}{dt}\right)^2=0.
		\end{equation}
		Then $\gamma$ satisfies the equation $edx^2+2fdxdy+gdy^2=0$ of \eqref{eqlaV}. We have that
		$\frac{dz}{dt}-q(t)\frac{dx}{dt}=0$ and
		$q(t)=-\left(\frac{a(\gamma(t))}{c(\gamma(t))}\right)-\left(\frac{b(\gamma(t))}{c(\gamma(t))}\right)p(t)$ and
		so $$\frac{dz}{dt}=-\left(\frac{a(\gamma(t))}{c(\gamma(t))}\right)\left(\frac{dx}{dt}\right)-\left(\frac{b(\gamma(t))}{c(\gamma(t))}\right)\left(\frac{dy}{dt}\right).$$
		With that, $\gamma$ satisfies the equation $dz=-\left(\frac{a}{c}\right)dx-\left(\frac{b}{c}\right)dy$ of \eqref{eqlaV}.
		This concludes that $\gamma$  is an asymptotic line of the plane field $\Delta$.
		
		Now, if $\gamma$  is an asymptotic line of the plane field, then $\gamma$ and $\gamma'=\left(\frac{dx}{dt},\frac{dy}{dt},\frac{dz}{dt}\right)$ satisfies the equations \eqref{eqlaV}:
		\begin{equation}\label{eqlaVgamma}
		\begin{split}
		&e(\gamma)\left(\frac{dx}{dt}\right)^2+2f(\gamma)\left(\frac{dx}{dt}\right)\left(\frac{dy}{dt}\right)
		+g(\gamma)\left(\frac{dy}{dt}\right)^2=0,
		\\
		&\frac{dz}{dt}=-\left(\frac{a(\gamma)}{c(\gamma)}\right)\left(\frac{dx}{dt}\right)-\left(\frac{b(\gamma)}{c(\gamma)}\right)\left(\frac{dy}{dt}\right).
		\end{split}
		\end{equation}
		As $dy-pdx=0$, $dz-qdx=0$ then $p(t)=\frac{dy}{dx}=\left.\left(\frac{dy}{dt}\right)\right/ \left(\frac{dx}{dt}\right)$,
		$q(t)=\frac{dz}{dx}=\left.\left(\frac{dz}{dt}\right)\right/ \left(\frac{dx}{dt}\right)$.
		Dividing the first equation
		(respectively the second equation)
		of \eqref{eqlaVgamma} by $\left(\frac{dx}{dt}\right)^2$ (respectively by $\frac{dx}{dt}$) we get that
		$$
		e(\gamma(t))+2f(\gamma(t))p(t)+g(\gamma(t))(p(t))^2=0, \ \ q(t)=-\frac{a(\gamma(t))}{c(\gamma(t))}-\left(\frac{b(\gamma(t))}{c(\gamma(t))}\right)p(t).
		$$
		Then $F(\widetilde{\gamma})=0$. Differentiating the equation  $F(\widetilde{\gamma})=0$ with relation of $t$
		we get that
		$$
		F_{x}(\widetilde{\gamma})\frac{dx}{dt}+F_{y}(\widetilde{\gamma})\frac{dy}{dt}+F_{z}(\widetilde{\gamma})\frac{dz}{dt}
		+F_{p}(\widetilde{\gamma})\frac{dp}{dt}=0.
		$$
		It follows that the tangents of $\widetilde{\gamma}$ belongs to the field of straight lines which is locally defined
		by $\mathcal{X}$ and so  $\widetilde{\gamma}$ is an integral curve of $\mathcal{X}$.
	\end{proof}
	\begin{defn}
		The vector field $\mathcal{X}$ given in Proposition \ref{prop0} is called Lie-Cartan vector field, see \cite{MR2532372}.
		It is called of suspended vector field, see \cite{MR1634428} and lifted field \cite{MR1328597}.
	\end{defn}
	\begin{rem}
		To prove a more general version of the second item of the Proposition \ref{propLC1}, we need to consider $G:\mathbb{R}^4\rightarrow \mathbb{R}$ defined by
		$G(x,y,z,p)=eq^2+2fq+g$, where $dz-qdx=0$.
	\end{rem}
	
	\begin{prop}\label{propsing}
		Let $\Delta$ be a plane field satisfying the assumptions of the Lemma \ref{propeqlaV}. Suppose that $(0,0,0)$ is a parabolic point
		such that the two asymptotic directions coincides. Let $\theta$ be the angle between $\nabla \mathcal{K}(0,0,0)$ and
		the asymptotic direction $\mathcal{A}$ on $(0,0,0)$.
		Then $(0,0,0,0)\in\widetilde{\mathbb{P}}$
		is a singular point of the Lie-Cartan vector field $\mathcal{X}$
		if, and only if, $\theta=\frac{\pi}{2}$, which implies that $a_{11}=0$.
	\end{prop}
	\begin{proof}
		By Lemma \ref{propetaV}, $\xi=(a,b,c)$, where $a$, $b$, $c$ are given by \ref{etaV}. By Proposition \ref{prop10}, $a_{2}=-b_{1}$, $b_{2}\neq0$
		and the asymptotic direction at $(0,0,0)$ is $\mathcal{A}=\left(1,0,0\right)$.
		
		From $(\mathcal{K})_{x}(0,0,0)=2a_{11}b_{2}$ it follows that
		\begin{equation}\label{ang}
		\langle\nabla\mathcal{K}(0,0,0),\mathcal{A}\rangle=
		|\nabla\mathcal{K}(0,0,0)|cos(\theta)=2a_{11}b_{2}.
		\end{equation}
		After calculations, $\mathcal{X}(0,0,0,0)=\left(0,0,0,-\frac{|\nabla\mathcal{K}(0,0,0)|cos(\theta)}{b_{2}}\right)$.
		Then $(0,0,0,0)$ is a singular point of the Lie-Cartan vector field if, and only if, $\theta=\frac{\pi}{2}$.
		If $\theta=\frac{\pi}{2}$, $a_{11}=0$ follows from \eqref{ang}.
	\end{proof}

	\begin{prop}\label{prop2}
		Let $\Delta$ be a plane field satisfying the assumptions of the Lemma \ref{propeqlaV} and let $\widetilde{\varphi}$ be a curve of singular points of $\mathcal{X}$
		passing by $(0,0,0,0)$. Then in a neighbourhood of $(0,0,0,0)$, the jacobian matrix $D\mathcal{X}(\widetilde{\varphi})$ of $\mathcal{X}$
		evaluated at $\widetilde{\varphi}$ is given by
		\begin{equation}
		D\mathcal{X}(\widetilde{\varphi})=\left(
		\begin{array}{cccc}
		F_{px} & F_{py} & F_{pz} & F_{pp} \\
		pF_{px} & pF_{py} & pF_{pz} & pF_{pp} \\
		qF_{px} & qF_{py} & qF_{pz} & qF_{pp} \\
		A & B & C & D \\
		\end{array}
		\right)
		\end{equation}
		\noindent where $A=-(F_{xx}+pF_{xy}+q_{x}F_{z}+qF_{xz})$, $B=-(F_{xy}+pF_{yy}+q_{y}F_{z}+qF_{yz})$, $C=-(F_{xz}+pF_{yz}+q_{z}F_{z}+qF_{zz})$
		and $D=-(F_{xp}+pF_{yp}+F_{y}+q_{p}F_{z}+qF_{pz})$.
		
		Furthermore, the not necessarily zero eigenvalues $\lambda_{1}$ and $\lambda_{2}$ of $D\mathcal{X}$ are given by
		\begin{equation}\label{eigenVa}
		\begin{split}
		\lambda_{1}=-\frac{F_{y}+q_{p}F_{z}}{2}+\frac{\sqrt{\Omega}}{2}, \quad \lambda_{2}=-\frac{F_{y}+q_{p}F_{z}}{2}-\frac{\sqrt{\Omega}}{2}
		\end{split}
		\end{equation}
		\noindent where
		\begin{equation}\label{omega}
		\begin{split}
		\Omega&=F_{y}^2+4(pF_{py}+qF_{pz})F_{y}+4F_{px}^2+4(F_{y}+2pF_{py}+2qF_{pz})F_{px}\\
		&+4[p^2F_{py}^2+2pqF_{py}F_{pz}+q^2F_{pz}^2+(A_{x}+pB_{y}+qC_{z})F_{pp}].
		\end{split}
		\end{equation}
		The eigenvectors $\vartheta_{1},\vartheta_{2}$, associated to the eigenvalues $\lambda_{1},\lambda_{2}$ respectively,  are given by
		\begin{equation}\label{eigenVe}
		\begin{split}
		\vartheta_{1}&=\Bigg(1,p,q,-\frac{F_{y}+2(F_{px}+pF_{py}+qF_{pz})}{2F_{pp}}+\frac{\sqrt{\Omega}}{2F_{pp}}\Bigg)\\
		&=\Bigg(1,p,q,\frac{\lambda_{1}}{F_{pp}}-\frac{F_{px}+pF_{py}+qF_{pz}}{F_{pp}}\Bigg),\\
		\vartheta_{2}&=\Bigg(1,p,q,-\frac{F_{y}+2(F_{px}+pF_{py}+qF_{pz})}{2F_{pp}}-\frac{\sqrt{\Omega}}{2F_{pp}}\Bigg)\\
		&=\Bigg(1,p,q,\frac{\lambda_{2}}{F_{pp}}-\frac{F_{px}+pF_{py}+qF_{pz}}{F_{pp}}\Bigg).
		\end{split}
		\end{equation}
	\end{prop}
	\begin{proof}
		Let $\mathcal{X}_{4}=-(F_{x}+pF_{y}+qF_{z})$. Then the jacobian matrix $D\mathcal{X}$ is given by
		\begin{equation}
		D\mathcal{X}=\left(
		\begin{array}{cccc}
		F_{px} & F_{py} & F_{pz} & F_{pp} \\
		pF_{px} & pF_{py} & pF_{pz} & F_{p}+pF_{pp} \\
		q_{x}F_{p}+qF_{px} & q_{y}F_{p}+qF_{py} & q_{z}F_{p}+qF_{pz} & q_{p}F_{p}+qF_{pp} \\
		(\mathcal{X}_{4})_{x} & (\mathcal{X}_{4})_{y} & (\mathcal{X}_{4})_{z} & (\mathcal{X}_{4})_{p} \\
		\end{array}
		\right)
		\end{equation}
		
		But in $\widetilde{\mathbb{P}}$ we have that $F=0$, $F_{p}=0$ and so
		
		\begin{equation}
		D\mathcal{X}(\widetilde{\varphi})=\left(
		\begin{array}{cccc}
		F_{px} & F_{py} & F_{pz} & F_{pp} \\
		pF_{px} & pF_{py} & pF_{pz} & pF_{pp} \\
		qF_{px} & qF_{py} & qF_{pz} & qF_{pp} \\
		(\mathcal{X}_{4})_{x} & (\mathcal{X}_{4})_{y} & (\mathcal{X}_{4})_{z} & (\mathcal{X}_{4})_{p} \\
		\end{array}
		\right).
		\end{equation}
		The matrix $D\mathcal{X}(\widetilde{\varphi})$ has two zero eigenvalues and two not necessarily zero eigenvalues
		$\lambda_{1}$ and $\lambda_{2}$.
		Performing the calculations, we find the expressions for $\lambda_{1}$ and $\lambda_{2}$ given by \eqref{eigenVa} and
		the expressions for the associated eigenvectors given by \eqref{eigenVe}.
	\end{proof}
	
	\begin{prop}\label{prop12}
		Let $\Delta$ be a plane field satisfying the assumptions of the Lemma \ref{propeqlaV}. Suppose that $(0,0,0)$ is a parabolic point
		such that the two asymptotic directions coincides and suppose that $(0,0,0,0)$ is a singular point of the Lie-Cartan vector field with real not necessarily zero eigenvalues $\lambda_{1}$ and $\lambda_{2}$.
		Then, at $(0,0,0,0)$, the eigenvector $\vartheta_{i}$ associated  to the eigenvalue $\lambda_{i}$  is tangent to the criminant if, and only if, $\lambda_{i}=0$.
	\end{prop}
	\begin{proof}
		The tangent plane, in coordinates $(dx,dy,dz,dp)$, of the criminant surface  at $(0,0,0,0)$ is given by
		\begin{equation}\label{tpc}
		F_{x}dx+F_{y}dy+F_{z}dz=0, \ \ F_{px}dx+F_{py}dy+F_{pz}dz+F_{pp}dp=0,
		\end{equation}
		\noindent where $F_{i}$ and $F_{jk}$ are evaluated at $(0,0,0,0)$.
		By Proposition \ref{prop2}, at $(0,0,0,0)$, the eigenvector $\vartheta_{i}$ associated to the eigenvalue $\lambda_{i}$ is given by
		\begin{equation}
		\vartheta_{i}=\left(1,0,0,\frac{\lambda_{i}-F_{px}}{F_{pp}}\right).
		\end{equation}
		Set $\ell(\Phi)=(1,0,0,\Phi)$. We will prove that $\ell(\Phi)$ is a solution of \eqref{tpc} only if $\Phi=-\frac{F_{px}}{F_{pp}}$.
		After replacing $(dx,dy,dz,dp)=\ell(\Phi)$ at \eqref{tpc} we get $F_{px}+F_{pp}\Phi=0$ and $F_{x}=0$, which is already satisfied since $(0,0,0,0)$ is a singular point of
		$\mathcal{X}$. Then we get $\Phi=-\frac{F_{px}}{F_{pp}}=-\frac{a_{12}+2b_{11}-a_{3}b_{1}}{2b_{2}}$ is the only solution of \eqref{tpc} evaluated at $\ell(\Phi)$.
		Since $\vartheta_{i}=\ell\left(\frac{\lambda_{i}-F_{px}}{F_{pp}}\right)$, it follows that $\vartheta_{i}$ is tangent to the criminant if, and only if, $\lambda_{i}=0$.
	\end{proof}
	
	\subsection{Parabolic surface and curve of special parabolic points}
	
	\begin{defn}
		A point $r$ of the parabolic surface is called special parabolic point if
		$\langle\nabla\mathcal{K}(r),\mathcal{A}(r)\rangle=0$ at $r$.
	\end{defn}
	
	\begin{lem}\label{lemma:s1}
		Let $\Delta$ be a plane field satisfying the assumptions of the Lemma \ref{propeqlaV}.  Let $\varphi$ be 
		a regular curve of parabolic points. Then $\varphi$
		is a curve of special parabolic points if, and only if,  $\widetilde{\varphi}$ is a curve of singular points of the Lie-Cartan vector field $\mathcal{X}$.	
	\end{lem}
	
	\begin{proof}
		From $edx^2+2fdxdy+gdy^2=0$ and $F=e+2fp+gp^2=0$ it follows that, at the parabolic surface, 
		$p=\frac{dy}{dx}=-\frac{f}{g}$. Then
		\begin{equation}
		q=\frac{dz}{dx}=-\frac{a}{c}+\frac{bf}{bg}.
		\end{equation}
		Then, at the parabolic surface, the asymptotic direction at a point $r$ is given by
		\begin{equation}
		\mathcal{A}=\left(1,-\frac{f}{g},-\frac{a}{c}+\frac{bf}{bg}\right).
		\end{equation}
		After a straightfoward calculation, it follows that
		\begin{equation}
		cg^2\langle\nabla\mathcal{K},\mathcal{A}\rangle-cg^3(F_{x}+pF_{y}+qF_{z})=
		[(bg_{z}-cg_{y})f+(cg_{x}-ag_{z})g](eg-f^2).
		\end{equation}
		The result follows since $eg-f^2=0$ at the parabolic surface,d $F_{p}=0$ if, and only if, $eg-f^2=0$
		and $\mathcal{X}=(F_{p},pF_{p},qF_{p},-(F_{x}+pF_{y}+qF_{z}))$.
	\end{proof}
	
	\begin{lem}\label{lemma:spct}
		Let $\Delta$ be a plane field satisfying the assumptions of the Lemma \ref{propeqlaV}.  Let $\varphi$ be 
		a regular curve of special parabolic points passing though $(0,0,0)$.
		Then $\varphi'(0,0,0)$ belongs to the plane field if, and only if, 
		$\lambda_{1}+\lambda_{2}=0$, $\lambda_{1}\lambda_{2}=0$ or $\Upsilon=0$
		where $\Upsilon$ is given by \eqref{upsilon}. 
		
		Furthermore, if 	$\lambda_{1}+\lambda_{2}=0$ then $\varphi'(0,0,0)$
		is an asymptotic direction and if  $\lambda_{1}\lambda_{2}=0$  then
		the  asymptotic direction at $(0,0,0)$ and $\varphi'(0,0,0)$ generate the tangent plane of the parabolic surface.

	\end{lem}
	
	\begin{proof}
		The curve $\varphi$ of special parabolic points is given by 
		$\mathcal{K}(x,y,z)=0$ and $\langle\nabla\mathcal{K}(x,y,z),\mathcal{A}(x,y,z)\rangle=0$,  where
		\begin{equation}
		\mathcal{A}(x,y,z)=\left(1,-\frac{f(x,y,z)}{g(x,y,z)},
		-\frac{a(x,y,z)}{b(x,y,z)}+\frac{b(x,y,z)f(x,y,z)}{c(x,y,z)g(x,y,z)}\right).
		\end{equation}
		By Lemma \ref{propeqlaV} and Proposition \ref{propsing} or Lemmma \ref{lemma:s1}
		\begin{equation}
		\nabla\mathcal{K}(0,0,0)=(0,(a_{3}b_{1}+a_{12})b_{2},[a_{13}-(a_{3})^2]b_{2}).
		\end{equation}
		Since the parabolic set is a regular surface and $b_{2}\neq0$ then 
		$a_{3}b_{1}+a_{12}\neq0$ or $a_{13}-(a_{3})^2\neq0$.

		\begin{itemize}
			\item Suppose that $a_{13}-(a_{3})^2\neq0$.
		\end{itemize}
		Then, in a neighbourhood of $(0,0,0)$, the parabolic surface is given by
		\begin{equation}
		z(x,y)=-\frac{(a_{3}b_{1}+a_{12})y}{a_{13}-(a_{3})^2}+\mathcal{O}^2(x,y).
		\end{equation}
		It follows that the equation $\langle\nabla\mathcal{K},\mathcal{A}\rangle=0$
		becomes
		\begin{equation}
		[6 a_{111}b_{2}-(a_{12}+b_{11})(a_{12}-a_{3}b_{1}+2b_{11})]x+\frac{\Upsilon y}{a_{13}-(a_{3})^2}+\mathcal{O}^2(x,y)=0,
		\end{equation}
		\noindent where 
			\begin{equation}\label{upsilon}
		\begin{split}
		\Upsilon&=(a_{12})^2 a_{23}-2a_{12} a_{13}a_{22}+2a_{12}a_{22}(a_{3})^2 
		+ a_{12} a_{23} a_{3} b_{1} - 
		2 a_{13} a_{22} b_{11} + a_{12} a_{23} b_{11} \\
		&
		+ 2 a_{22} (a_{3})^2 b_{11}
		+ a_{23} a_{3} b_{1} b_{11} - 
		a_{12} a_{13} b_{12} + a_{12} (a_{3})^2 b_{12} - a_{13} b_{11} b_{12} 
		+ (a_{3})^2 b_{11} b_{12} \\
		&
		+ (a_{12})^2 b_{13} +
		a_{12} a_{3} b_{1} b_{13} + a_{12} b_{11} b_{13} + a_{3} b_{1} b_{11} b_{13}
		- 2 a_{113} a_{12} b_{2} + 
		2 a_{112} a_{13} b_{2} \\
		&
		+ 3 a_{12} a_{13} a_{3} b_{2} 
		- 2 a_{112} (a_{3})^2 b_{2} - a_{12} (a_{3})^3 b_{2} + 
		2 (a_{13})^2 b_{1} b_{2} - 2 a_{113} a_{3} b_{1} b_{2} \\
		&
		+ a_{13} a_{3} b_{11} b_{2} - (a_{3})^3 b_{11} b_{2} - 
		2 (a_{12})^2 a_{3} b_{3} - a_{12} a_{13} b_{1} b_{3} 
		- a_{12} (a_{3})^2 b_{1} b_{3} \\
		&
		- 2 a_{12} a_{3} b_{11} b_{3} - 
		a_{13} b_{1} b_{11} b_{3} - (a_{3})^2 b_{1} b_{11} b_{3} 
		+ 2 a_{12} a_{3} b_{2} c_{11} + 2 a_{13} b_{1} b_{2} c_{11}.
		\end{split}
		\end{equation}
		Direct calculations shows that
		\begin{equation}
		\lambda_{1}+\lambda_{2}=a_{3}b_{1}+a_{12}, \quad
		\lambda_{1}\lambda_{2}=6 a_{111}b_{2}-(a_{12}+b_{11})(a_{12}-a_{3}b_{1}+2b_{11})
		\end{equation}
		
		Since $\varphi$ is a regular curve, then $\lambda_{1}\lambda_{2}\neq 0$ or $\Upsilon\neq0$.
		
		If $\Upsilon\neq0$ then
		\begin{equation}
		\varphi(x)=\left(x,-\frac{[a_{13}-(a_{3})^2]\lambda_{1}\lambda_{2}x}{\Upsilon}
		+\mathcal{O}^2(x),
		\frac{(\lambda_{1}+\lambda_{2})
			\lambda_{1}\lambda_{2}x}{\Upsilon}+\mathcal{O}^2(x)\right)
		\end{equation}
		Since $\xi(0,0,0)=(0,0,1)$ then
		\begin{equation}
		\langle\xi(0,0,0),\varphi'(0)\rangle=\frac{(\lambda_{1}+\lambda_{2})
			\lambda_{1}\lambda_{2}}{\Upsilon}.
		\end{equation}
		
		If $\lambda_{1}\lambda_{2}\neq0$ then
		
		\begin{equation}
		\varphi(y)=\left(-\frac{\Upsilon y}{[a_{13}-(a_{3})^2]\lambda_{1}\lambda_{2}}+\mathcal{O}^2(y)+\mathcal{O}^2(y),y,
		-\frac{(\lambda_{1}+\lambda_{2})y}{a_{13}-(a_{3})^2}+\mathcal{O}^2(y)\right)
		\end{equation}
		
		\noindent and 
		\begin{equation}
		\langle\xi(0,0,0),\varphi'(0)\rangle=-\frac{\lambda_{1}+\lambda_{2}}{a_{13}-(a_{3})^2}.
		\end{equation}
		
		\begin{itemize}
			\item Suppose that $a_{3}b_{1}+a_{12}\neq0$, i.e, 
			$\lambda_{1}+\lambda_{2}\neq0$.
		\end{itemize}
		Then, in a neighbourhood of $(0,0,0)$, the parabolic surface is given by
		\begin{equation}
		y(x,z)=-\frac{(a_{13}-(a_{3})^2)z}{a_{3}b_{1}+a_{12}}+\mathcal{O}^2(x,z)
		\end{equation}
		\noindent and $\langle\nabla\mathcal{K},\mathcal{A}\rangle=0$
		becomes
		\begin{equation}
		[6 a_{111}b_{2}-(a_{12}+b_{11})(a_{12}-a_{3}b_{1}+2b_{11})]x+\frac{\Upsilon z}{a_{3}b_{1}+a_{12}}+\mathcal{O}^2(x,z)=0,
		\end{equation}
		\noindent that is, $\lambda_{1}\lambda_{2}x+\frac{\Upsilon z}{\lambda_{1}+\lambda_{2}}+\mathcal{O}^2(x,z)=0$.
		
		If $\Upsilon\neq0$ then 
		
		\begin{equation}
		\varphi(x)=\left(x,\frac{[a_{13}-(a_{3})^2]\lambda_{1}\lambda_{2}x}{\Upsilon}
		+\mathcal{O}^2(x)	,-\frac{(\lambda_{1}+\lambda_{2})\lambda_{1}\lambda_{2}x}{\Upsilon}+\mathcal{O}^2(x)\right)
		\end{equation}
		\noindent and
		\begin{equation}
		\langle\xi(0,0,0),\varphi'(0)\rangle=\frac{(\lambda_{1}+\lambda_{2})\lambda_{1}\lambda_{2}}{\Upsilon}.
		\end{equation}
		
		If $\lambda_{1}\lambda_{2}\neq0$ then
		
		\begin{equation}
		\varphi(z)=\left(-\frac{\Upsilon z}{(\lambda_{1}+\lambda_{2})\lambda_{1}\lambda_{2}}+\mathcal{O}^2(z),
		\frac{(a_{13}-(a_{3})^2)z}{\lambda_{1}+\lambda_{2}}+\mathcal{O}^2(z),z\right)
		\end{equation}
		\noindent and
		$\langle\xi(0,0,0),\varphi'(0)\rangle=1$.

	\end{proof}
	
	\section{Asymptotic lines near the parabolic surface}
	\label{par}
	In this section, assume that $\Delta$ is a plane field satisfying the assumptions of the Lemma \ref{propeqlaV} and assume that the parabolic set is a regular surface
	such that $d\mathcal{K}\neq0$ at it, where $d\mathcal{K}=k_{1}dk_{2}+k_{2}dk_{1}$. This implies that $\mathcal{H}$ does not vanishes at the parabolic surface.
	\subsection{Cuspidal parabolic point}
	\begin{prop}\label{propcusp1}
		Suppose that $(0,0,0)$ is a parabolic point where the two asymptotic directions coincides at it. Let $\theta$ be the angle between $\nabla\mathcal{K}(0,0,0)$ and the asymptotic direction at $(0,0,0)$. If
		$\theta\neq\frac{\pi}{2}$,
		then $a_{11}\neq0$ and the parabolic surface in the neighbourhood of $(0,0,0)$ is given by
		\begin{equation}\label{cuspPsL}
		x(y,z)=-\left(\frac{a_{3}b_{1}+a_{12}}{2a_{11}}\right)y-\left(\frac{a_{13}-(a_{3})^2}{2a_{11}}\right)z+\mathcal{O}^{2}(y,z).
		\end{equation}
	\end{prop}
	\begin{proof}
		By Proposition \ref{prop10}, $b_{2}\neq0$ and $a_{2}=-b_{1}$ and $\mathcal{A}=(1,0,0)$. It follows that $\langle\nabla\mathcal{K}(0,0,0),\mathcal{A}\rangle=\mathcal{K}_{x}(0,0,0)=2b_{2}a_{11}\neq0$. By the Implicit Function Theorem, the parabolic surface
		in a neighbourhood of $(0,0,0)$ is given by \ref{cuspPsL}.
	\end{proof}
	
	\begin{thm}\label{t1}
		Suppose that $(0,0,0)$ is a parabolic point where the two asymptotic directions coincides at it.
		Denote by $\Gamma_{1}$ and $\Gamma_{2}$ the asymptotic lines in $\mathbb{H}\cup\mathbb{P}$ such that
		$\Gamma_{1}(0,0,0)=\Gamma_{2}(0,0,0)=(0,0,0)$.
		Let $\theta$ be the angle between $\nabla\mathcal{K}(0,0,0)$ and the asymptotic direction at $(0,0,0)$. If
		$\theta\neq\frac{\pi}{2}$, then the image of the asymptotic lines $\Gamma_{1}$ and $\Gamma_{2}$ is locally
		parameterized by $\gamma:(-\varepsilon,\varepsilon)\rightarrow\mathbb{R}^3$,
		\begin{equation*}
		\begin{split}
		\gamma(t)&=\Bigg(-\left(\frac{3^{\frac{2}{3}}\mathcal{I}}{2^{\frac{1}{3}}(a_{11})^{\frac{2}{3}}}\right)t^2
		+\mathcal{O}^3(t),t^{3}+\mathcal{O}^4(t),\left(\frac{3^{\frac{2}{3}}\mathcal{I}\mathcal{R}}{2^{\frac{4}{3}}5(a_{11})^{\frac{2}{3}}}\right)t^5+\mathcal{O}^6(t)\Bigg),
		\end{split}
		\end{equation*}
		\noindent where $\mathcal{I}=|\nabla\mathcal{K}(0,0,0)||\mathcal{A}|cos(\theta)$ and $\mathcal{R}=\langle curl(\xi)(0,0,0),\xi(0,0,0)\rangle$.
		The curve $\gamma$ is locally of the type of the curve $(t^2,t^3,t^5)$
		and has a cuspidal singularity of the type $(t^2,t^3,0)$ in $(0,0,0)$. Furthermore, the asymptotic lines in a neighbourhood of $(0,0,0)$ is as show in the Figures \ref{cusp} and \ref{cusp2}.
	\end{thm}
	\begin{figure}[H]
		\captionsetup[subfigure]{width=1\linewidth}
		\centering
		\subfloat{\includegraphics[width=.7\textwidth]{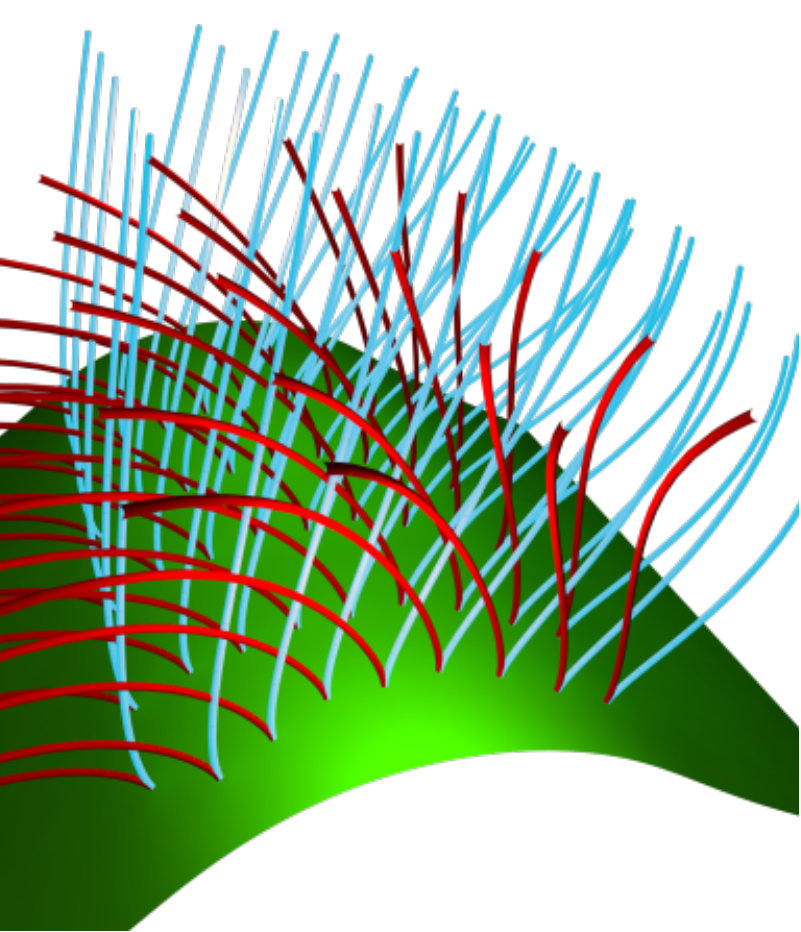}\label{p1}}
		\caption{The two foliations of asymptotic lines near the parabolic surface (green surface) with cuspidal parabolic points. Here, one foliation is coloured with blue and the other one with red. Let $\Gamma_{1}$ (resp. $\Gamma_{2}$) be a blue (resp. red) asymptotic line such that $\Gamma_{1}(0)=\Gamma_{2}(0)\in\mathbb{P}$. Then, near $P$, the the image of the asymptotic lines $\Gamma_{1}$ and $\Gamma_{2}$ can be parametrized by a curve with a cuspidal point in $P$ of the type $(t^2,t^3,t^5)$.}
		\label{cusp}
	\end{figure}
	\begin{figure}[H]
		\centering
		\captionsetup[subfigure]{width=.5\linewidth}
		\subfloat[][Projection of the Figure \ref{cusp} on the plane $(x,y)$.]{\includegraphics[width=.5\textwidth]{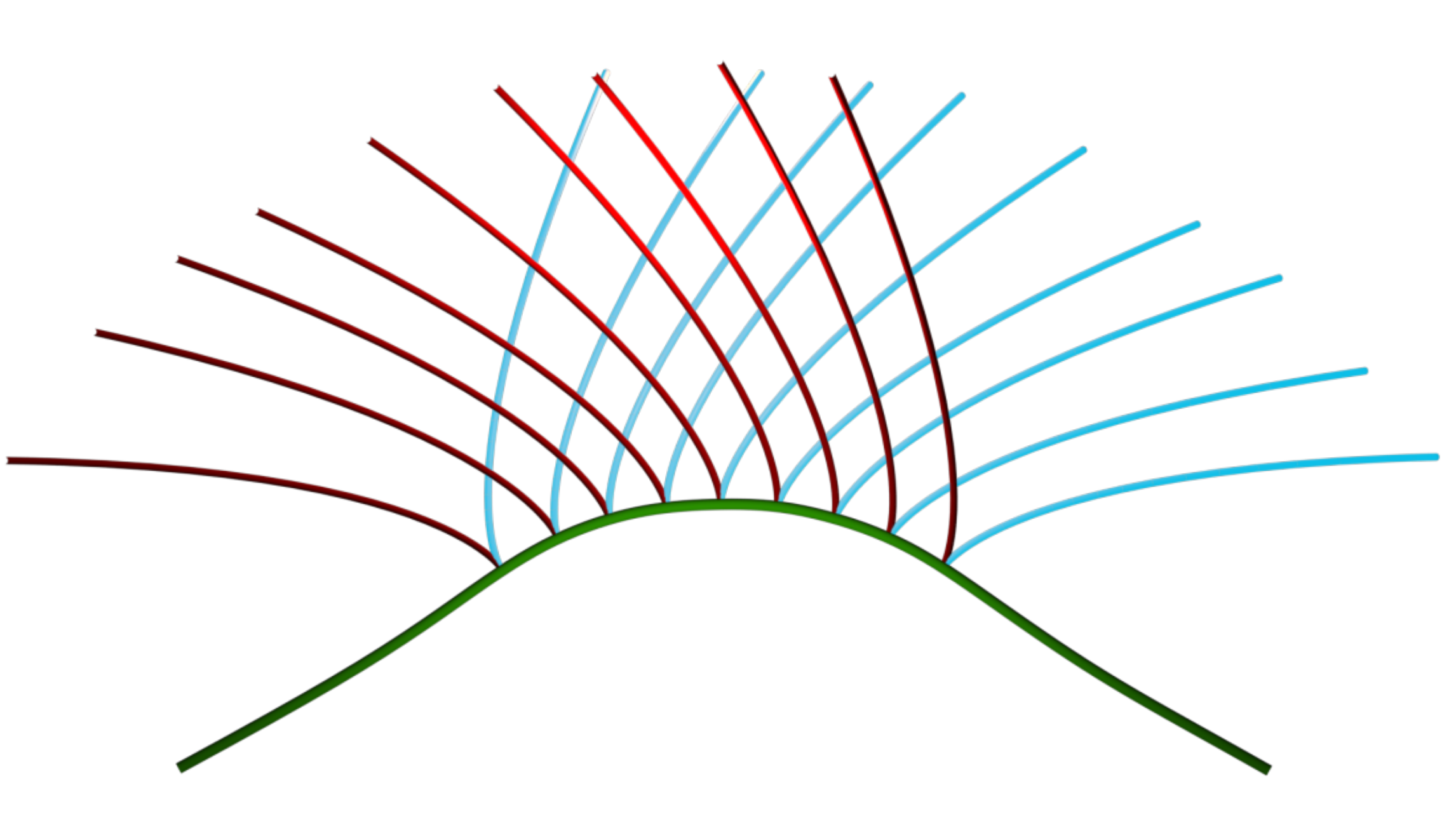}\label{p2}}
		\captionsetup[subfigure]{width=.3\linewidth}
		\qquad
		\subfloat[][Curve $c(t)=(t^2,t^3,t^5)$.]{\includegraphics[width=.3\textwidth]{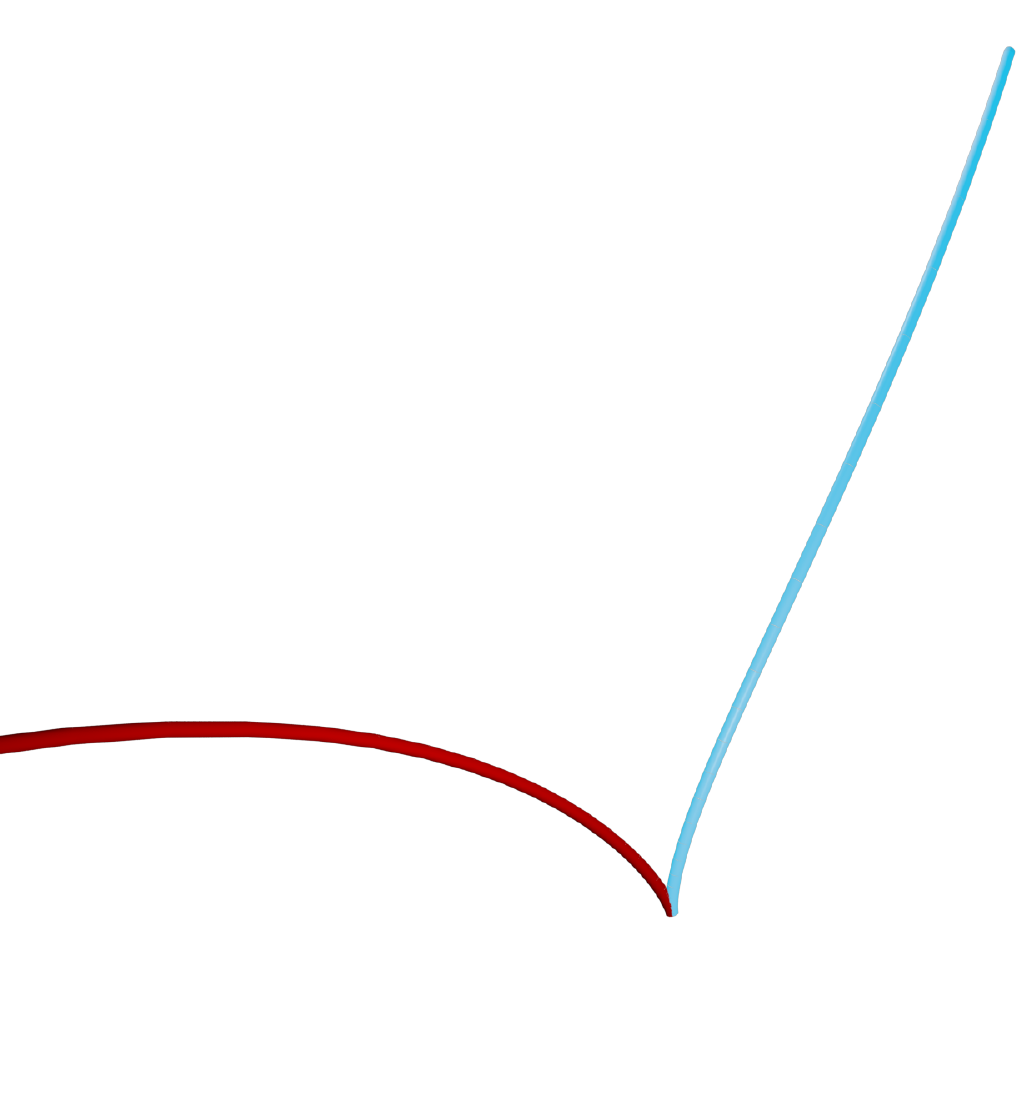}\label{p3}}
		\caption{Cuspidal parabolic point.}
		\label{cusp2}
	\end{figure}
	\begin{proof}
		By Propositions \ref{prop10} and \ref{propcusp1}, $b_{2}\neq0$, $a_{2}=-b_{1}$ and $a_{11}\neq0$.
		We first observe that $\mathcal{F}(t)=0+\mathcal{O}^5(t)$ and $\mathcal{G}(t)=0+\mathcal{O}^5(t)$, where $\mathcal{F}(t)=\langle\xi(\gamma(t)),\gamma'(t)\rangle$ and $\mathcal{G}(t)=\langle \xi(\gamma(t)),\gamma''(t)\rangle$, which
		are nothing but \eqref{eqla} evaluated at $\gamma(t)$ in the direction $\gamma'(t)$. It follows that the image of the asymptotic lines $\Gamma_{1}$ and $\Gamma_{2}$ is locally
		parameterized by $\gamma$. Note that $\gamma_{1}(s)=\gamma(\sqrt{s})$ parametrises one asymptotic line, say $\Gamma_{1}$, and $\gamma_{2}(s)=\gamma(-\sqrt{s})$
		parametrises the another asymptotic line $\Gamma_{2}$, and $\gamma_{1}'(0)=\gamma_{2}'(0)=-\left(\frac{3^{\frac{2}{3}}\mathcal{I}}{2^{\frac{1}{3}}(a_{11})^{\frac{2}{3}}}\right)\mathcal{A}$, where $\mathcal{A}$ is the asymptotic direction
		at $(0,0,0)$ given in \ref{prop10}, with $dx=1$.
		Applying the change of coordinates $\varsigma(x,y,z)=\left(-\left(\frac{2^{\frac{1}{3}}(a_{11})^{\frac{2}{3}}}{3^{\frac{2}{3}}\mathcal{I}}\right)x,y,
		\left(\frac{2^{\frac{4}{3}}5(a_{11})^{\frac{2}{3}}}{3^{\frac{2}{3}}\mathcal{I}\mathcal{R}}\right)z\right)$ in $\gamma$, we get
		$\varsigma\circ\gamma(t)=\left(t^2+\mathcal{O}^{3}(t),t^3+\mathcal{O}^4(t),t^5+\mathcal{O}^6(t)\right)$. From \cite[Theorem 2.1]{MR1198413} it follows that
		$\gamma$ has a cuspidal singularity of the type $(t^2,t^3,0)$ in $(0,0,0)$.
	\end{proof}

	\subsection{Parabolic point of saddle type, node type and focus type}
	\begin{defn}
		A singular point $(x,y,z,p)$ of the
		Lie-Cartan vector field is called of saddle type (resp. node type) if the eigenvalues $\lambda_{1}$ and $\lambda_{2}$ of $D\mathcal{X}(x,y,z,p)$, given in Proposition \ref{prop2},  are real and $\lambda_{1}\lambda_{2}<0$ (resp. $\lambda_{1}\lambda_{2}>0$) and is called of focus type if the eigenvalues $\lambda_{1}$ and $\lambda_{2}$ are complex and $\lambda_{1}\lambda_{2}\neq0$. If $(x,y,z,p)$ is a singular point of saddle type (resp. node type and focus type), then $\pi(x,y,z,p)=(x,y,z)$
		is called of parabolic point of saddle type (resp. node type and focus type).
	\end{defn}
	\begin{thm}\label{t2}
		Suppose that $(0,0,0)$ is a parabolic point where the two asymptotic directions coincides at it
		and suppose
		there exists a curve $\varphi$ of parabolic points passing by $(0,0,0)$ such that the two asymptotic directions coincides at each of then and
		where his lift $\widetilde{\varphi}$ to the Lie-Cartan hypersurface is a curve of singular points of the
		Lie-Cartan vector field $\mathcal{X}$ such that the singular points are entirely of the saddle, node or
		focus type and in a neighbourhood of $\widetilde{\varphi}$, the only singular points are in $\widetilde{\varphi}$.
		Then the integral curves of $\mathcal{X}$ near $(0,0,0,0)$ is as show in Figures \ref{snf1}, \ref{snf2}, \ref{snf3} and the asymptotic lines near $(0,0,0)$ is as show in Figures \ref{snf4}, \ref{snf7}, \ref{snf2}, \ref{snf8}, \ref{snf6}, \ref{snf9}.
	\end{thm}
	
		\begin{figure}[H]
		\centering
		\captionsetup[subfigure]{width=.7\linewidth}
		\subfloat[][Projection by $\pi$ of the integral curves of Fig. \ref{snf1} onto the asymptotic lines (coloured as red and blue), the criminant surface of Fig. \ref{snf1} onto the parabolic surface (coloured as green) and the curve of singular points of 
		Fig. \ref{snf1} onto the curve of parabolic point of the type saddle (coloured as yellow).]{\includegraphics[width=.7\textwidth]{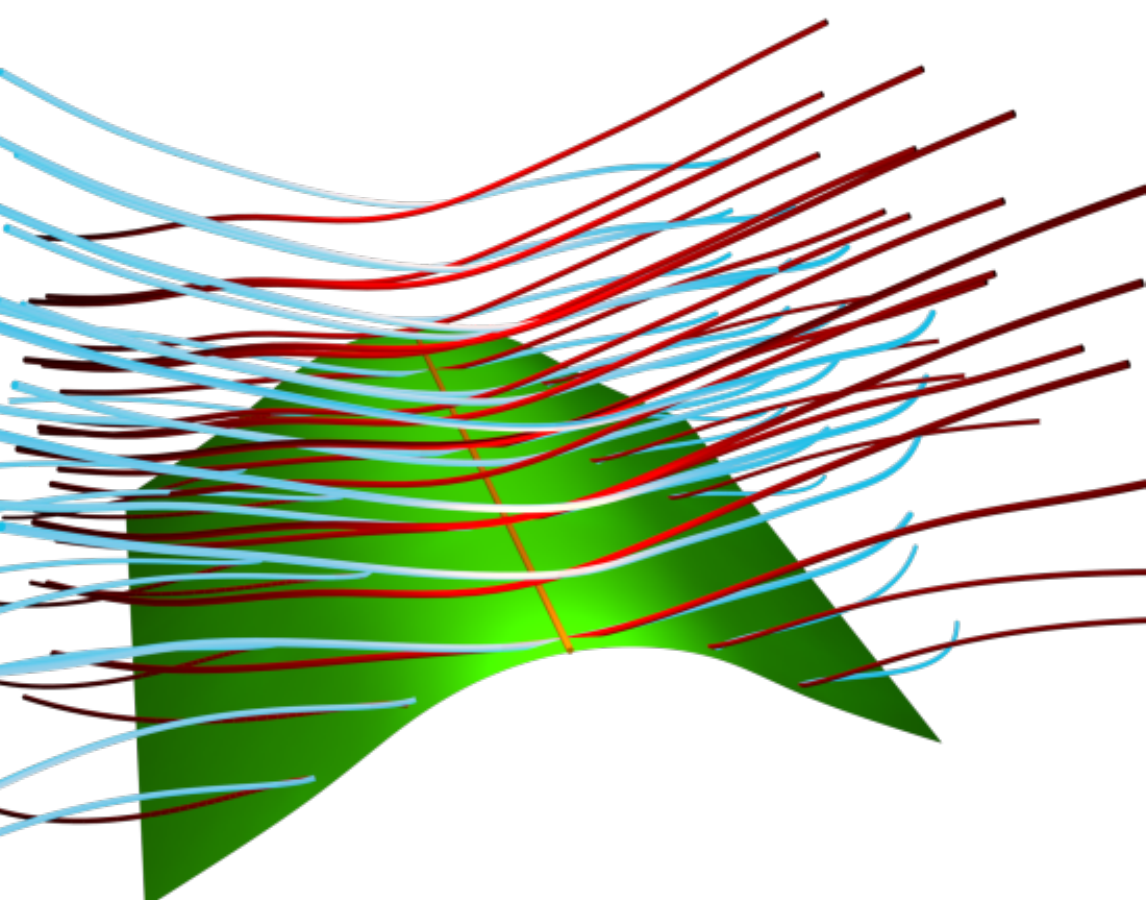}\label{snf4}}
		\\
		\captionsetup[subfigure]{width=.5\linewidth}
		\subfloat[][Integral curves (coloured as pink) of the Lie Cartan vector field with a curve (coloured as green) of singular points of saddle type. The criminant surface is coloured as blue.]{\includegraphics[width=.5\textwidth]{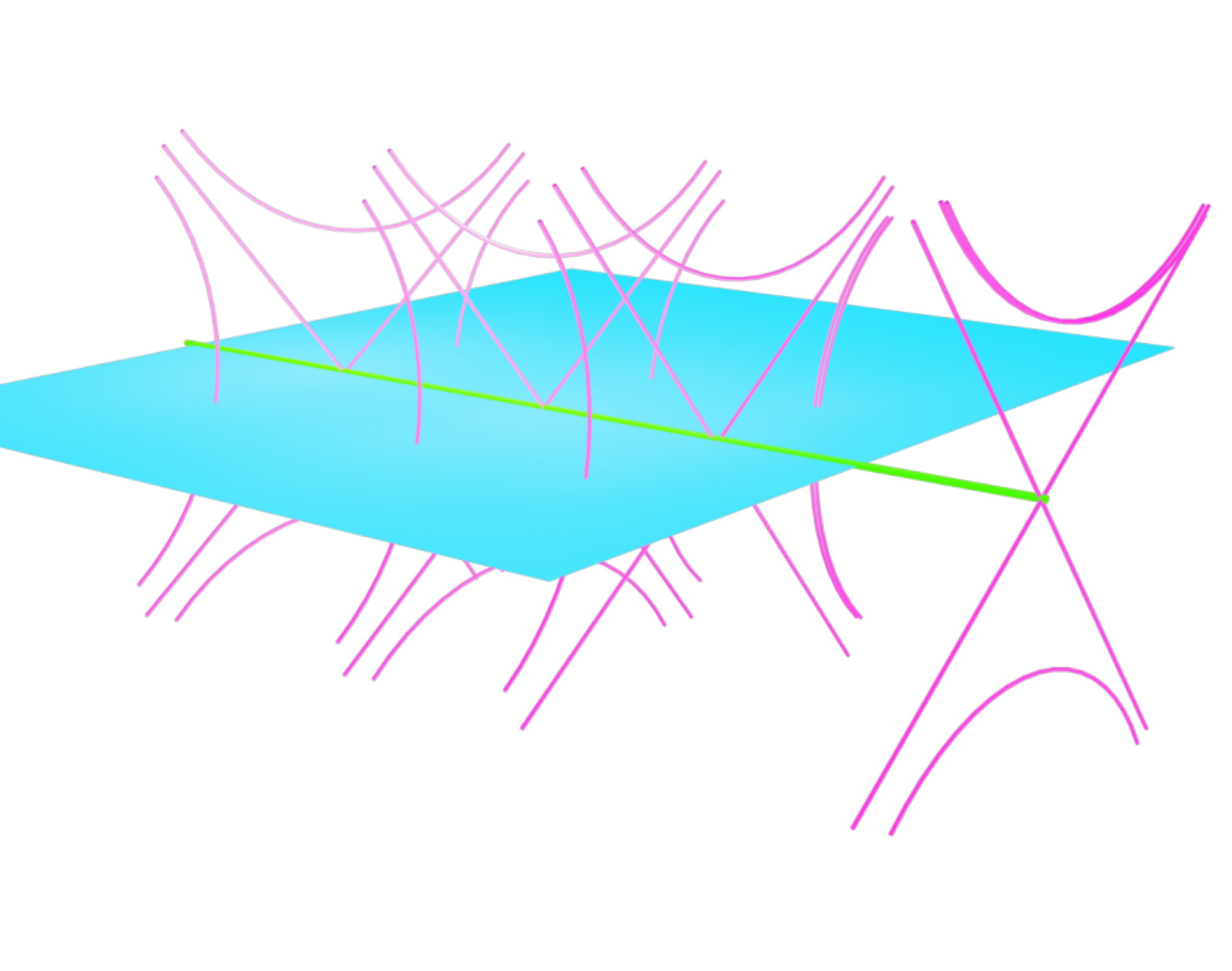}\label{snf1}}
		\qquad
		\captionsetup[subfigure]{width=.3\linewidth}
		\subfloat[][Frontal view of \ref{snf4}.]{\includegraphics[width=.3\textwidth]{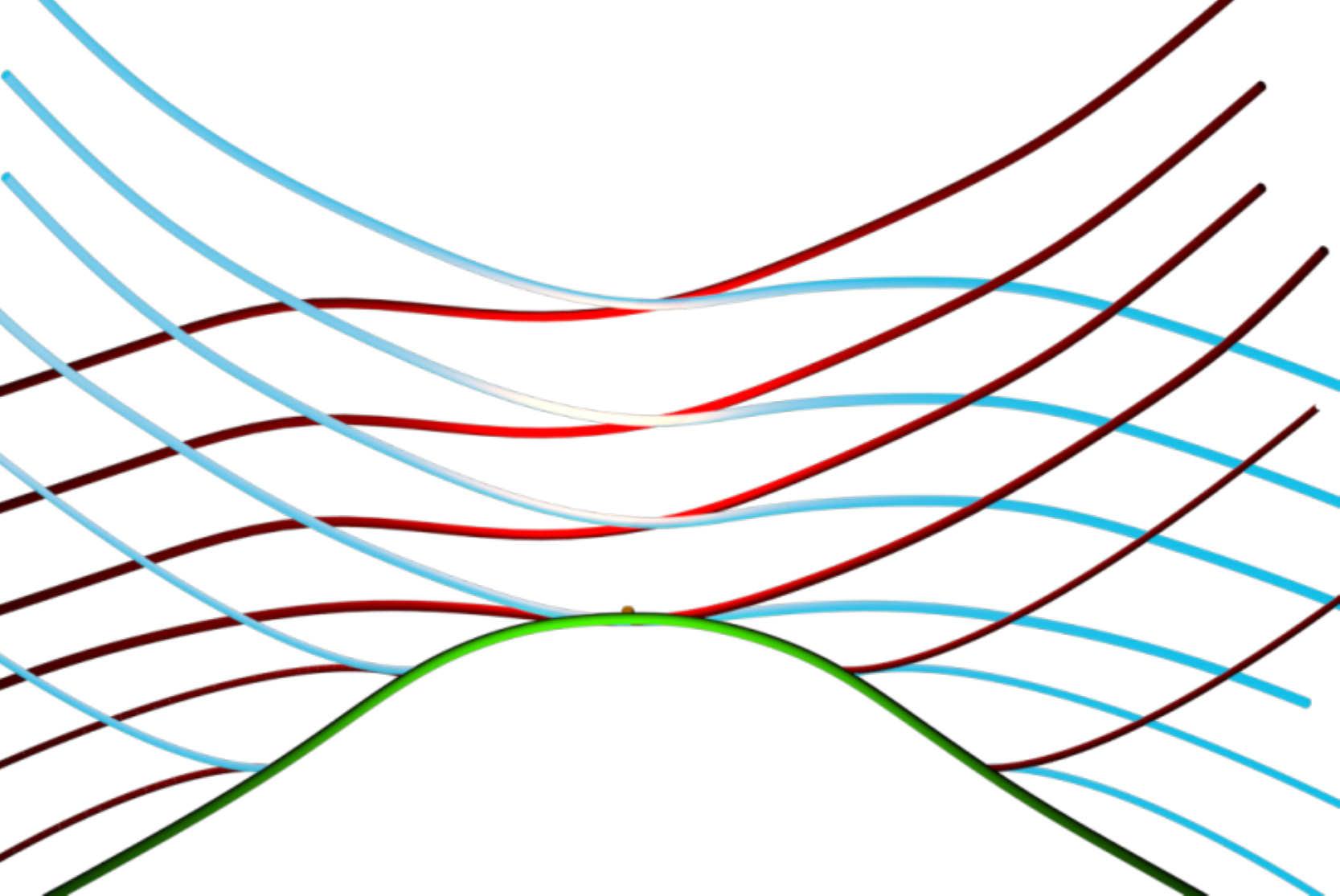}\label{snf7}}
		\caption{Parabolic point of saddle type.}
		\label{snfp1}
	\end{figure}
	\begin{figure}[H]
		\centering
		\captionsetup[subfigure]{width=.7\linewidth}
		\subfloat[][Projection by $\pi$ of the integral curves of Fig. \ref{snf2} onto the asymptotic lines (coloured as red and blue), the criminant surface of Fig. \ref{snf2} onto the parabolic surface (coloured as green) and the curve of singular points of Fig. \ref{snf2} onto the curve of parabolic point of the type node (coloured as yellow).]{\includegraphics[width=.7\textwidth]{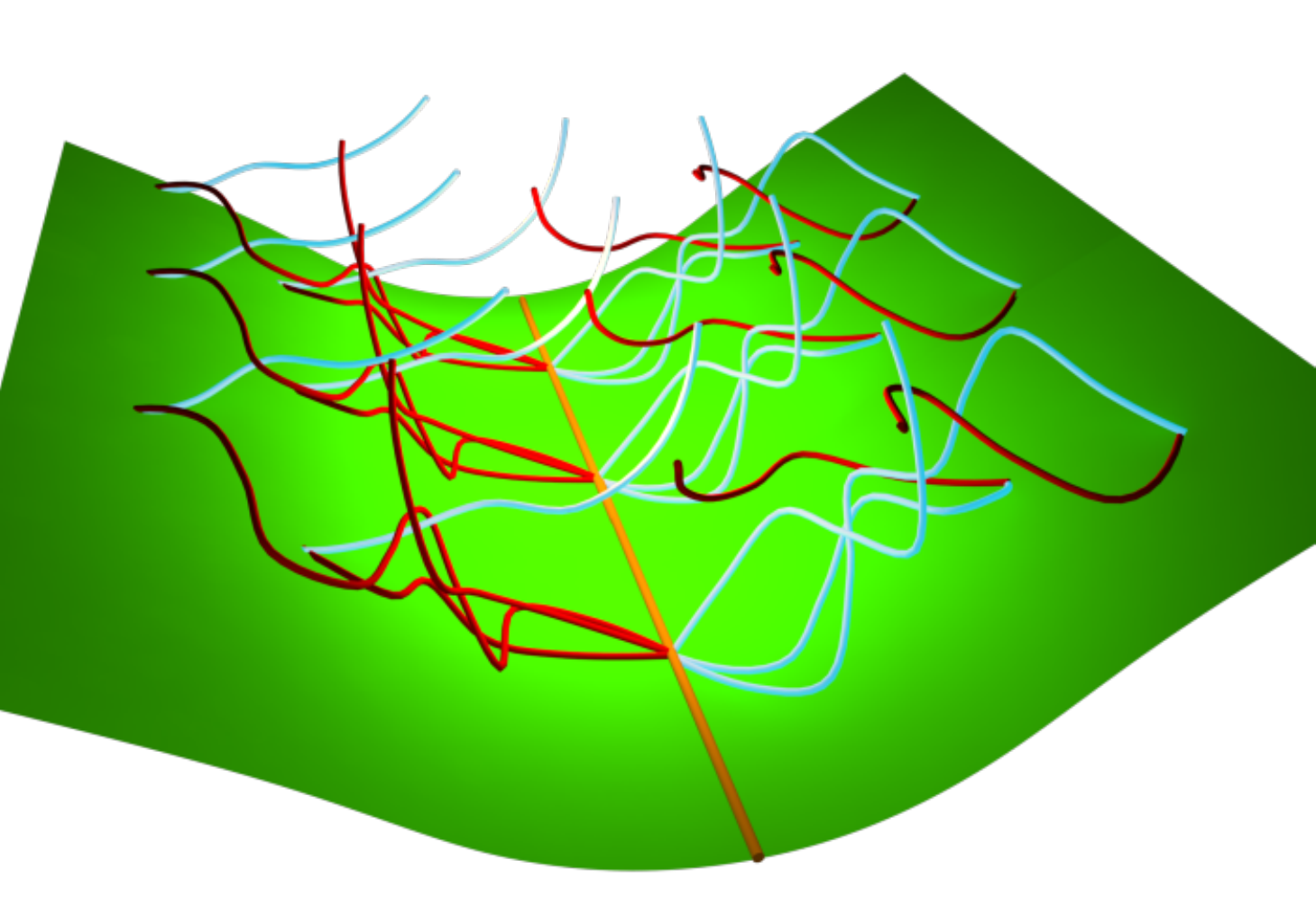}\label{snf5}}
		\\
		\captionsetup[subfigure]{width=.5\linewidth}
		\subfloat[][Integral curves (coloured as pink) of the Lie Cartan vector field with a curve (coloured as green) of singular points of node type. The criminant surface is coloured as blue.]{\includegraphics[width=.4\textwidth]{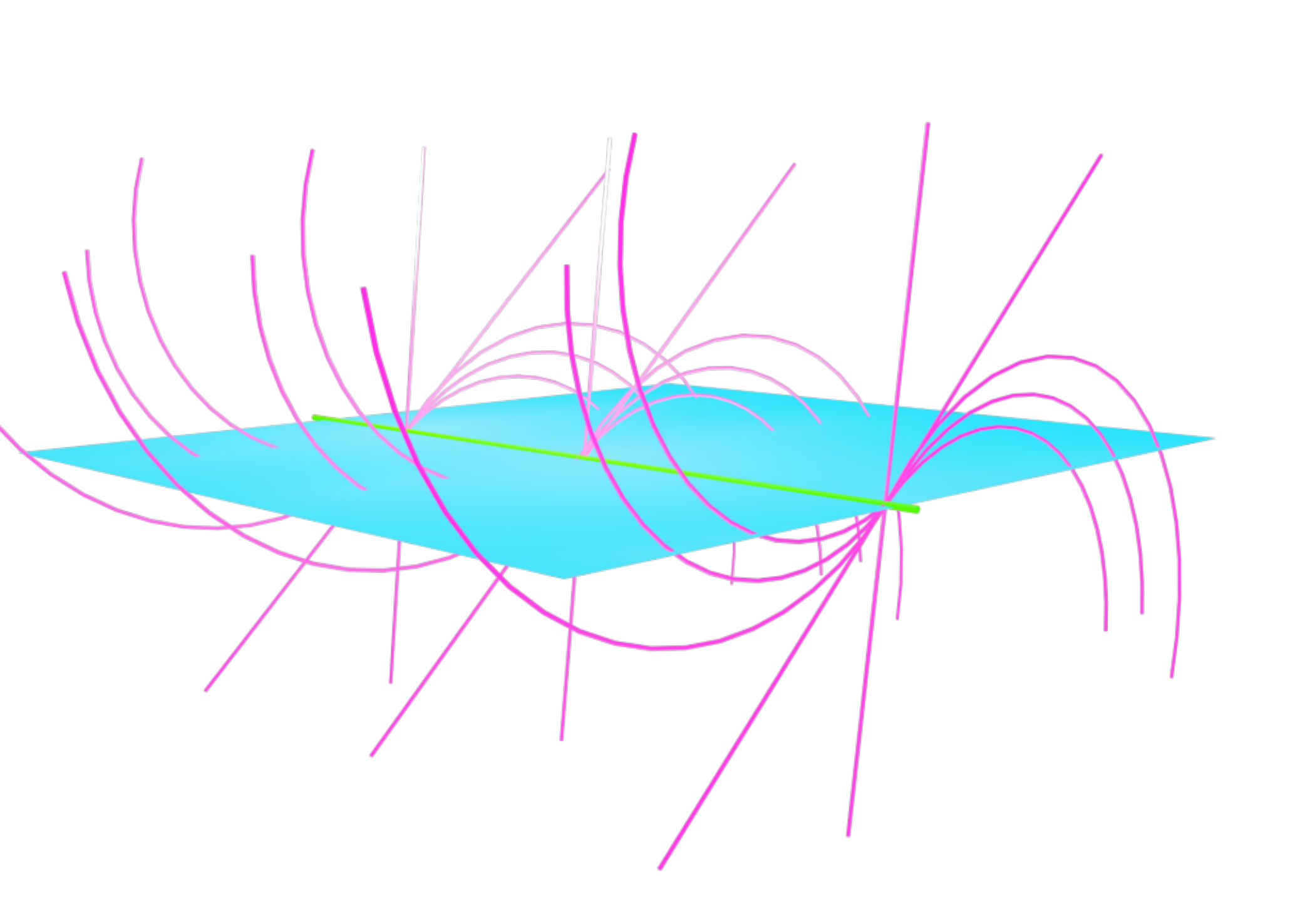}\label{snf2}}
		\captionsetup[subfigure]{width=.3\linewidth}
		\subfloat[][Frontal view of \ref{snf5}.]{\includegraphics[width=.4\textwidth]{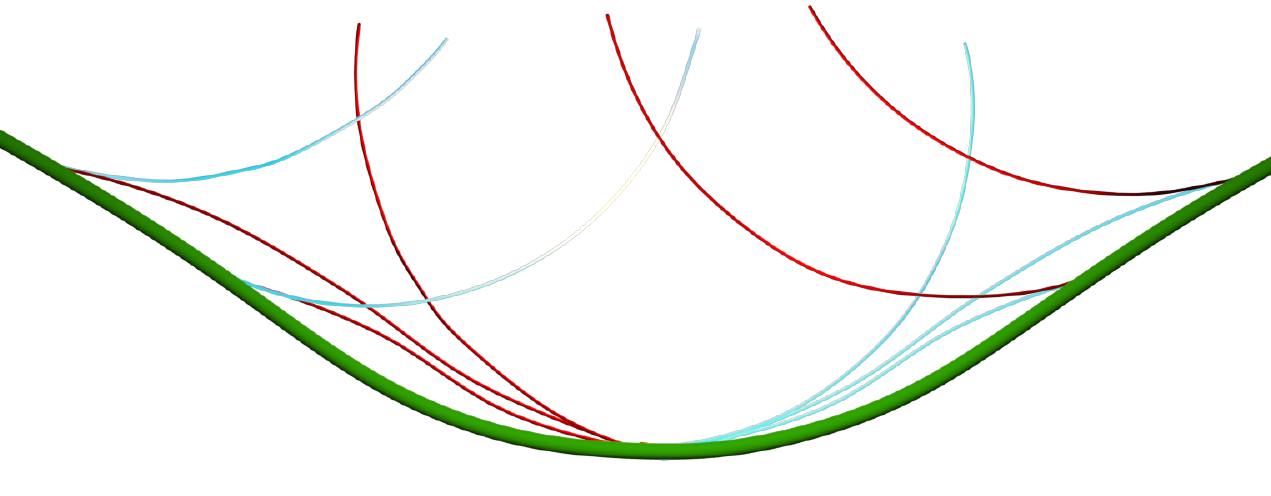}\label{snf8}}
		\caption{Parabolic point of node type.}
		\label{snfp2}
	\end{figure}

	\begin{figure}[H]
		\centering
		\captionsetup[subfigure]{width=.7\linewidth}
		\subfloat[][Projection by $\pi$ of the integral curves of Fig. \ref{snf3} onto the asymptotic lines (coloured as red and blue), the criminant surface of Fig. \ref{snf3} onto the parabolic surface (coloured as green) and the curve of singular points of Fig. \ref{snf3} onto the curve of parabolic point of the type focus (coloured as yellow).]{\includegraphics[width=.7\textwidth]{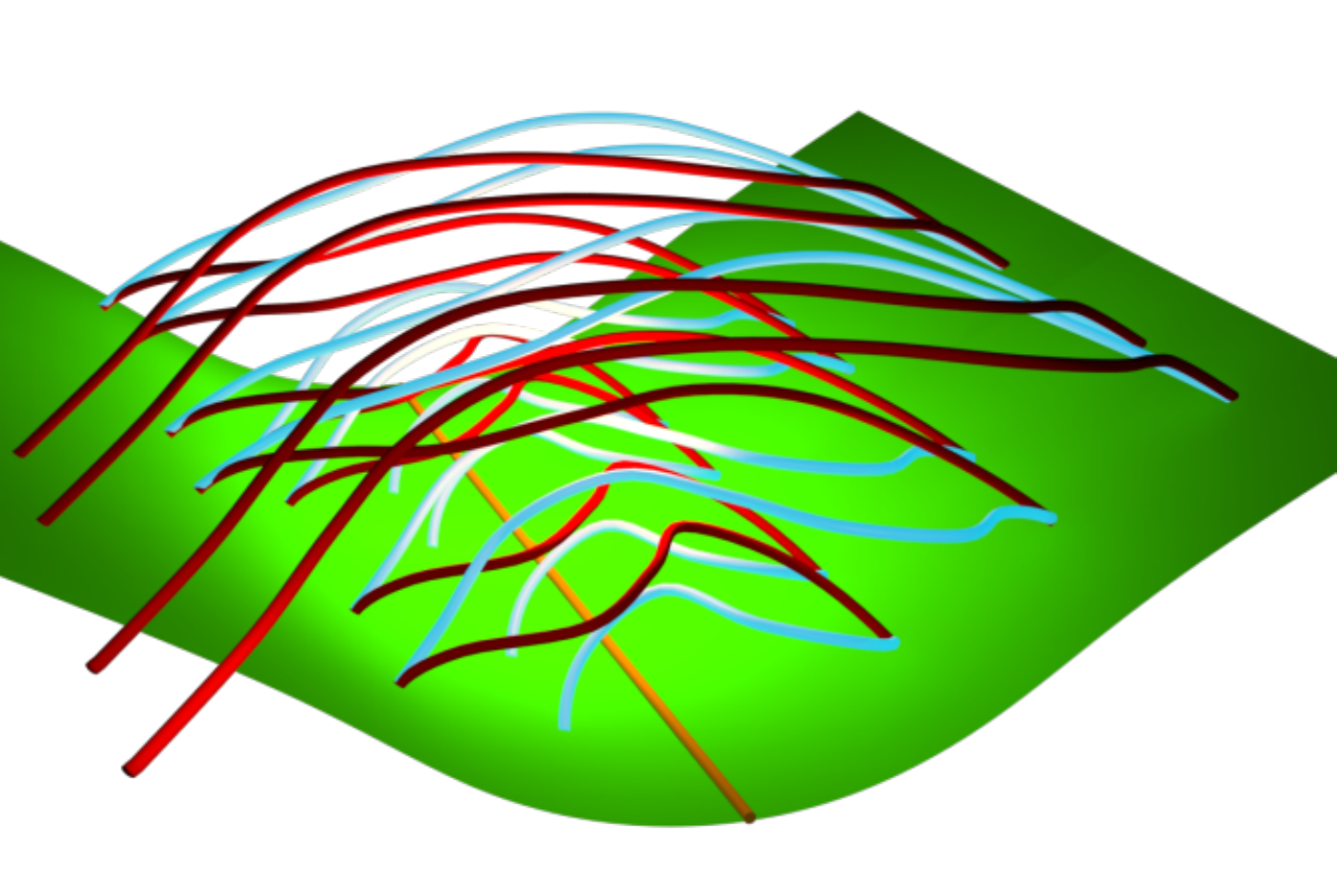}\label{snf6}}
		\\
		\captionsetup[subfigure]{width=.5\linewidth}
		\subfloat[][Integral curves (coloured as pink) of the Lie Cartan vector field with a curve (coloured as green) of singular points of focus type. The criminant surface is coloured as blue.]{\includegraphics[width=.4\textwidth]{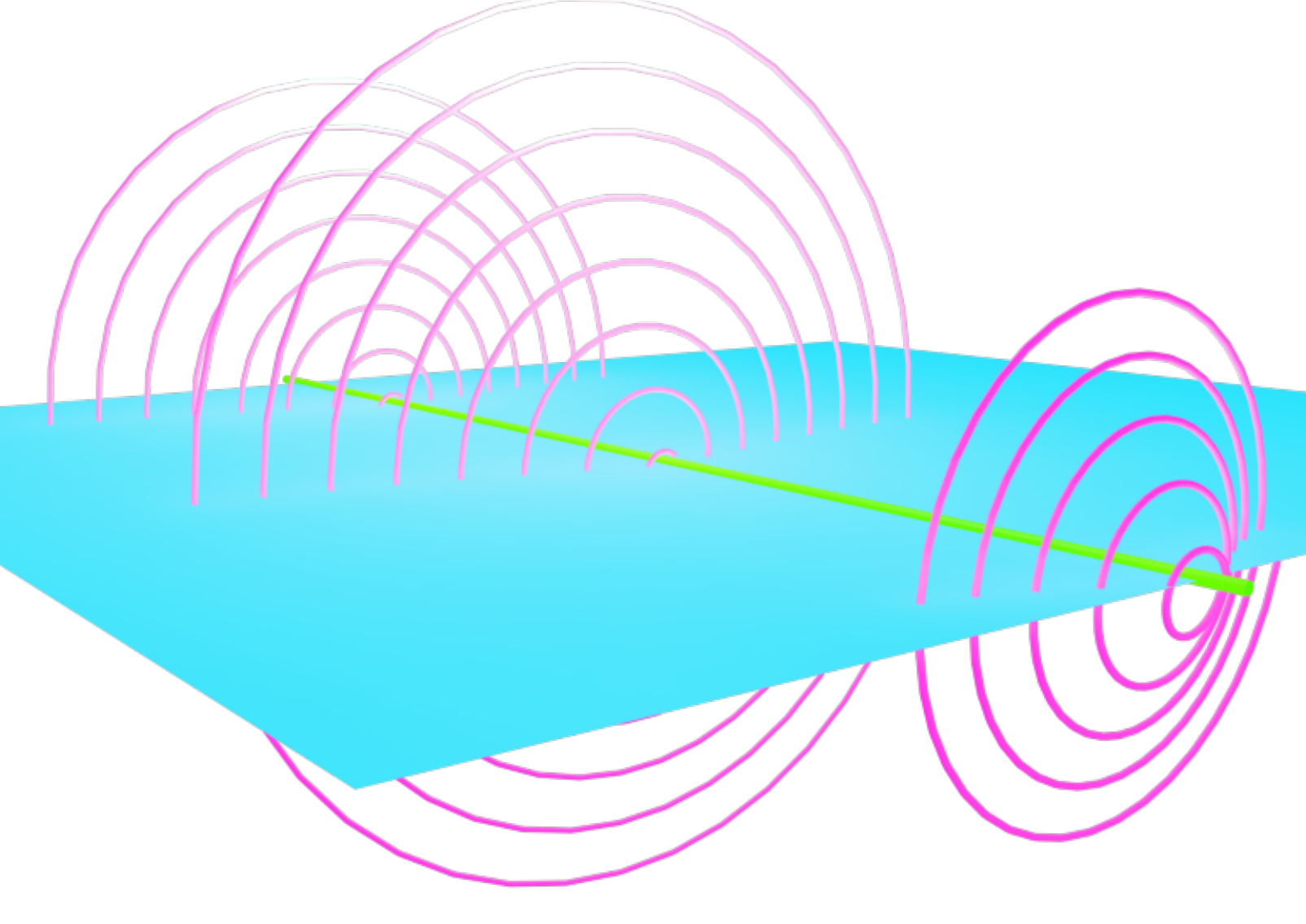}\label{snf3}}
		\qquad
		\captionsetup[subfigure]{width=.3\linewidth}
		\subfloat[][Frontal view of \ref{snf6}.]{\includegraphics[width=.4\textwidth]{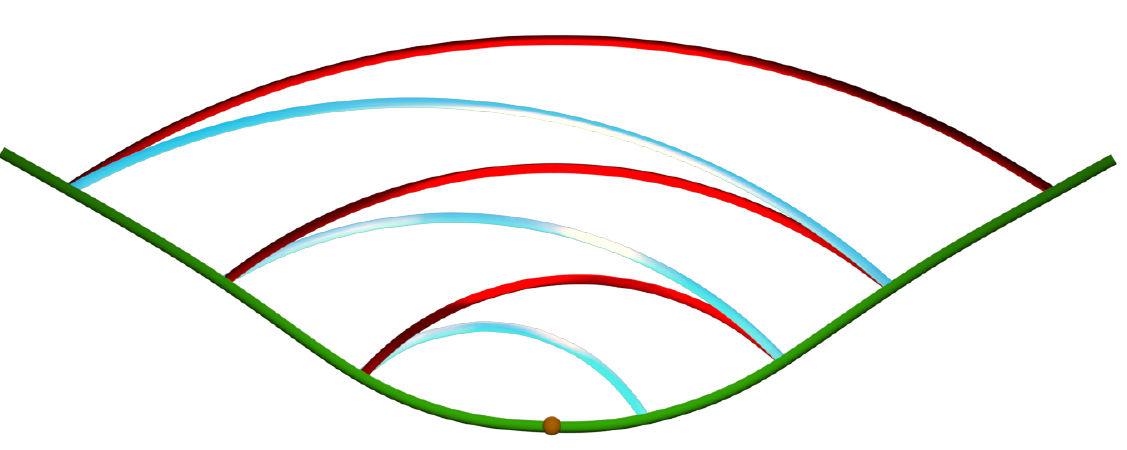}\label{snf9}}
		\caption{Parabolic point of focus type.}
		\label{snfp3}
	\end{figure}
	
	\begin{proof}
		By Proposition \ref{prop13}, the criminant set is a regular surface and
		by Proposition \ref{prop12}, at each singular point $\widetilde{\varphi}(t)$, both eigenvectors of $DX(\widetilde{\varphi}(t))$ are transversal to the criminant surface.
		Then the integral curves of $\mathcal{X}$ near $(0,0,0,0)$ is as show in Figures \ref{snf1}, \ref{snf2}, \ref{snf3}. By Proposition \ref{prop0}, the asymptotic lines are the projection of
		the integral curves of $\mathcal{X}$ by $\pi(x,y,z,p)=\pi(x,y,z)$. Then the asymptotic lines near $(0,0,0)$ is as show in Figures
		\ref{snf4}, \ref{snf7}, \ref{snf2}, \ref{snf8}, \ref{snf6}, \ref{snf9}.
	\end{proof}

	\subsection{Parabolic point of node-focus transition type}
	\label{node-focus}

	\begin{defn}
		A singular point $(x,y,z,p)$ of the
		Lie-Cartan vector field $\mathcal{X}$ is called node-focus if the eigenvalues $\lambda_{1}$ and $\lambda_{2}$ of $D\mathcal{X}(x,y,z,p)$, given in Proposition \ref{prop2},
		satisfies the following conditions:
		\begin{itemize}
			\item $\lambda_{1}-\lambda_{2}=0$ and $(\lambda_{1},\lambda_{2})\neq(0,0)$
			at $(x,y,z,p)$,
			\item the $\Omega$ given by \eqref{omega}  changes sign at $(x,y,z,p)$ as we move along $\widetilde{\varphi}$.
		\end{itemize}
	\end{defn}

By Propositions \ref{prop10} and \ref{propsing}, $b_{2}\neq0$, $a_{2}=-b_{1}$ and $a_{11}=0$. It follows that, at $(0,0,0,0)$,
\begin{equation}
\lambda_{i}=-\frac{a_{3}b_{1}+a_{12}}{2}\pm\frac{\sqrt{(3a_{12}-a_{3}b_{1}+4b_{11})^2-48b_{2}a_{111}}}{2}.
\end{equation}

The assumption $\lambda_{1}-\lambda_{2}=0$ means that

\begin{equation}
a_{111}=\frac{(3a_{12}-a_{3}b_{1}+4b_{11})^2}{48b_{2}}.
\end{equation}
	
	\begin{thm}\label{thm:node-focus}
		Suppose that $(0,0,0)$ is a parabolic point where the two asymptotic directions coincide at it
		and suppose
		there exists a curve $\varphi$ of parabolic points passing by $(0,0,0)$ such that the two asymptotic directions coincide at each of then and
		where its lift $\widetilde{\varphi}$ to the Lie-Cartan hypersurface is a curve of singular points of the
		Lie-Cartan vector field $\mathcal{X}$ such that at $(0,0,0,0)$ there exists a node-focus transition.
		
		Then the integral curves of $\mathcal{X}$ near $(0,0,0,0)$ is as show in Figure \ref{fig:foco-no} and the asymptotic lines near $(0,0,0)$ is as shown in Figure \ref{fig:la-foco_no}.
	\end{thm}

 
	\begin{figure}[H]
	\captionsetup[subfigure]{width=.3\linewidth}
	\centering
	\subfloat{\includegraphics[trim=0 0 0 100,clip,width=.7\textwidth]{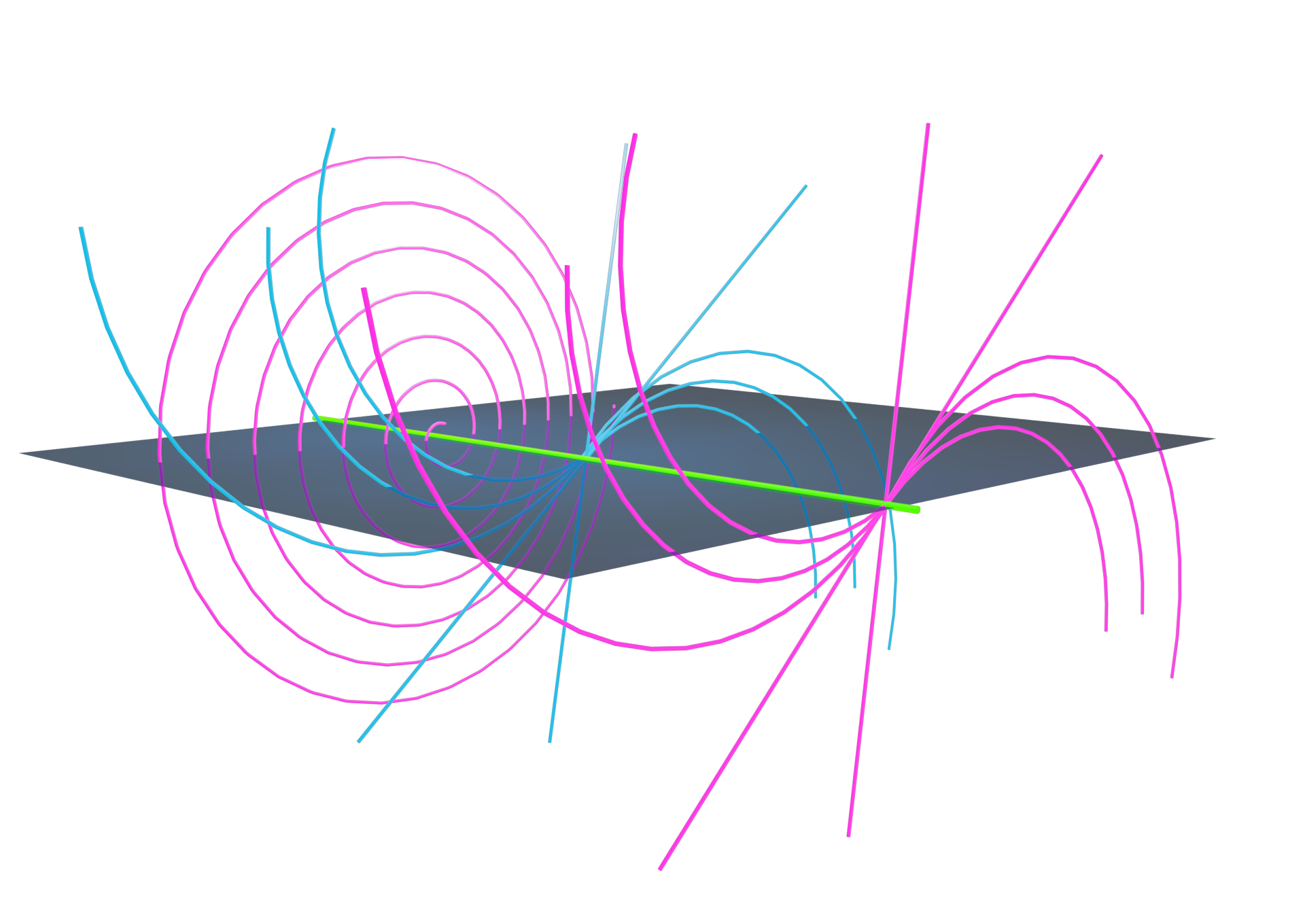}}
	\caption{Integral curves (coloured as pink and blue) of the Lie Cartan vector field with a curve (coloured as green) of singular points of node and focus type and the point of node-focus type. The tangent plane of the criminant surface at the node-focus point is coloured as transparent grey.}
	\label{fig:foco-no}
\end{figure}

	\begin{figure}[H]
	\centering
	\subfloat{\includegraphics[width=.7\textwidth]{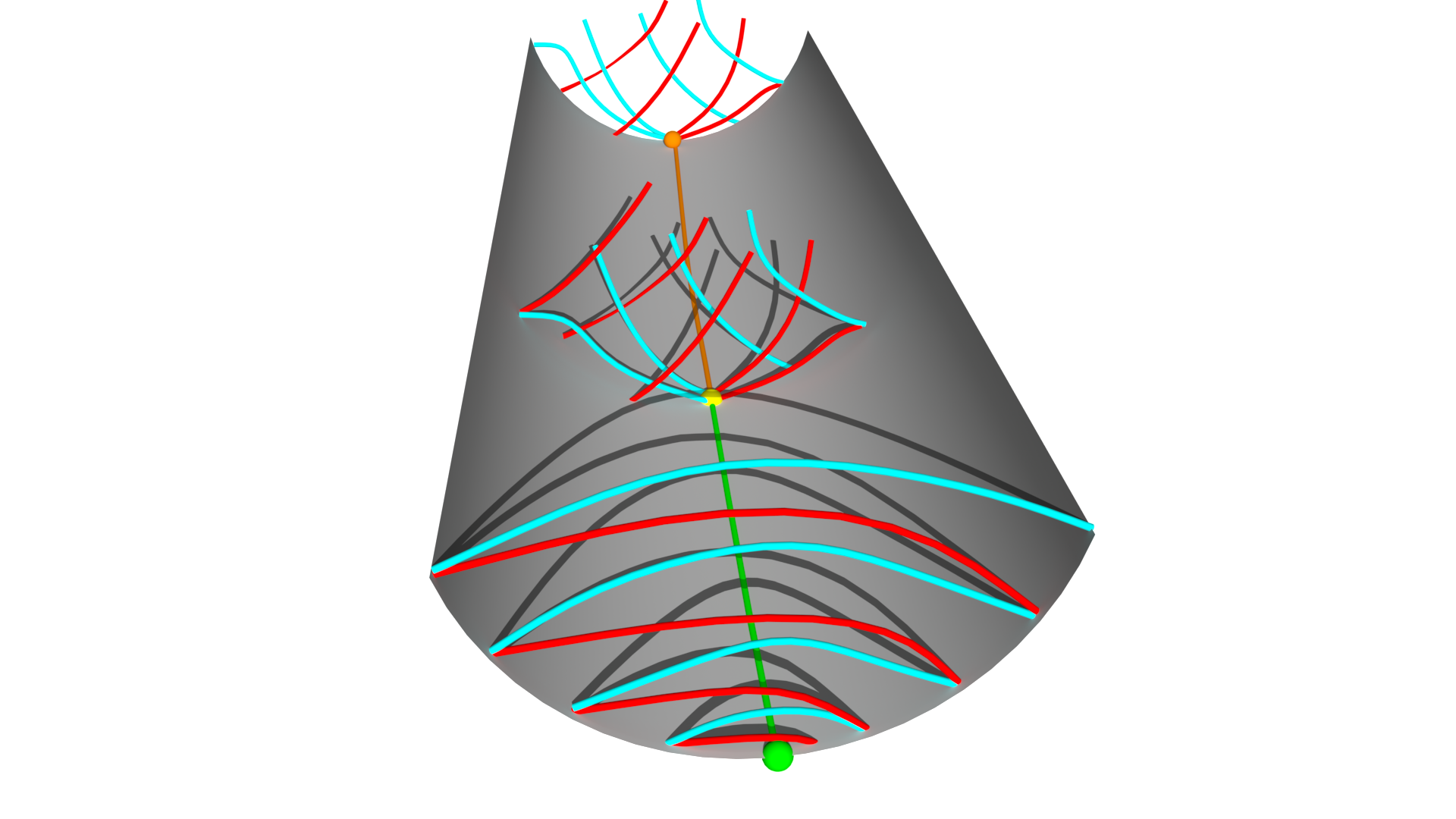}}
	\caption{Projection by $\pi$ of the integral curves of Fig. \ref{fig:foco-no} onto the asymptotic lines (coloured as red and blue), the criminant surface of Fig. \ref{fig:foco-no} onto the parabolic surface (coloured as gray) and the curve of singular points of Fig. \ref{fig:foco-no} onto the curve of parabolic points of the type foci (coloured as green) and nodes (coloured as orange). The node-focus point is the yellow point.}
	\label{fig:la-foco_no}
\end{figure}

\begin{figure}[H]
	\centering
		\subfloat{\includegraphics[width=.9\textwidth]{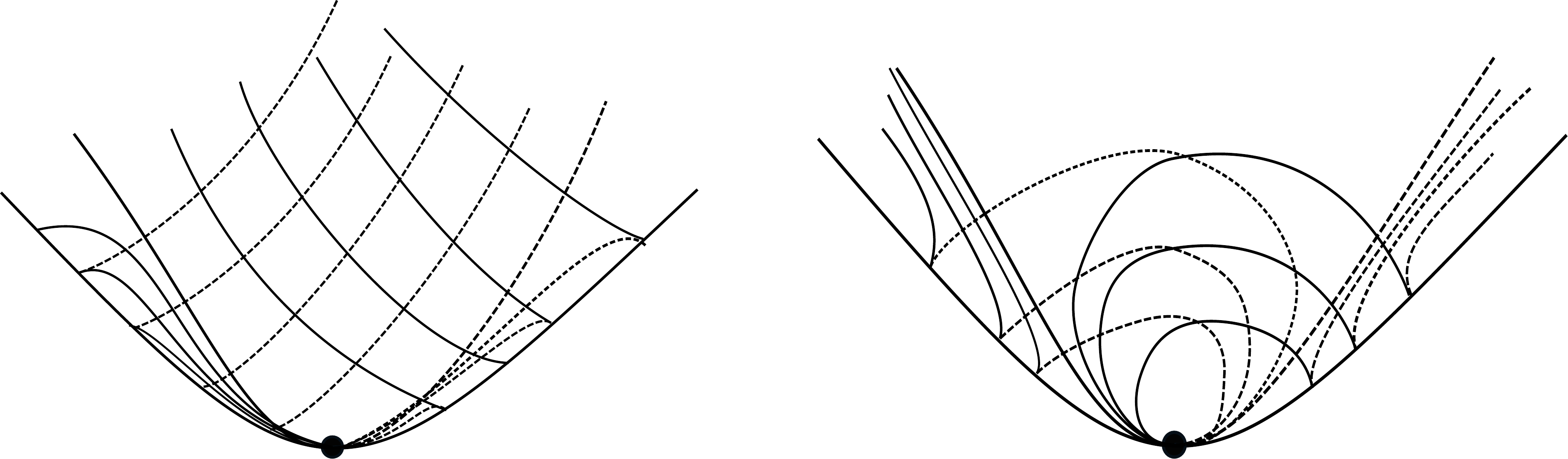}}
	\caption{Projection by $\pi$ of the integral curves of Fig. \ref{fig:foco-no} onto the asymptotic lines at the node-focus parabolic point.}
	\label{fig:foco_no}
\end{figure}

	\begin{proof}
		The singular point $(0,0,0,0)$ is of node type since $\lambda_{1}=\lambda_{2}\neq0$
		at it.
		
		By Proposition \ref{prop12}, both node weak and strong separatrices are not tangent to the criminant surface at $(0,0,0,0)$.
		
		\textcolor{blue}{}
	\end{proof}

	\subsection{Parabolic point of saddle-node transition type}
	\label{saddle-node}
	\begin{thm}\label{t3}
		Suppose that $(0,0,0)$ is a parabolic point where the two asymptotic directions coincides at it
		and suppose
		there exists a curve $\varphi$ of parabolic points passing by $(0,0,0)$ such that the two asymptotic directions coincides at each of then and
		where his lift $\widetilde{\varphi}$ to the Lie-Cartan hypersurface is a curve of singular points of the
		Lie-Cartan vector field $\mathcal{X}$ such that at $(0,0,0,0)$ there exists a saddle-node transition.
		More specifically,
		the eigenvalues
		$\lambda_{1}$ and $\lambda_{2}$, given in \ref{prop2}, satisfies the conditions that $\lambda_{1}\lambda_{2}$ changes sign at $(0,0,0,0)$ as we move along $\widetilde{\varphi}$ and $\lambda_{1}\lambda_{2}=0$, $(\lambda_{1},\lambda_{2})\neq(0,0)$
		at $(0,0,0,0)$.
	\end{thm}
	
		\begin{figure}[H]
		\centering
 	\includegraphics[width=.4\textwidth]{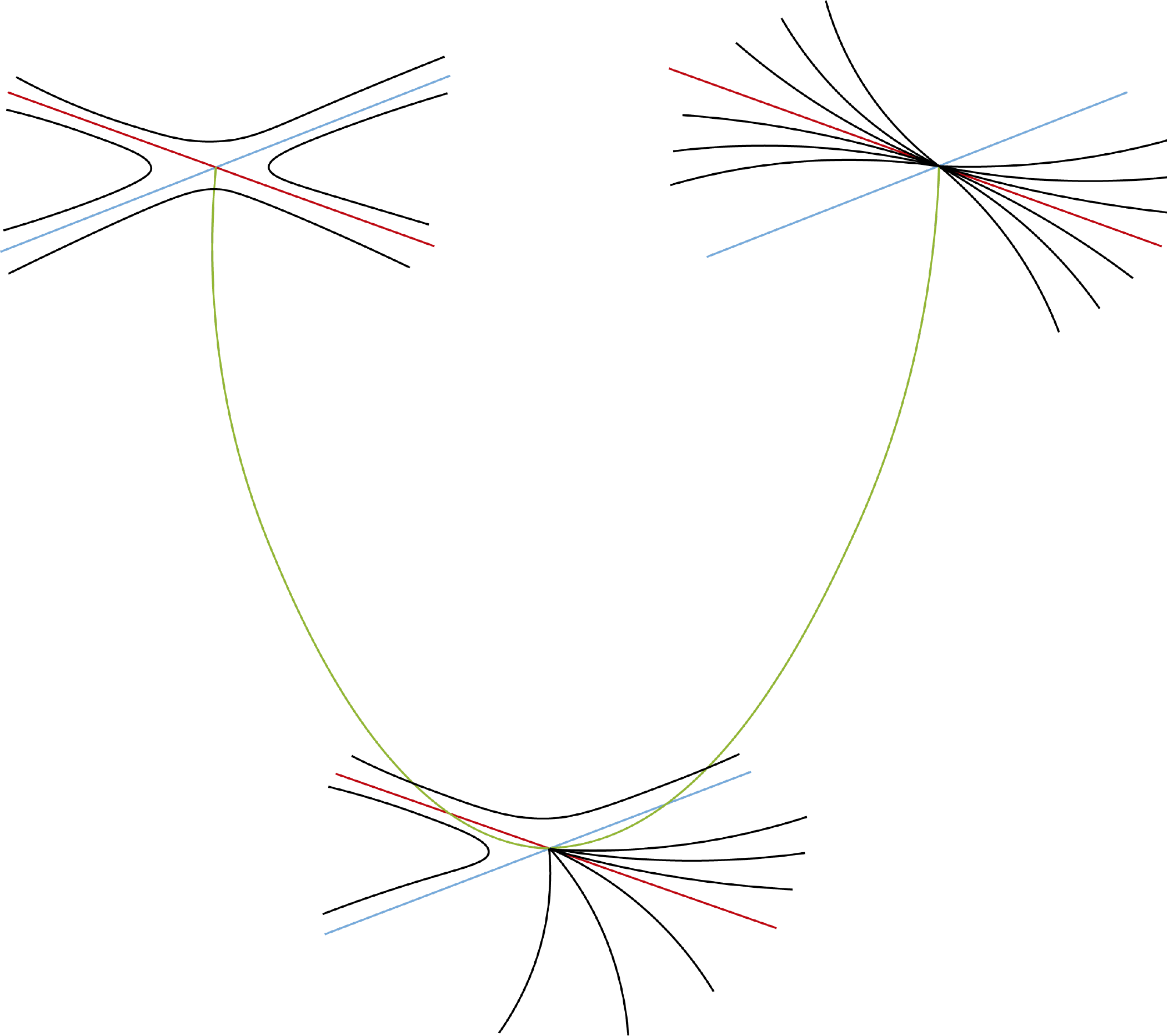}
		\caption{Saddle-node parabolic point. The curve of singular points of the Lie-Cartan vector field is coloured with green.
			At a singular point, the strong separatrix (resp. weak separatix) is coloured with blue (resp. red).}
		\label{sd}
	\end{figure}
	
	\begin{figure}[H]
		\centering
 \includegraphics[trim=0 0  0 20,clip,width=.3\textwidth]{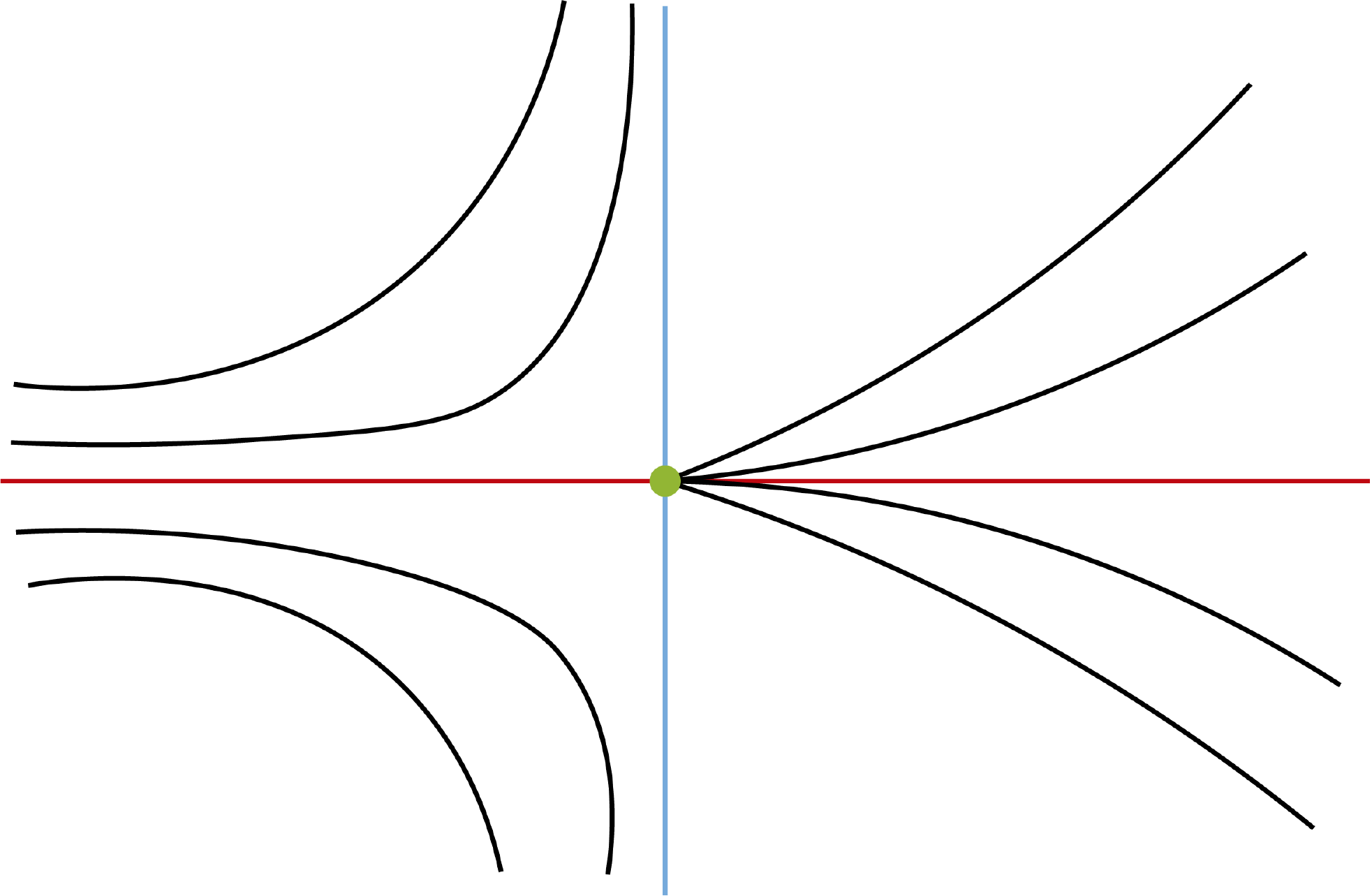}
		\caption{Saddle-node parabolic point. The strong separatrix (resp. weak separatix) is coloured with blue (resp. red).}
		\label{sd2}
	\end{figure}

\begin{figure}[H]
	\centering
 	\includegraphics[trim=250 50 200 50,clip,width=.55\textwidth]{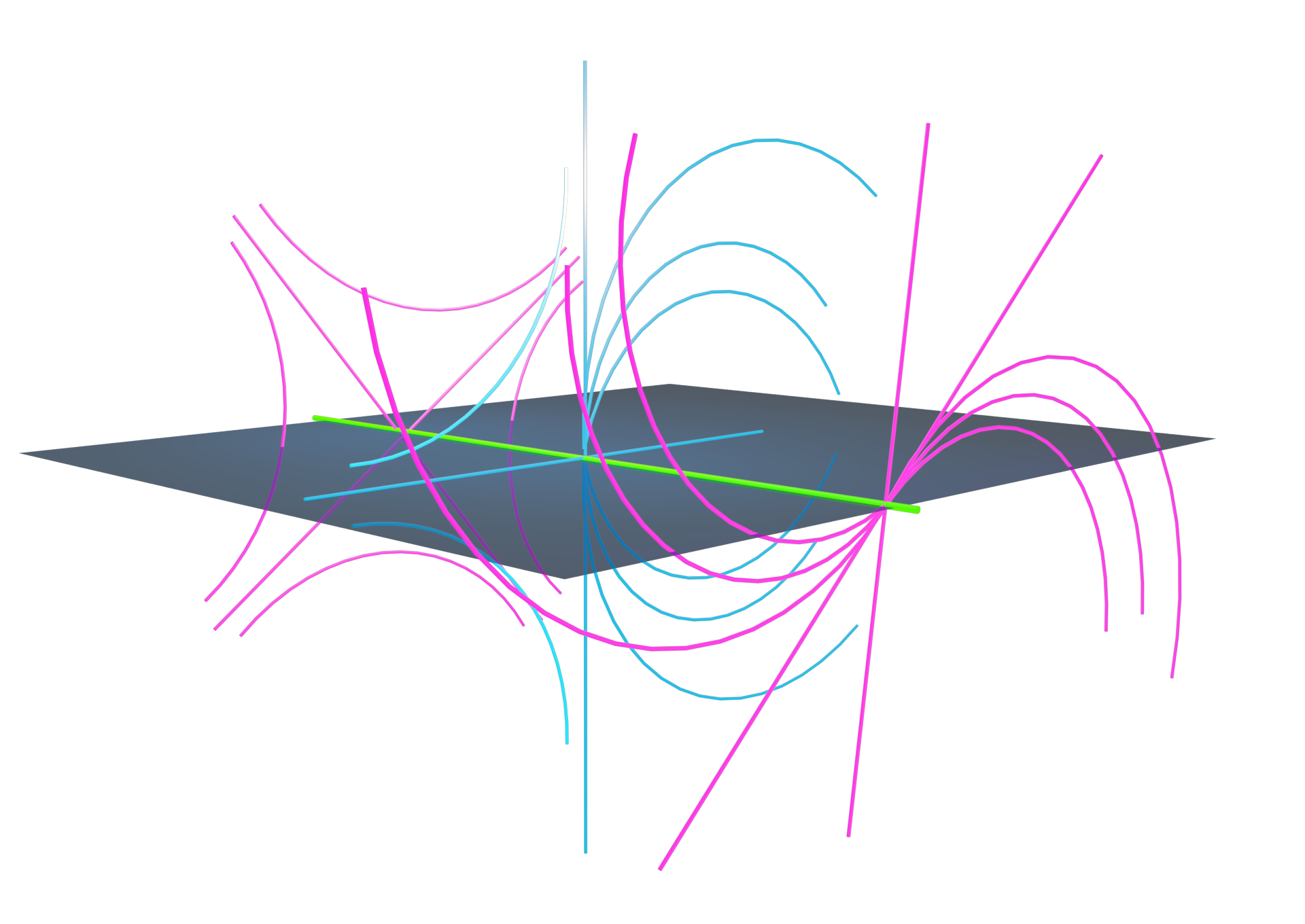}
	\caption{Integral curves (coloured as pink and blue) of the Lie Cartan vector field with a curve (coloured as green) of singular points of saddle and node type and the point of saddle-node type. The tangent plane of the criminant surface at the saddle-node point is coloured as transparent grey.  The strong separatrix (resp. weak separatix) is coloured with blue (resp. red).}
	\label{lc-sn}
\end{figure}

\begin{figure}[H]
	\centering
 \includegraphics[trim=250 50 200 50,clip,width=.6\textwidth]{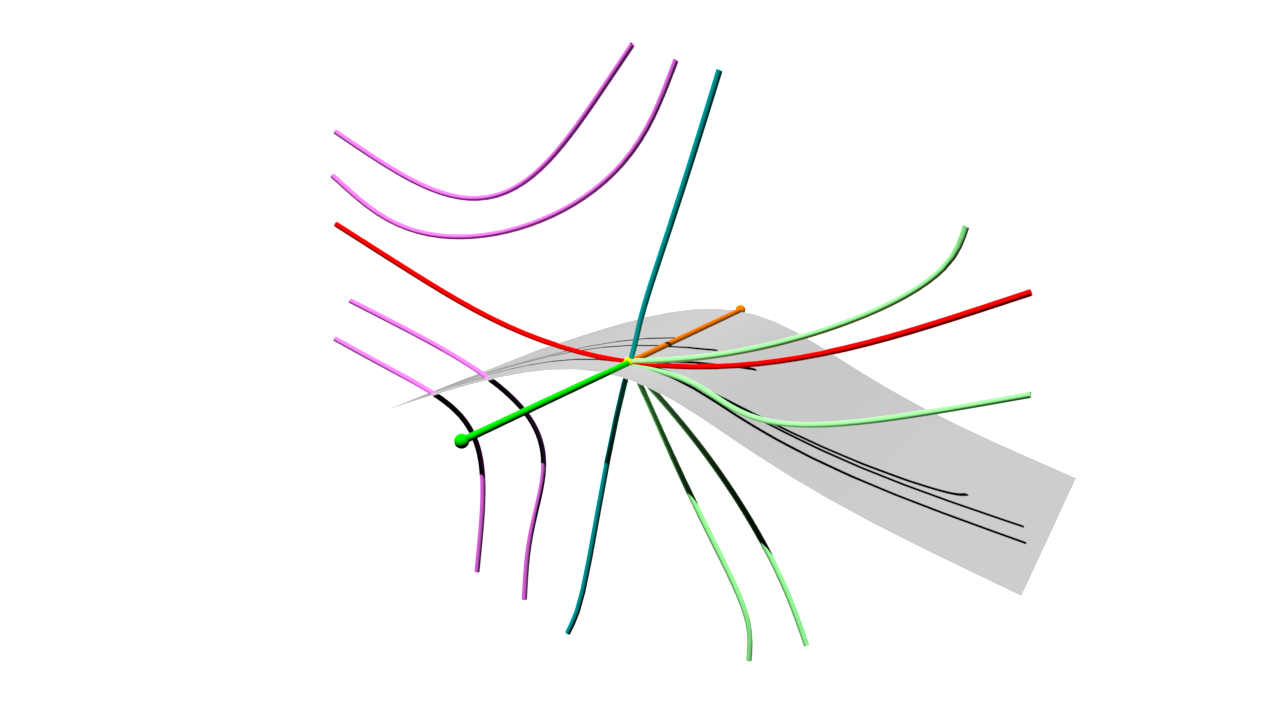}
	\caption{Integral curves of the Lie Cartan vector field with a curve of singular points of saddle (coloured as green)  and node (coloured as orange) type and the point of saddle-node type (coloured as yellow). The criminant surface is coloured as grey. The strong separatrix (resp. weak separatix) is coloured with blue (resp. red). The weak separatix has quadratic contact with
		the criminant.}
	\label{lc-sn-q}
\end{figure}

\begin{figure}[H]
	\centering
{	\includegraphics[trim=250 50 200 50,clip,width=.6\textwidth]{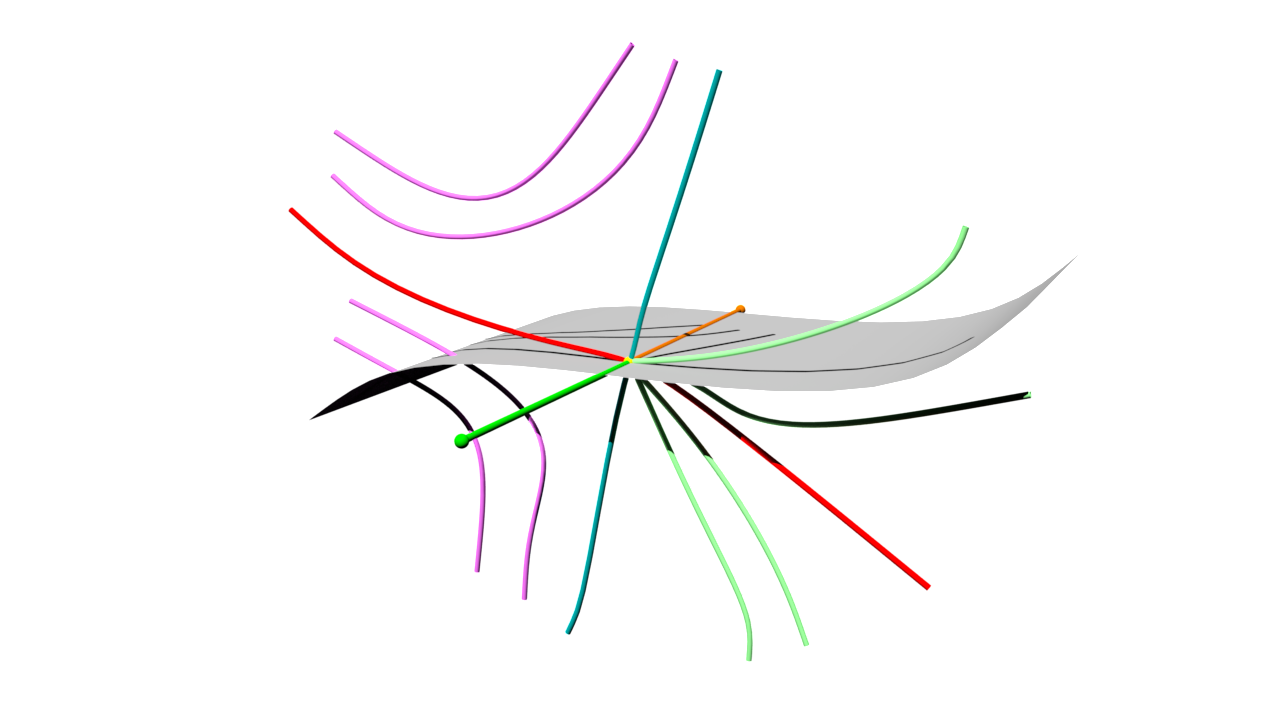}}
	\caption{Integral curves of the Lie Cartan vector field with a curve of singular points of saddle (coloured as green)  and node (coloured as orange) type and the point of saddle-node type (coloured as yellow). The criminant surface is coloured as grey. The strong separatrix (resp. weak separatix) is coloured with blue (resp. red). The weak separatix has cubic contact with
		the criminant.}
	\label{lc-sn-c}
\end{figure}

	\begin{figure}[H]
	\centering
 	\includegraphics[trim=450 50 300 50,clip,width=.6\textwidth]{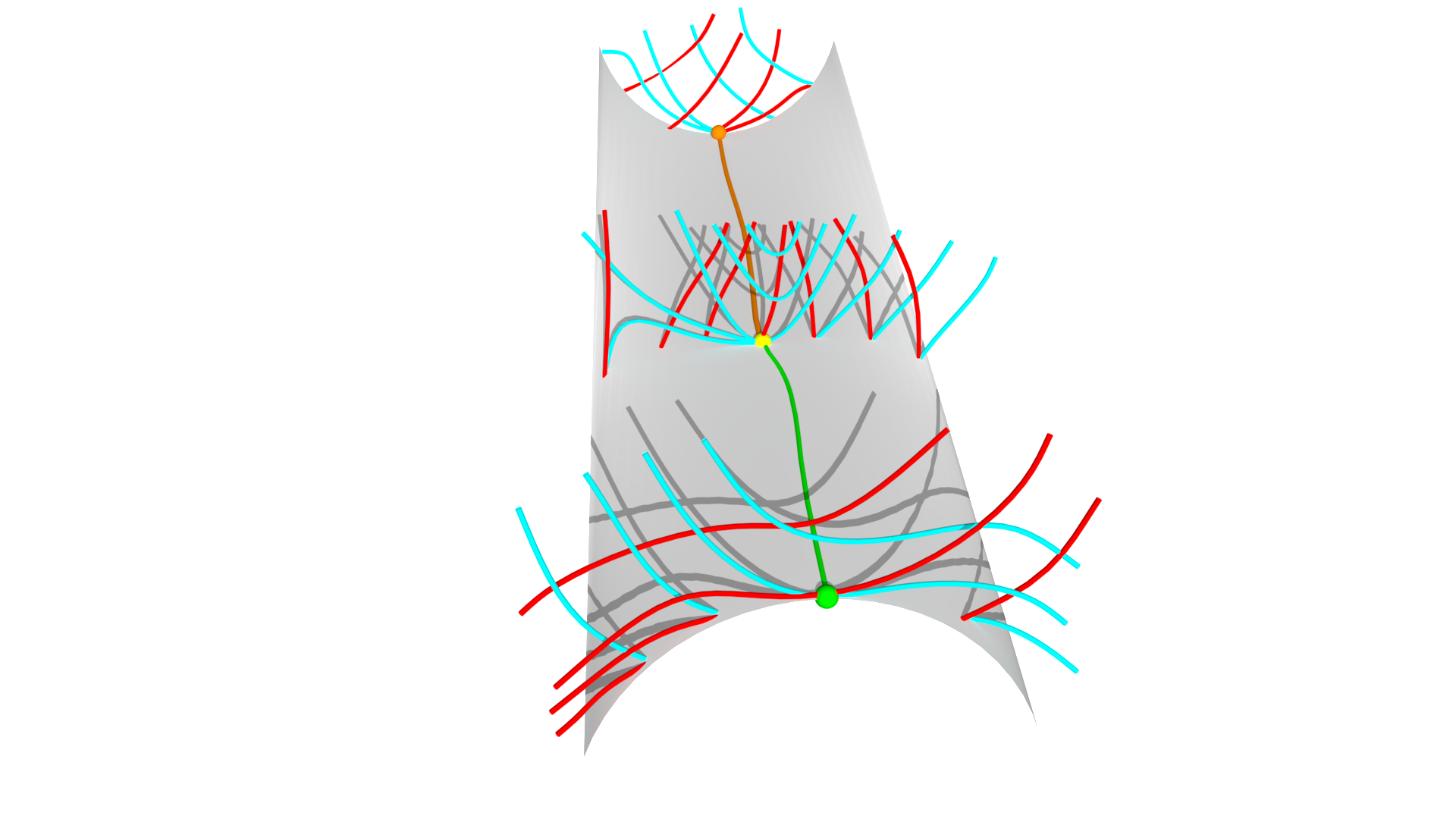} 
	\caption{Projection by $\pi$ of the integral curves of Fig. \ref{lc-sn-q}  onto the asymptotic lines (coloured as red and blue), the criminant surface of Fig. \ref{lc-sn-q} onto the parabolic surface (coloured as gray) and the curve of singular points of \ref{lc-sn-q} onto the curve of parabolic point of the type saddle (coloured as green) and node (coloured as orange). The saddle-node point is the yellow point.}
	\label{la-sn-q}
\end{figure}

	\begin{figure}[H]
	\centering
 \includegraphics[trim=450 50 300 50,clip,width=.6\textwidth]{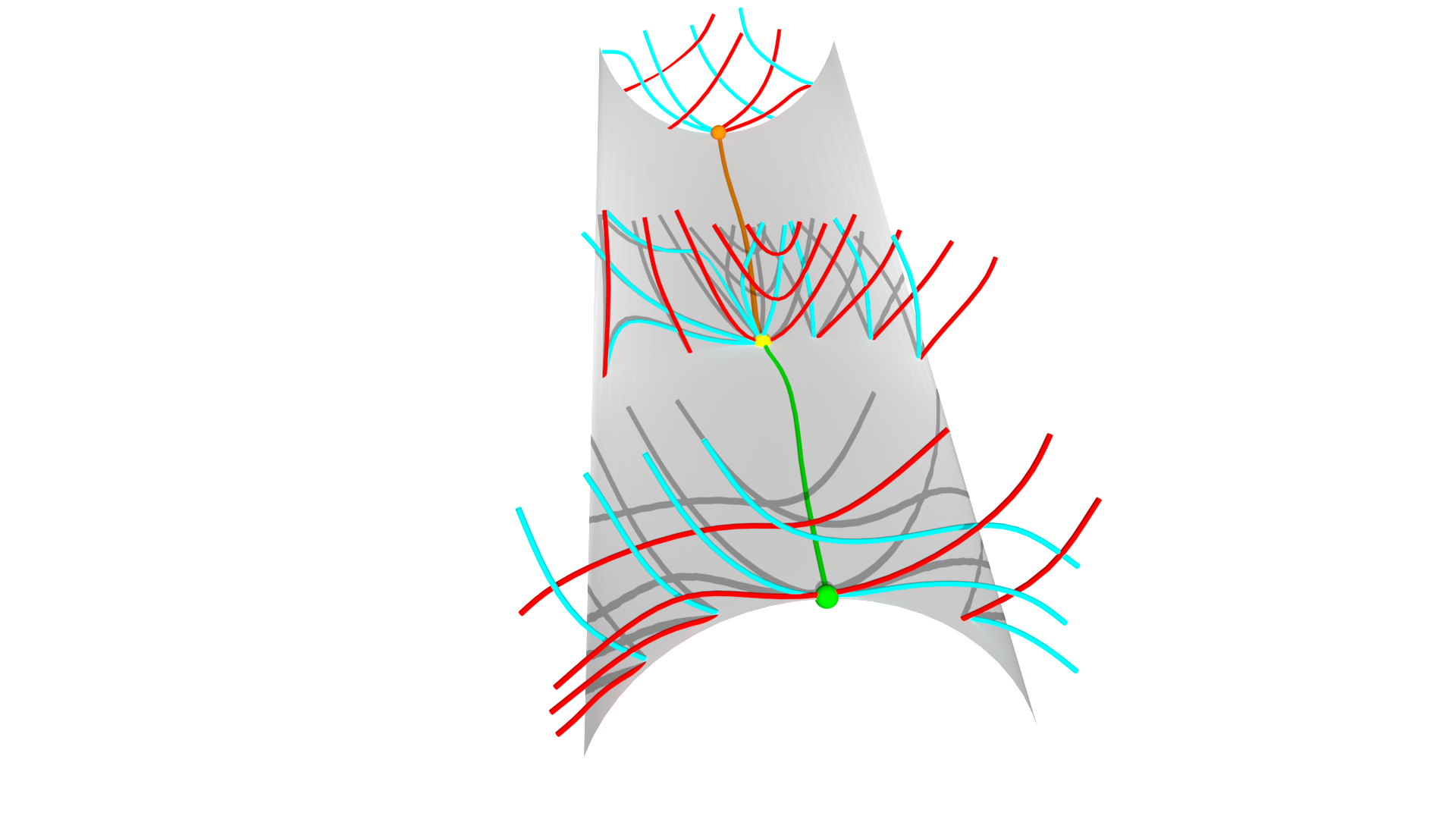} 
	\caption{Projection by $\pi$ of the integral curves of Fig. \ref{lc-sn} onto the asymptotic lines (coloured as red and blue), the criminant surface of Fig. \ref{lc-sn-c} onto the parabolic surface (coloured as gray) and the curve of singular points of \ref{lc-sn-c} onto the curve of parabolic point of the type saddle (coloured as green) and node (coloured as orange). The saddle-node point is the yellow point.}
	\label{la-sn-c}
\end{figure}
	
	\begin{proof}
		By Propositions \ref{prop10} and \ref{propsing}, $b_{2}\neq0$, $a_{2}=-b_{1}$ and $a_{11}=0$.
		The assumption $\lambda_{1}\lambda_{2}=0$ means that
		\begin{equation}
		a_{111}=\frac{(a_{12})^2+3a_{12}b_{11}+2(b_{11})^2-a_{3}b_{1}(a_{12}+b_{11})}{6b_{2}}.
		\end{equation}
		With that, $\lambda_{1}=0$. Now, the assumption that $\lambda_{2}\neq0$ means that $a_{3}b_{1}+a_{12}\neq0$.
		By the Implicit Function Theorem, the Lie-Cartan hypersurface in a neighbourhood of
		$(0,0,0,0)$ is parametrized by $y=y(x,z,p)$, where
		\begin{equation}\label{snLCH}
		\begin{split}
		&y(x,z,p)=\left(\frac{(a_{3})^2-a_{13}}{a_{3}b_{1}+a_{12}}\right)z+L_{1}pz+L_{2}xz+L_{3}z^2-\left(\frac{2b_{2}}{a_{3}b_{1}+a_{12}}\right)p^2\\
		&+\left(\frac{(a_{12}+b_{11})(a_{3}b_{1}-a_{12}-2b_{11})}{(a_{3}b_{1}+a_{12})b_{2}}\right)x^2+\left(\frac{a_{3}b_{1}-a_{12}-2b_{11}}{a_{3}b_{1}+a_{12}}\right)px
		+\mathcal{O}^3(x,z,p).
		\end{split}
		\end{equation}
		Let $\mathcal{Y}$ be the Lie-Cartan vector field $\mathcal{X}$ restricted to \eqref{snLCH}. Then
		\begin{equation}
		\mathcal{Y}(x,z,p)=(F_{p},qF_{p},-(F_{x}+pF_{y}+qF_{z}))(x,y(x,y,p),z,p).
		\end{equation}
		The eigenvector $\vartheta_{2}$ associated to the eigenvalue $\lambda_{2}=-(a_{3}b_{1}+a_{12})$ is given by
		$\vartheta_{2}=\left(1,0,-\frac{a_{12}+b_{11}}{b_{2}}\right)$ and
		the eigenvector $\vartheta_{1}$ associated to the eigenvalue $\lambda_{1}=0$
		is given by $\vartheta_{1}=\left(1,0,\frac{a_{3}b_{1}-a_{12}-b_{11}}{2b_{2}}\right)$.
		
		The criminant surface $F=F_{p}=0$ can be parametrized by $p=p(x,z)$ and the curve $\widetilde{\varphi}$ can be parametrized by $z=z(x)$.
		
		We have that $\frac{\partial}{\partial x}(\lambda_{1}\lambda_{2})(\widetilde{\varphi})=(\lambda_{1})_{x}(0)\lambda_{2}\neq0$.
		Then, by \cite[Theorem 4.1, p.42]{MR0501173}, there exist a invariant manifold $W^{s}(\widetilde{\varphi})$, of class $C^{k-3}$,
		and $\widetilde{\varphi}$ is normally hyperbolic attractor, which is locally given by $p(x,z)=0+\mathcal{O}^2(x,z)$.
		
		Let $\mathcal{Z}$ be the Lie-Cartan vector field $\mathcal{Y}$ restricted to $W^{s}(\widetilde{\varphi})$.
		After the change of coordinates $(\widetilde{x},\widetilde{z})=M\cdot(x,y)$, where
		\begin{equation}
		M=\left(
		\begin{array}{cc}
		\frac{(a_{12}+b_{11})(a_{12}-a_{3}b_{1}+2b_{11})}{b_{2}\lambda_{2}} & \frac{a_{12}-a_{3}b_{1}+2b_{11}}{\lambda_{2}} \\
		\frac{-(a_{12}+b_{11})(a_{12}-a_{3}b_{1}+2b_{11})}{b_{2}\lambda_{2}} & -\frac{2(a_{12}+b_{11})}{\lambda_{2}} \\
		\end{array}
		\right),
		\end{equation}
		\noindent we get
		$\widetilde{\mathcal{Z}}(x,z)=M\cdot\mathcal{Z}(\widetilde{x},\widetilde{z})=(A(x,z),\lambda_{2}z+B(x,z))$, where
		$A(0,0)=B(0,0)=A_{x}(0,0)=B_{x}(0,0)=A_{z}(0,0)=B_{z}(0,0)=0$.
		
		Let $z=z(x)$ be the solution of the equation $\lambda_{2}z+B(x,z)=0$ in a neighbourhood of $(0,0)$ and set $\widetilde{G}(x)=A(x,z(x))$. Performing the calculations, we have that $\widetilde{G}(x)=\rho_{2}x^2+\mathcal{O}^3(x)$.
		
		By \cite[Theorem 2.19, item $iii$, p. 74, 75]{MR2256001} the Lie-Cartan vector field, in a neighbourhood of $(0,0,0,0)$, is as show in the Figures \ref{sd} and \ref{sd2}.
		
		By Proposition \ref{prop12}, the strong separatrix (resp. weak separatrix) is transversal (resp. tangent) to the criminant surface at
		the point $(0,0,0,0)$.
	\end{proof}

	\subsection{Parabolic point with a pair of complex eigenvalues crossing the imaginary axis}
	\label{hopf}
	
	\begin{thm}\label{aci}
		Let $\varphi$ be a parabolic   curve where the tangent directions coincides and suppose $(0,0,0)$ that is a parabolic point
		of $\varphi$ where its lift $(0,0,0,0)$ to the Lie-Cartan hypersurface is a singular point of the Lie-Cartan vector-field
		which  possesses a pair of nonzero eigenvalues which cross the imaginary axis as we move along
		the curve of singular points $\widetilde{\varphi}$ in such way that the derivative $\delta$ of the real part of the
		eigenvalues \eqref{eigenVa} in the direction of the tangent of $\widetilde{\varphi}$ does not vanished when evaluated in $(0,0,0,0)$. This means
		that the nonzero eigenvalues cross the imaginary axis transversely.
		
		Then if $\delta>0$ (resp. $\delta<0$)
		the stable and unstable invariant manifolds of the Lie-Cartan vector field is as shown in the Figures \ref{hop1} (resp. Figures \ref{hop2}),
		and the asymptotic lines are 
		as shown in the Figure \ref{lahop1} (resp. Figure \ref{lahop2}).
	\end{thm}
	\begin{defn}
		In the Theorem \ref{aci}, the point $(0,0,0)$ is called Hopf parabolic point and $(0,0,0,\widetilde{p})$ is called of
		Hopf singular point of the Lie-Cartan vector field. The parabolic point of the case $(a)$ (respectively $(b)$) is called
		hyperbolic Hopf parabolic point (respectively elliptic Hopf parabolic point).
	\end{defn}
	
		\begin{figure}[H]
		\captionsetup[subfigure]{width=.3\linewidth}
		\centering
		\subfloat[][]{\includegraphics[width=.55\textwidth]{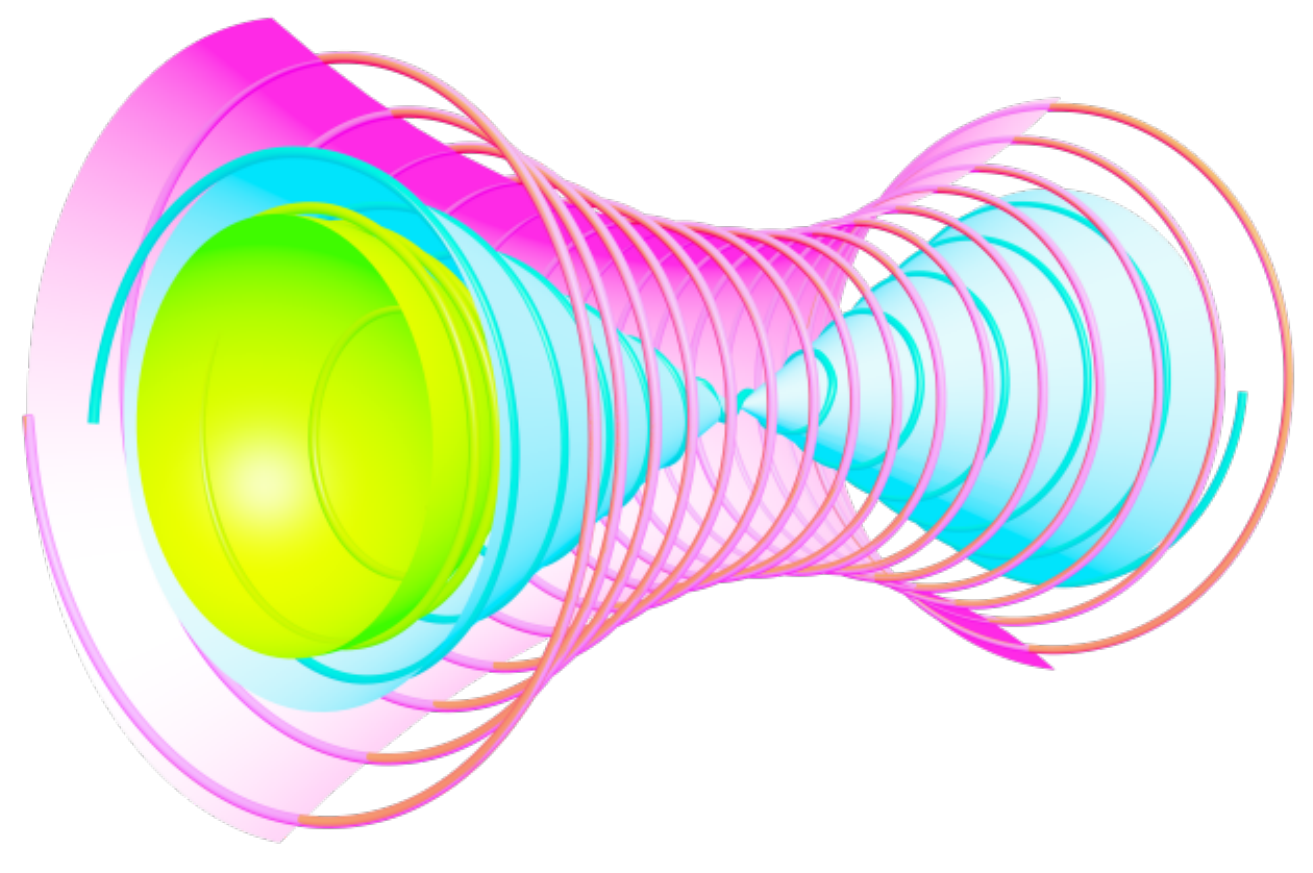}\label{thi2}}
		\\
		\subfloat[][]{\includegraphics[width=.35\textwidth]{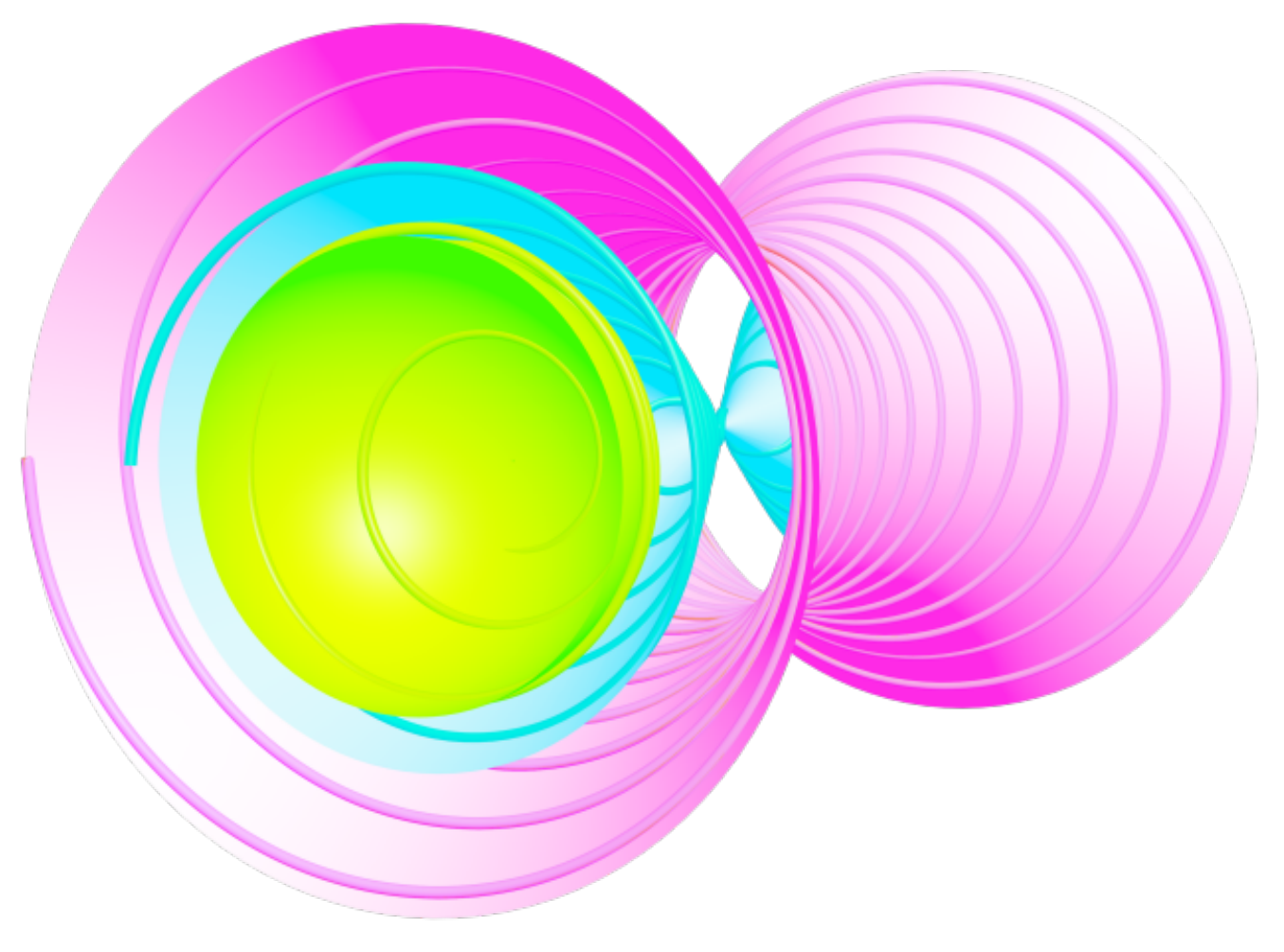}\label{thi4}}
		\subfloat[][]{\includegraphics[width=.35\textwidth]{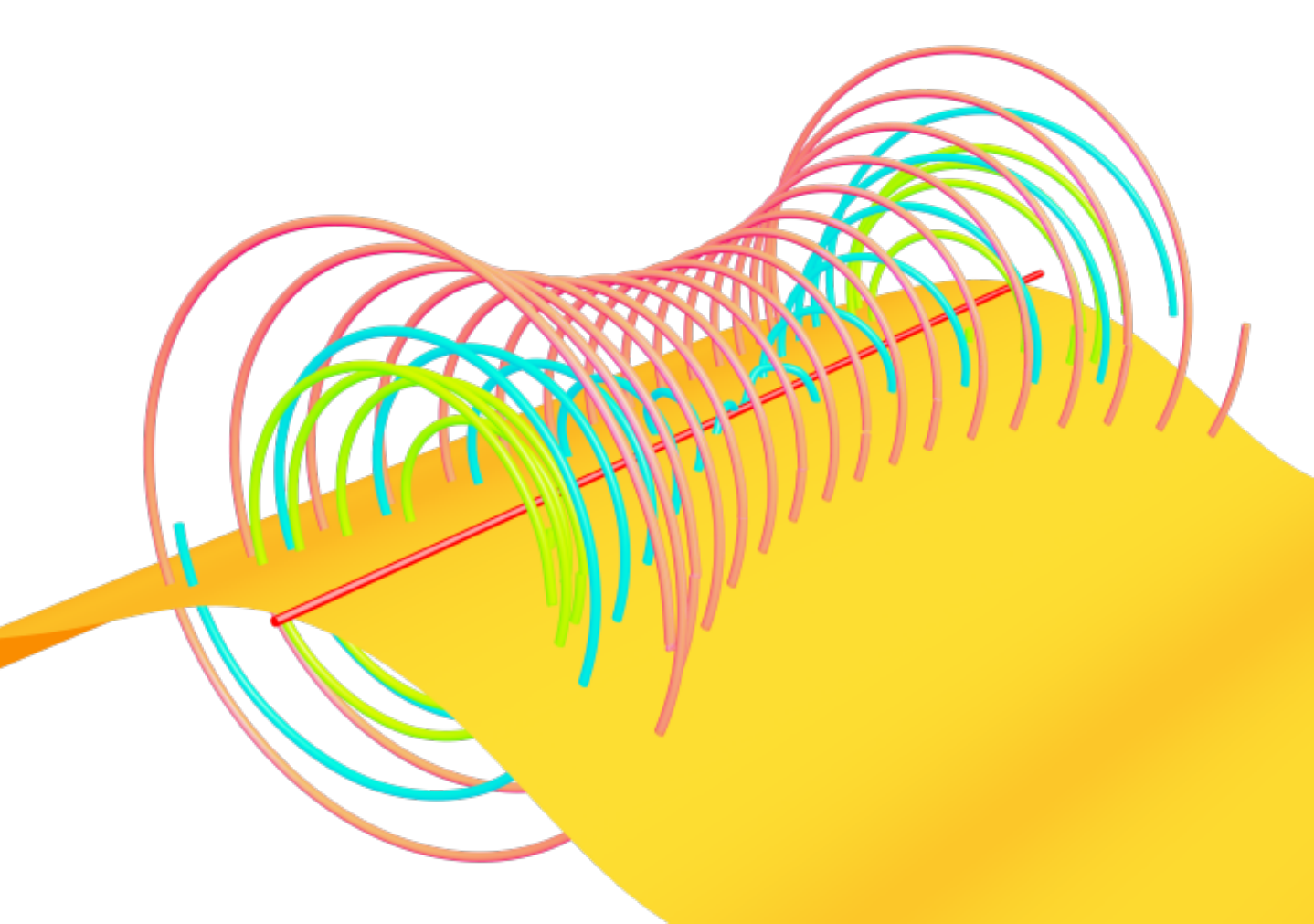}\label{thi3}}
		\\
		\subfloat[][]{\includegraphics[width=.65\textwidth]{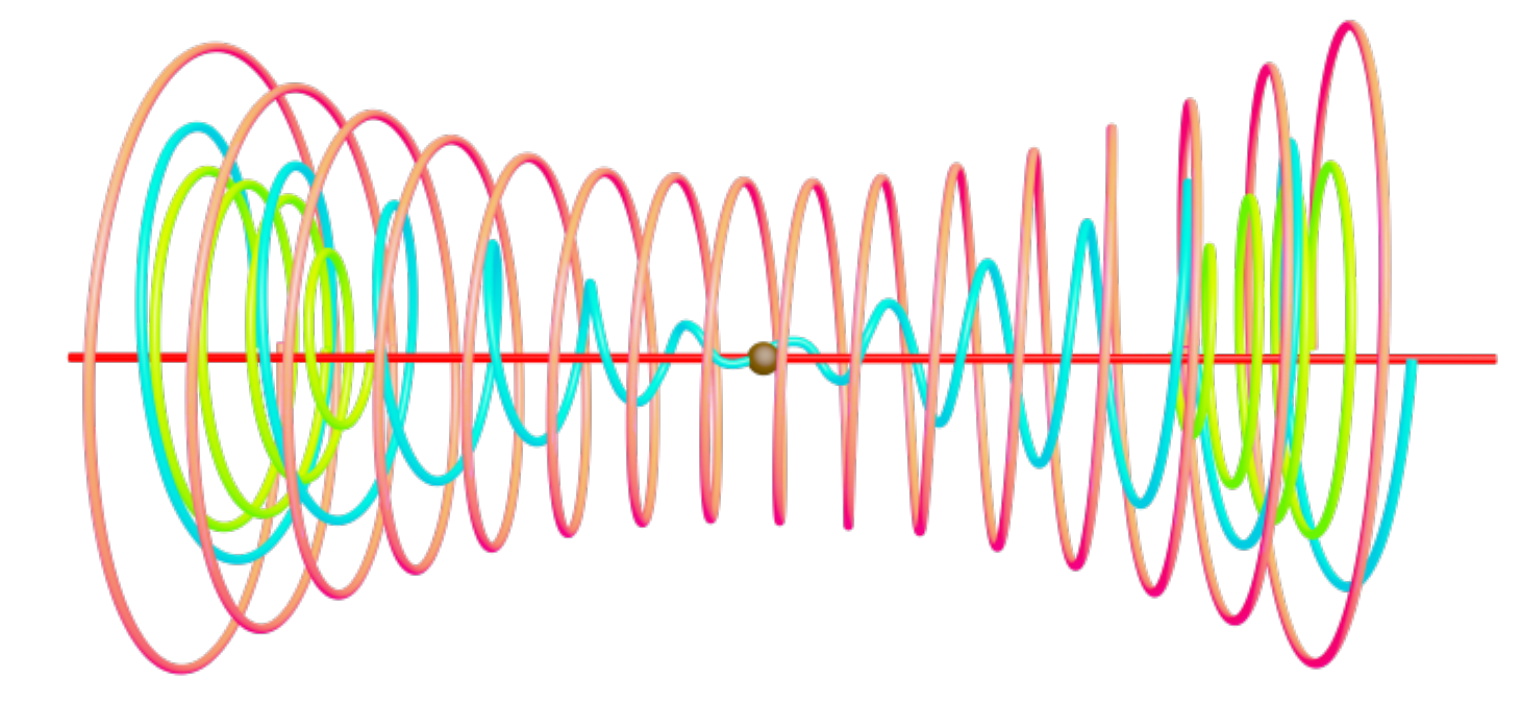}\label{thi5}}
		\caption{The hyperbolic Hopf singularity associated to the hyperbolic Hopf parabolic point.}
		\label{hop1}
	\end{figure}

\begin{figure}[H]
 	\captionsetup[subfigure]{width=.5\linewidth}
	\centering
 	\subfloat[][]{\includegraphics[width=.5\linewidth]{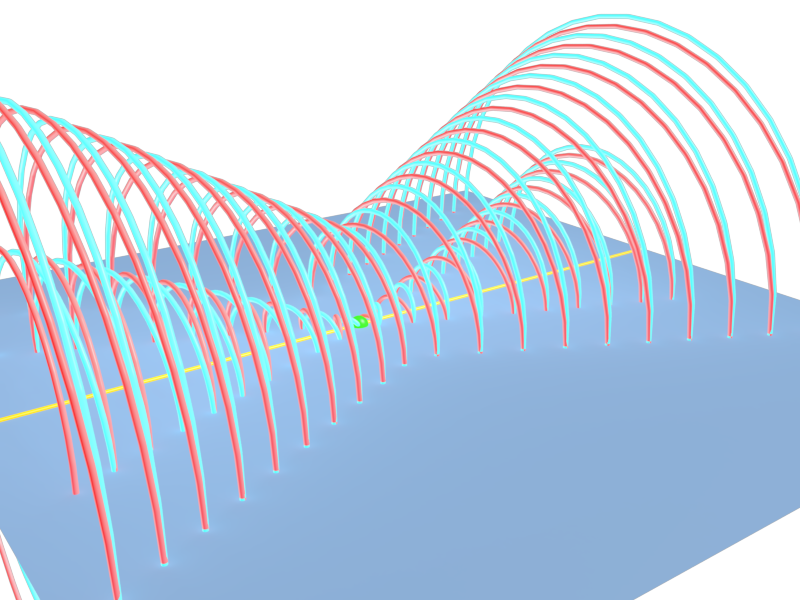}\label{lath1}}
	\qquad
	\subfloat[][]{\includegraphics[width=.5\linewidth]{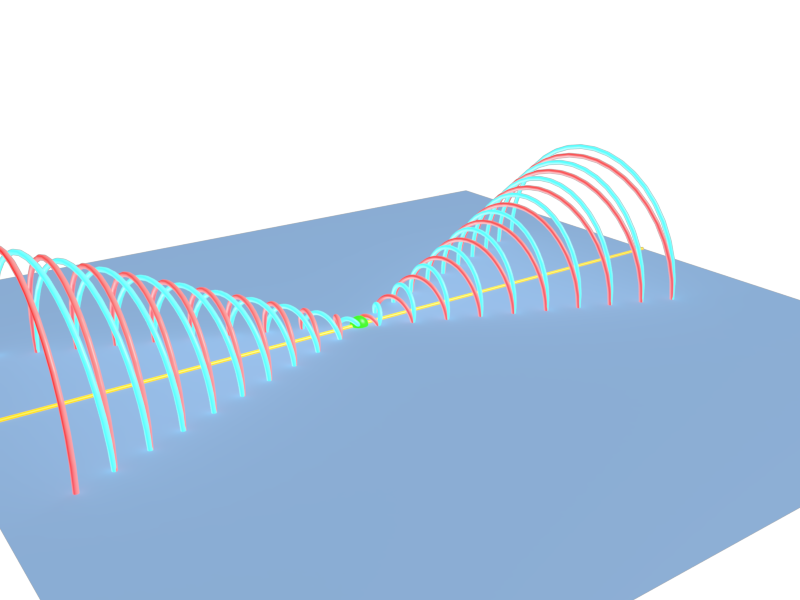}\label{lath2}}
	\caption{Projection by $\pi$ of the integral curves of 
	Fig. \ref{hop1} onto the asymptotic lines (coloured as red and blue), the criminant surface of Fig. \ref{hop1} onto the parabolic surface (coloured as gray) and the curve of singular points of Fig. \ref{hop1} onto the curve of special parabolic points (coloured as yellow). The Hopf point is the green point.}
	\label{lahop1}
\end{figure}

	\begin{figure}[H]
		\captionsetup[subfigure]{width=.3\linewidth}
		\centering
		\subfloat[][]{\includegraphics[width=.35\linewidth]{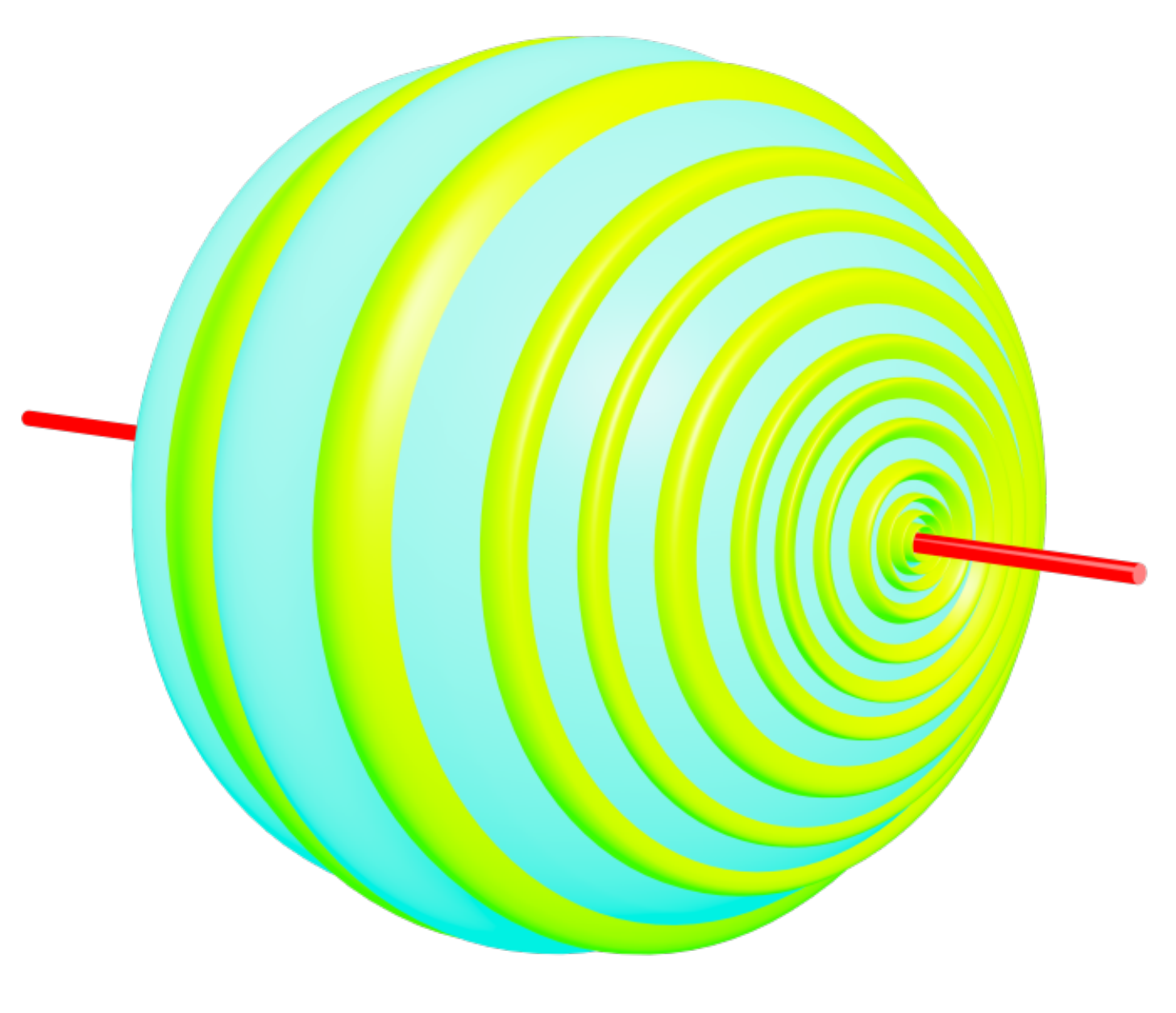}\label{th1}}
		\qquad
		\subfloat[][]{\includegraphics[width=.35\linewidth]{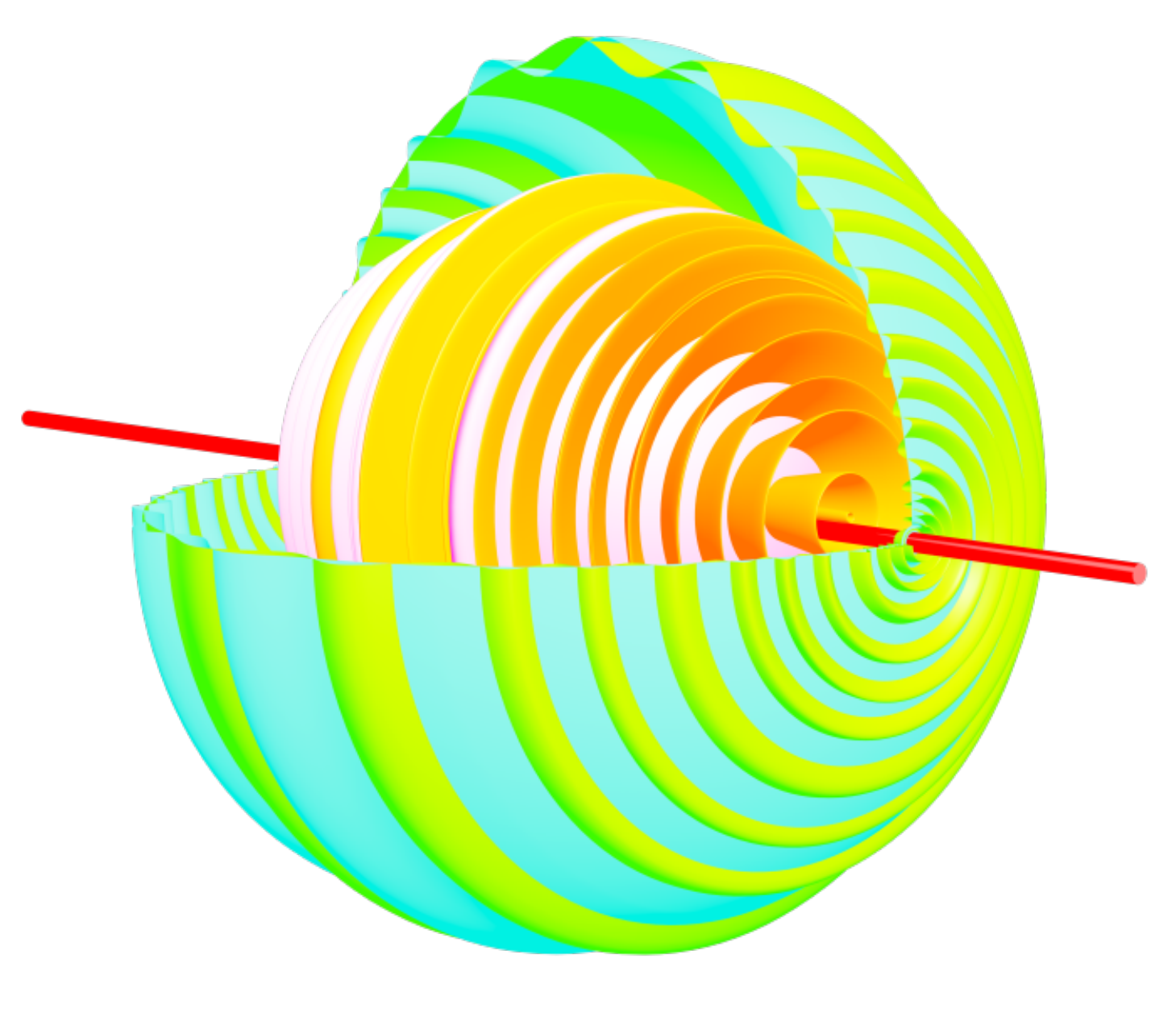}\label{th2}}
		\\
		\subfloat[][]{\includegraphics[width=.42\linewidth]{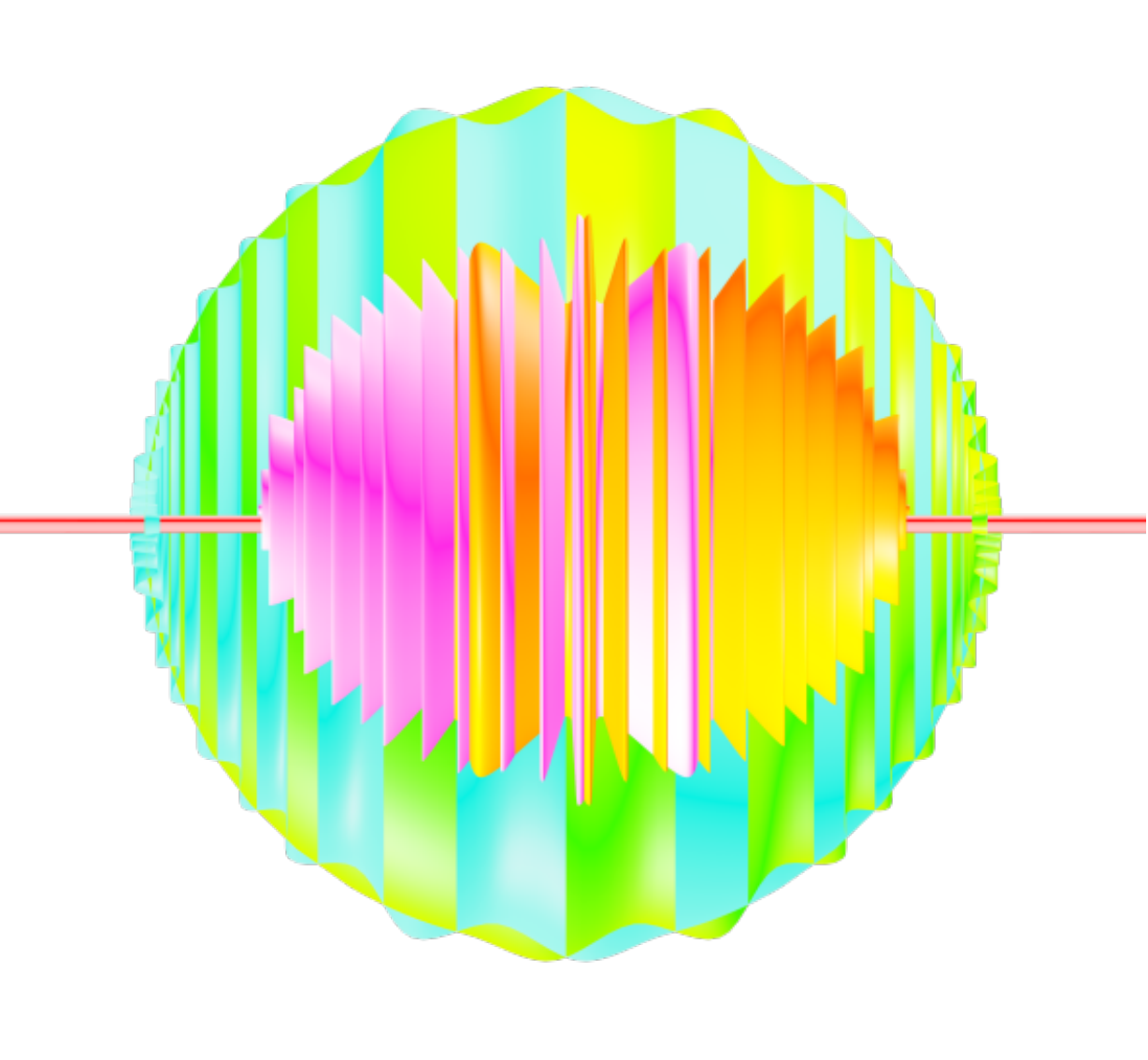}\label{th3}}
		\qquad
		\subfloat[][]{\includegraphics[width=.4\linewidth]{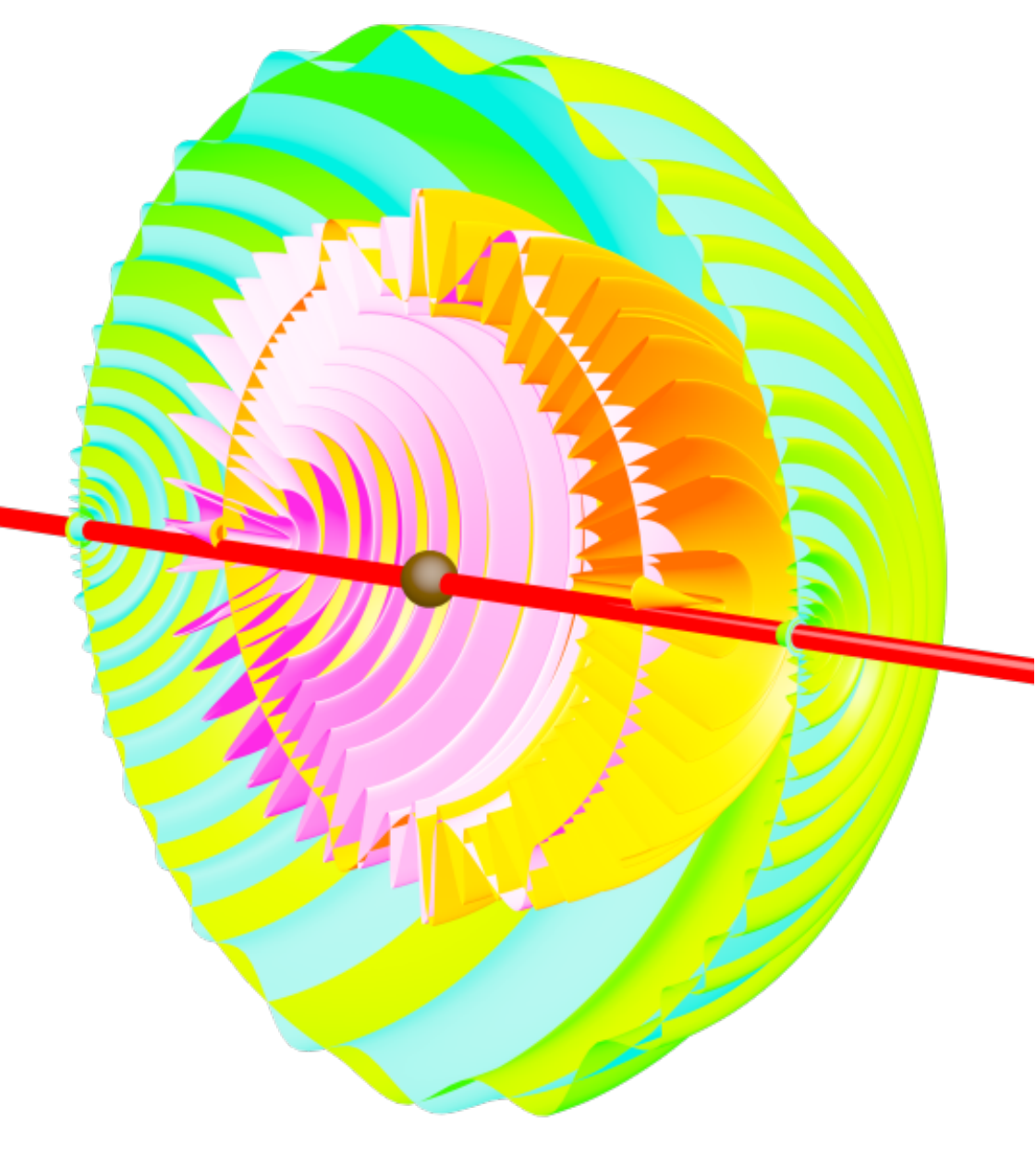}\label{th4}}
		\\
		\subfloat[][]{\includegraphics[width=.4\linewidth]{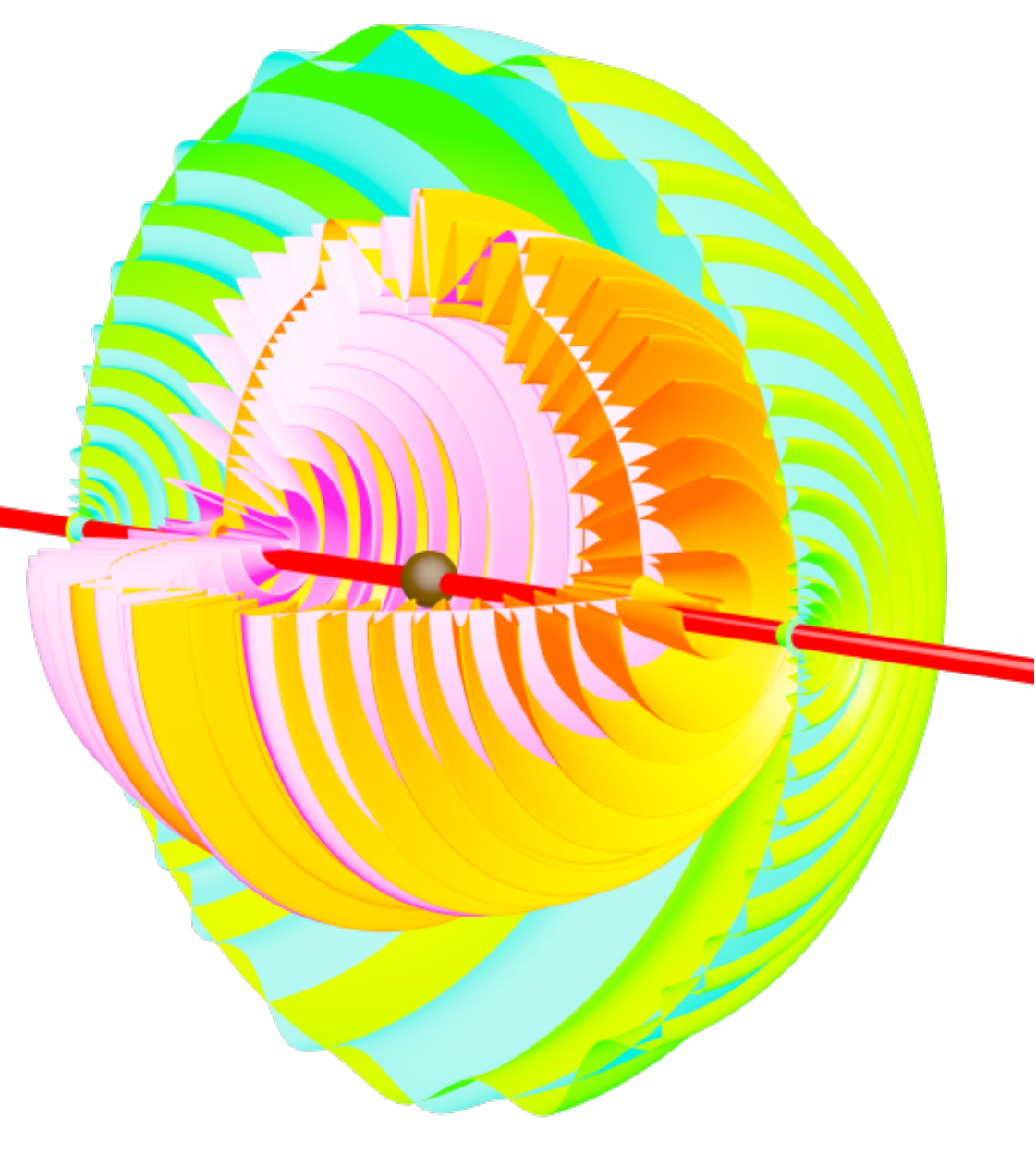}\label{th5}}
		\qquad
		\subfloat[][]{\includegraphics[width=.4\linewidth]{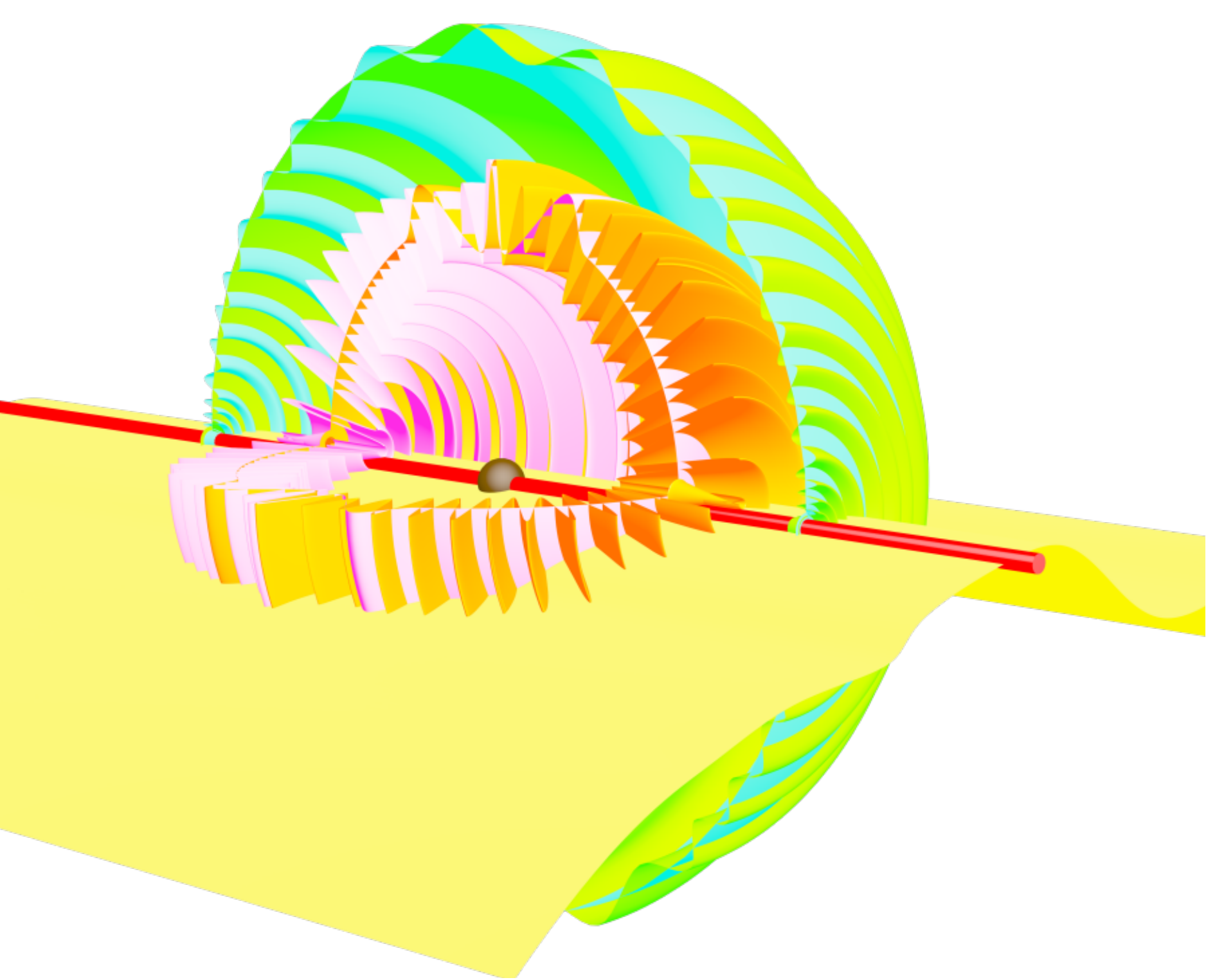}\label{th6}}
		\caption{The elliptic Hopf singularity associated to the elliptic Hopf parabolic point.}
		\label{hop2}
	\end{figure}

\begin{figure}[H]
	\captionsetup[subfigure]{width=.9\linewidth}
	\centering
	\subfloat[][]   {\includegraphics[width=.9\linewidth]{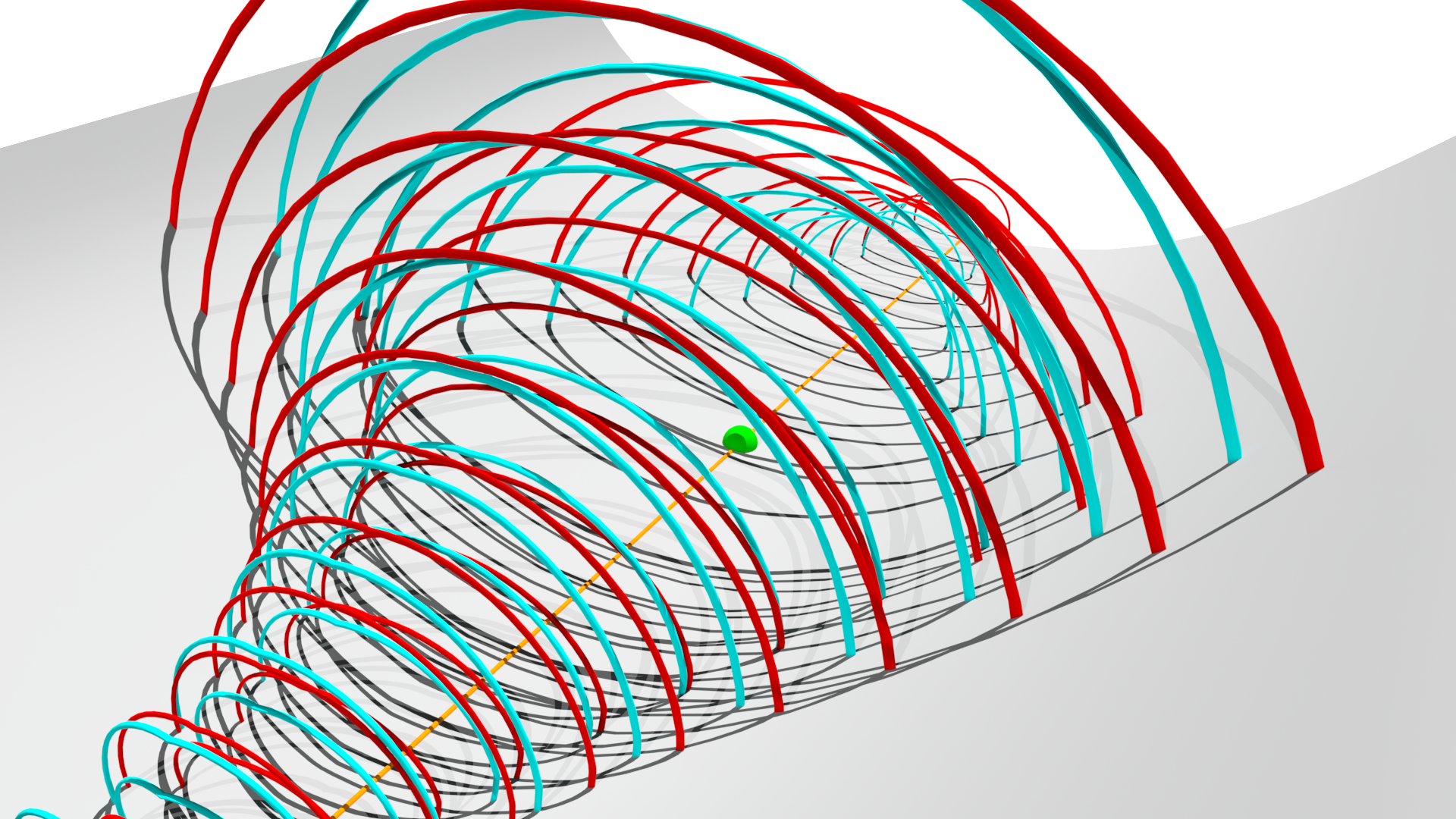}\label{lahe}}
	\caption{Projection by $\pi$ of the integral curves of 
	Fig. \ref{hop2} onto the asymptotic lines (coloured as red and blue), the criminant surface of Fig. \ref{hop2} onto the parabolic surface (coloured as gray) and the curve of singular points of Fig. \ref{hop2} onto the curve of special parabolic points (coloured as yellow). The Hopf point is the green point.}
	\label{lahop2}
\end{figure}
	
	\begin{proof}[Proof of Theorem \ref{aci}]
		By Propositions \ref{prop10} and \ref{propsing}, $b_{2}\neq0$, $a_{2}=-b_{1}$ and $a_{11}=0$. By Proposition \ref{prop2}, $F_{y}(0,0,0,0)=0$, which means that $a_{12}=-a_{3}b_{1}$. The Implicit Function Theorem implies that the Lie-Cartan hypersurface $\mathbb{L}$ (resp. the criminant surface $\mathbb{P}$ and the curve $\widetilde{\varphi}$) is locally given by $z=z(x,y,p)$ (resp. $p=p(x,y)$ and $x=x(y)$) where
		\begin{equation}
		\begin{split}
		z(x,y,p)&=\left(\frac{3a_{111}}{(a_{3})^2-a_{13}}\right)x^2+\left(\frac{(a_{3})^2b_{1}+a_{13}b_{1}+2b_{1}c_{11}+2a_{112}}{(a_{3})^2-a_{13}}\right)xy\\
		&+\left(\frac{a_{23}b_{1}-a_{22}a_{3}+b_{1}c_{12}+a_{122}}{(a_{3})^2-a_{13}}\right)y^2+\left(\frac{2b_{11}-2a_{3}b_{1}}{(a_{3})^2-a_{13}}\right)px\\
		&+\left(\frac{b_{12}-a_{3}b_{2}+b_{1}b_{3}+2a_{22}}{(a_{3})^2-a_{13}}\right)py+\left(\frac{b_{2}}{(a_{3})^2-a_{13}}\right)p^2+\mathcal{O}^3(x,y,p).
		\end{split}
		\end{equation}
		\begin{equation}
		p(x,y)=\left(\frac{a_{3}b_{1}-b_{11}}{b_{2}}\right)x+\left(\frac{a_{3}b_{2}-b_{1}b_{3}+2a_{22}+b_{12}}{b_{2}}\right)y+\mathcal{O}^{2}(x,y).
		\end{equation}
		\begin{equation}
		\begin{split}
		x(y)&=8\Bigg(\frac{
			(a_{3}b_{2}-b_{1}b_{3}-2a_{22}-b_{12})b_{11}+a_{3}(b_{1})^2b_{3}+2a_{112}b_{2}
		}{(\lambda_{1}-\lambda_{2})^2}\\
		&+\frac{(2a_{13}b_{2}-(a_{3})^2b_{2}+2a_{22}a_{3}+a_{3}b_{12}+2b_{2}c_{11})b_{1}}{(\lambda_{1}-\lambda_{2})^2}
		\Bigg)y+\mathcal{O}^2(y).
		\end{split}
		\end{equation}
		The pair of nonzero eigenvalues crosses the imaginary axis as we move along
		the curve of singular points $\widetilde{\varphi}$ in such way that the derivative of each
		eigenvalue in the direction of the tangent of $\varphi$ does not vanishes if
		\begin{equation}
		\delta=\left.\frac{\partial}{\partial y}\bigg(F_{y}\Big(\mu(y)\Big)+q_{p}\Big(\mu(y)\Big)F_{z}\Big(\mu(y)\Big)\bigg)\right|_{y=0}\neq0,
		\end{equation}
		\noindent where $\mu(y)=(x(y),x,z\big(x(y),x(y),p(x(y),x)\big),p(x(y),y))$.
		Let $\widetilde{\mathcal{X}}$ be the restriction of the Lie-Cartan vector field to the Lie-Cartan hypersurface,
		\begin{equation}
		\widetilde{\mathcal{X}}(x,y,p)=(\widetilde{\mathcal{X}}_{1},\widetilde{\mathcal{X}}_{2},\widetilde{\mathcal{X}}_{3})=\left(F_{p},pF_{p},-(F_{x}+pF_{y}+qF_{z})\right)(x,y,z(x,y,p),p).
		\end{equation}
		By \cite[Theorem 1.1]{MR1766554}, \cite[Theorem 1.1]{MR1780924}, \cite[Theorem 5.1]{MR3243511},
		the coordinates in the eigenspace of $D\widetilde{\mathcal{X}}$ associated with the pair of complex eigenvalues at $(0,0,0)$ are the $x$ and $p$ axis.
		It follows that $(\widetilde{\mathcal{X}}_{2})_{xx}(0,0,0)+(\widetilde{\mathcal{X}}_{2})_{pp}(0,0,0)=4b_{2}=\frac{\mathcal{H}(0,0,0)}{4}\neq0$.
		By \cite[Theorem 1.5]{MR1785113}, there exist $\varepsilon>0$ such that any solution $\widetilde{\gamma}(t)$
		witch stays in a neighbourhood of radios $\varepsilon$ of $(0,0,0,\widetilde{p})$ for all positive or negative times
		(possibly both) converges to a single singular point on the curve $\widetilde{\varphi}$, which locally is the $x$ axis.
		
		If $\delta>0$, then we have the hyperbolic case of the \cite[Theorem 1.5]{MR1785113}: all the nonequilibrium trajectories leave the
		neighborhood $U$ in positive or negative time directions (possibly both). The asymptotically
		stable and unstable sets of $(0,0,0,0)$ form the cone of the Figures \ref{thi2}, \ref{thi4}.
		Furthermore, the integral curves of the Lie-Cartan vector-field is tangent to a family of hyperboloids of one and two sheets, some are in Figure \ref{thi2} and \ref{thi3}.
		
		If $\delta<0$, then we have the elliptic case of the \cite[Theorem 1.5]{MR1785113}: all nonequilibrium trajectories starting sufficiently close
		to $(0,0,0,\widetilde{p})$ are heteroclinic between the singular points $\widetilde{\varphi}(x^{+})$ and $\widetilde{\varphi}(x^{-})$, witch are
		on opposite sides of $\widetilde{\varphi}(0)=(0,0,0,\widetilde{p})$. The two-dimensional strong stable and strong unstable manifolds of such
		singular point $\widetilde{\varphi}(x^{+})$, $\widetilde{\varphi}(x^{-})$ intersect at an angle with exponentially small upper bound
		in terms of $|s_{\pm}|$.
	\end{proof}

	\section{Generic parabolic points}
	\label{genp}
	
	Consider $C^{k}(\mathbb{R}^3,\mathbb{R}^3)$ with the Whitey $C^{k}$-topology, $k\geq3$, as defined in \cite[Section 5]{MR0209602}.
	
	\begin{defn}\label{p11}
		Define $P_{1}\subset C^{k}(\mathbb{R}^3,\mathbb{R}^3)$, $k\geq3$, as the set of the vector fields $\xi$ such that:
		\begin{itemize}
			\item The Gauss curvature $\mathcal{K}$ of $\Delta$ has the property that $\mathcal{K}$ and $d\mathcal{K}$ do not vanish simultaneously.
		\end{itemize}
	\end{defn}
	\begin{defn}
		Set $\phi=\langle \nabla \mathcal{K},\mathcal{A}\rangle$ and $\rho=\langle \mathcal{P},\varphi'\rangle$, where at $\mathbb{P}$,
		$\mathcal{P}$ is the principal direction associated with the principal curvature that equals zero, $\mathcal{A}$ is the asymptotic direction
		that coincides with $\mathcal{P}$ and $\mathcal{K}$ is the Gaussian curvature.
	\end{defn}
	\begin{defn}\label{p22}
		Define $P_{2}$ as the set of the vector fields $\xi$ such that:
		\begin{itemize}
			\item At $\mathbb{P}$, the function $\phi$ has the property that $\phi$ and $d\phi$ do not vanish simultaneously.
		\end{itemize}
	\end{defn}
	\begin{defn}\label{p33}
		Define $P_{3}$ as the set of the vector fields $\xi$ such that:
		\begin{itemize}
			\item At $\{\phi=0\}$, the function $\rho$ has the property that $\rho$ and $d\rho$ do not vanish simultaneously.
			\item At $\{\phi=0\}$, the function $(k_{1})_{v}$ has the property that $(k_{1})_{v}$ and $d((k_{1})_{v})$ do not vanish simultaneously.
		\end{itemize}
	\end{defn}
	\begin{defn}\label{p44}
		Define $P_{4}$ as the set of the vector fields $\xi$ such that:
		\begin{itemize}
			\item At $\{\phi=0\}$,  the function $\delta$ of Theorem \ref{aci} has the property that $\delta$ and $d\delta$ do not vanish simultaneously.
		\end{itemize}
	\end{defn}
	
	\begin{thm}\label{t11}
		The sets $P_{1}$, $P_{2}$, $P_{3}$ and $P_{4}$ are residual subsets of $C^{k}(\mathbb{R}^3,\mathbb{R}^3)$, $k\geq3$.
	\end{thm}
	
	\begin{rem}
		The definitions \ref{p11}, \ref{p22} and \ref{p33} are adaptations for plane fields of the definitions given in \cite{MR0487812}.
		The Theorem \ref{t11} is a partial version for plane fields of the \cite[Theorem 1.1]{MR0487812} and \cite[Theorem 3.3]{MR0206974}. The Theorem \ref{t11} was motivated by \cite{MR0209602}, \cite{MR0142859} and \cite{MR0487812}.
	\end{rem}
	
	\begin{defn}
		Set 
		\begin{itemize}
			\item $\mathbb{K}_{1}=\{j^r(\xi)(p)\in J^r(\mathbb{R}^3,\mathbb{R}^3); \mathcal{K}(p)=0 \ \ and \ \ d\mathcal{K}_{p}\neq0\}$,
			\item $\mathbb{K}_{2}=\{j^r(\xi)(p)\in \mathbb{K}_{1}; \phi(p)=\langle \nabla \mathcal{K}(p),\mathcal{A}(p)\rangle=0 \ \ and \ \ d\phi_{p}\neq0\}$,
			\item $\mathbb{K}_{3}=\{j^r(\xi)(p)\in \mathbb{K}_{2}; \rho(p)=\langle \mathcal{P}(p),\varphi'(p)\rangle=0 \ \ and \ \ d\rho_{p}\neq0\}$,
			\item $\mathbb{K}_{4}=\{j^r(\xi)(p)\in \mathbb{K}_{2}; (k_{1})_{v}(p)=0 \ \ and \ \  d((k_{1})_{v})_{p}\neq0\}$,
			\item $\mathbb{K}_{5}=\{j^r(\xi)(p)\in \mathbb{K}_{2}; \delta(p)=0 \ \ and \ \ d\delta_{p}\neq0\}$.
		\end{itemize}
	\end{defn}
	
	\begin{prop}
		$\mathbb{K}_{1}$ is a regular submanifold of codimension 1 in $J^r(\mathbb{R}^3,\mathbb{R}^3)$,
		$\mathbb{K}_{2}$ is a regular submanifold of codimension 2 in $J^r(\mathbb{R}^3,\mathbb{R}^3)$,
		$\mathbb{K}_{3}$, $\mathbb{K}_{4}$ and $\mathbb{K}_{5}$ are  regular submanifolds of codimension 3 in $J^r(\mathbb{R}^3,\mathbb{R}^3)$.
	\end{prop}
	
	\begin{proof}
		Let $\Psi_{1}:J^r(\mathbb{R}^3,\mathbb{R}^3)\rightarrow \mathbb{R}$, $\Psi_{1}(j^r(\xi)(p))=\mathcal{K}(p)$. Since $\mathcal{K}$ only depends of the derivatives 
		of $\xi$ of order 1, $\Psi_{1}$ is well defined. $(\Psi_{1})^{-1}(0)=\mathbb{K}_{1}$ 
		and since $\xi\in \mathbb{K}_{1}$,  $\Psi_{1}$ and $d \Psi_{1}$ does not vanish simultaneously.
		
		The proof for $\mathbb{K}_{i}\subset P_{i}$, $i=2,3,4,5$, is similar. Set 
		\begin{itemize}
			\item $\Psi_{2}:J^r(\mathbb{R}^3,\mathbb{R}^3)\rightarrow \mathbb{R}$, $\Psi_{2}(j^r(\xi)(p))=
			(\mathcal{K}(p),\phi(p))$,
			\item $\Psi_{3}:J^r(\mathbb{R}^3,\mathbb{R}^3)\rightarrow \mathbb{R}$, $\Psi_{3}(j^r(\xi)(p))=
			(\mathcal{K}(p),\phi(p),\rho(p))$,
			\item $\Psi_{4}:J^r(\mathbb{R}^3,\mathbb{R}^3)\rightarrow \mathbb{R}$, $\Psi_{4}(j^r(\xi)(p))=
			(\mathcal{K}(p),\phi(p),(k_{1})_{v})$,
			\item $\Psi_{5}:J^r(\mathbb{R}^3,\mathbb{R}^3)\rightarrow \mathbb{R}$, $\Psi_{5}(j^r(\xi)(p))=
			(\mathcal{K}(p),\phi(p),\delta(p))$.
		\end{itemize}
		Since $\phi$ (resp. $\rho$, $(k_{1})_{v}$, $\delta$) only depends of the derivatives 
		of $\xi$ of order 1, 2 (resp. order 1, 2, 3), the maps $\Psi_{i}$, $i=2,3,4,5$, are well defined.
	\end{proof}

	\begin{prop}\label{propGe1}
		If $j^r(\xi)\pitchfork \mathbb{K}_{1}$ then $(j^r)^{-1}(\mathbb{K}_{1})$ is a regular submanifold
		of codimension 1 of $\mathbb{R}^3$ and $(j^r)^{-1}(\mathbb{K}_{1})=\mathbb{P}$. 
		
		If $j^r(\xi)\pitchfork \mathbb{K}_{2}$ then $(j^r)^{-1}(\mathbb{K}_{2})$ is a regular submanifold
		of codimension 2 of $\mathbb{R}^3$ and $(j^r)^{-1}(\mathbb{K}_{2})=\{\phi=0\}$.
		
		If $j^r(\xi)\pitchfork \mathbb{K}_{3}$ then $(j^r)^{-1}(\mathbb{K}_{3})$ is a regular submanifold
		of codimension 3 of $\mathbb{R}^3$ and $(j^r)^{-1}(\mathbb{K}_{3})=\{\rho=0 \}
		\subset \{\phi=0\}$.
		
		If $j^r(\xi)\pitchfork \mathbb{K}_{4}$ then $(j^r)^{-1}(\mathbb{K}_{4})$ is a regular submanifold
		of codimension 3 of $\mathbb{R}^3$ and $(j^r)^{-1}(\mathbb{K}_{4})=\{(k_{1})_{v}=0\}
		\subset \{\phi=0\}$.
		
		If $j^r(\xi)\pitchfork \mathbb{K}_{5}$ then $(j^r)^{-1}(\mathbb{K}_{5})$ is a regular submanifold
		of codimension 3 of $\mathbb{R}^3$ and $(j^r)^{-1}(\mathbb{K}_{5})=\{\delta=0\}\subset \{\phi=0\}$.
	\end{prop}
	\begin{proof}
		It follows from \cite[Proposition 1, p.23]{10.1007/BFb0066810} 
		that each $(j^r)^{-1}(\mathbb{K}_{i})$ is a regular submanifold of $\mathbb{R}^3$ with the respective codimension given by the statement of the Proposition \ref{propGe1}.
		Now, $(j^r)^{-1}(\mathbb{K}_{1})=\{p\in\mathbb{R}^3;\mathcal{K}(p)=0 \}=\mathbb{P}$. 
		
		It is easy to check that $(j^r)^{-1}(\mathbb{K}_{2})=\{\phi=0\}$, $(j^r)^{-1}(\mathbb{K}_{3})=\{\rho=0 \}
		\subset \{\phi=0\}$, $(j^r)^{-1}(\mathbb{K}_{4})=\{(k_{1})_{v}=0\}
		\subset \{\phi=0\}$, $(j^r)^{-1}(\mathbb{K}_{5})=\{\delta=0\}\subset \{\phi=0\}$.
	\end{proof}
	
	\begin{defn}
		Set $\widetilde{P}_{i}=\{\xi\in C^r(\mathbb{R}^3,\mathbb{R}^3);j^r(\xi)\pitchfork \mathbb{K}_{i}\}$, $i=1,2,3,4,5$.
	\end{defn}
	
	\begin{proof}[Proof of Theorem \ref{t11}]
		By \cite[Transversality Theorem of Thom, p.32]{10.1007/BFb0066810}, \cite[Thom Transversality Theorem 4.9]{MR0341518}, the set $\widetilde{P}_{1}$ is residual in $C^k(\mathbb{R}^3,\mathbb{R}^3)$. 
		If $\xi\in \widetilde{P}_{1}$, 
		Proposition \ref{propGe1} shows  that $\mathcal{K}$ and $d\mathcal{K}$ does not vanish
		simultaneously and so $\xi\in P_{1}$. It follows that $\widetilde{P}_{1}\subset P_{1}$.
		The proof for $\widetilde{P}_{i}\subset P_{i}$, $i=2,3,4,5$, is similar.
	\end{proof}

	\begin{defn}
		Let $\mathcal{G}_{1}\subset C^{k}(\mathbb{R}^3,\mathbb{R}^3)$, $k\geq3$, be the set of vector fields $\xi$ such that
		the plane field $\Delta$ orthogonal to $\xi$ has the following properties:
		\begin{itemize}
			\item the parabolic set $\mathbb{P}$ of $\Delta$ is a regular surface (submanifolds of codimension 1),
			\item the set of parabolic points of type saddle, focus, node, saddle-node, node-focus and Hopf, is a regular curve (a submanifold of codimension 2),
			\item all the others parabolic points are of cusp type.
		\end{itemize}
	\end{defn}
	\begin{thm}\label{genpar}
		$\mathcal{G}_{1}$ is a residual subset of $C^{k}(\mathbb{R}^3,\mathbb{R}^3)$, $k\geq3$.
	\end{thm}
	\begin{proof}
		Let $\xi\in\bigcap\limits_{i=1}^{4} P_{i}$, which, by Theorem \mbox{\ref{t11}}, is a residual set. We will prove that $\xi\in\mathcal{G}_{1}$. Since $\xi\in P_{1}$, the parabolic set is a regular surface. Since $\xi\in P_{2}$,
		the set of singular points of the Lie-Cartan vector field is a regular curve $\widetilde{\varphi}$.
		Since $\xi\in P_{3}$,
		let $\varphi$ be a regular curve of special parabolic points such that $\varphi'(s_{0})$ belongs to the plane of $\Delta$ at $\varphi(s_{0})$. Let $\gamma$ be a integral curve of $\Delta$, parametrized by arc lenght $u$,
		such  that $\gamma(0)=\varphi(s_{0})$ and $\gamma'(0)=\varphi'(s_{0})$.
		Set $X(u)=\gamma'(u)$ and consider the Darboux frame defined in Section \ref{darbouxframe} with the equations \eqref{darboux}, \eqref{darboux2} and \eqref{darboux3}.
		Suppose that $k_{1}(u)=0$ end $k_{2}(u)<0$. By Theorems \ref{ecf} and \ref{proptau},
		\begin{equation}\label{g1kntg}
		k_{n}(u)=k_{2}(u)sin^2(\theta(u)), \ \ \tau_{g}(u)=\widetilde{\tau}_{g}(u)+k_{2}(u)cos(\theta(u))sin(\theta(u)).
		\end{equation}
		Since $\langle X(u),Y(u)\rangle=0$ for all $u$,
		then
		\begin{equation}\label{g1kntgY}
		k_{n,Y}(u)=k_{2}(u)cos^2(\theta(u)), \ \ \tau_{g,Y}(u)=\widetilde{\tau}_{g}(u)-k_{2}(u)cos(\theta(u))sin(\theta(u)).
		\end{equation}
		Let $\alpha:\Lambda\subset \mathbb{R}^3\rightarrow \mathbb{R}^3$ defined by
		$\alpha(u,v,w)=\gamma(u)+vY(u)+w(X\wedge Y)(u)$, where $\Lambda$ is a open subset.
		The map $\alpha$ is a tubular neighbourhood around $\gamma$.
		Since $d\alpha=\alpha_{u}du+\alpha_{v}dv+\alpha_{w}dw$, then
		\begin{equation}
		d\alpha=[(1-vk_{g}-wk_{n})du]X+[dv-w\tau_{g}du]Y+[v\tau_{g}du+dw]X\wedge Y.
		\end{equation}
		Since $\xi(u,0,0)=(X\wedge Y)(u)$, at the  tubular neighbourhood $\alpha$, the Taylor expansion of $\xi$ is given by
		\eqref{xic}.
		
		By equations \eqref{darboux}, \eqref{darboux2}, \eqref{darboux3} and equations \eqref{g1kntg}, \eqref{g1kntgY}, we have that
		\begin{equation}
		\begin{split}
		\xi&=\Bigg((\widetilde{\tau}_{g}-k_{2}cos(\theta)sin(\theta))v+l_{1}w+\frac{m_{11}v^2}{2}+m_{21}vw+\frac{m_{31}w^2}{2}\\
		&+\frac{n_{11}v^3}{6}+\frac{n_{21}v^2w}{2}+\frac{n_{31}vw^2}{2}+\frac{n_{41}w^3}{6}+\mathcal{O}^4_{u}(v,w)\Bigg)X\\
		&+\Bigg(-k_{2}cos^2(\theta)v+l_{12}w+\frac{v^2m_{12}}{2}+vwm_{22}+\frac{w^2m_{32}}{2}\\
		&+\frac{n_{12}v^3}{6}+\frac{n_{22}v^2w}{2}+\frac{n_{32}vw^2}{2}+\frac{n_{42}w^3}{6}+\mathcal{O}^4_{u}(v,w)\Bigg)Y\\
		&+\Bigg(1+\frac{v^2m_{13}}{2}+vwm_{23}+\frac{w^2m_{33}}{2}\\
		&+\frac{n_{13}v^3}{6}+\frac{n_{23}v^2w}{2}+\frac{n_{33}vw^2}{2}+\frac{n_{43}w^3}{6}+\mathcal{O}^4_{u}(v,w))\Bigg)X\wedge Y.
		\end{split}
		\end{equation}
		We have that $d\xi=\xi_{u}du+\xi_{v}dv+\xi_{w}dw$.
		The implicit differential equations  \eqref{eqla} of the asymptotic lines are given by $\langle\xi,d\alpha\rangle=0$ and $\langle d\xi,d\alpha\rangle=0$.
		We can solve the equation $\langle\xi,d\alpha\rangle=0$ for $dw$ and substitute it in $\langle d\xi,d\alpha\rangle=0$. Then the equations of the asymptotic lines
		becomes $dw=Adu+Bdv$ and $edu^2+2fdudv+gdv^2=0$, where
		\begin{equation}
		\begin{split}
		&A=-2\widetilde{\tau}_{g}v-l_{1}w+\bigg([\widetilde{\tau}_{g}-k_{2}cos(\theta)sin(\theta)]k_{g}-\frac{m_{11}}{2}\bigg)v^2
		+\bigg(k_{g}l_{1}-\frac{m_{31}}{2}\\
		&-(\widetilde{\tau}_{g}+k_{2}cos(\theta)sin(\theta))k_{2}cos^2(\theta)
		+(\widetilde{\tau}_{g}-k_{2}cos(\theta)sin(\theta))k_{2}sin^2(\theta)-m_{21}\bigg)vw\\
		&+\bigg(k_{2}l_{1}sin^2(\theta)
		+(\widetilde{\tau}_{g}+k_{2}cos(\theta)sin(\theta))l_{2}\bigg)w^2+\mathcal{O}_{u}^3(v,w),
		\end{split}
		\end{equation}
		\begin{equation}
		B=k_{2}cos^2(\theta)v-l_{2}w-\frac{m_{12}v^2}{2}-m_{22}vw-\frac{m_{32}w^2}{2}+\mathcal{O}_{u}^3(v,w),
		\end{equation}
		
		\begin{equation}
		\begin{split}
		&e=-k_{2}sin^2(\theta)+(\widetilde{\tau}_{g}'-2k_{2}\theta'cos^2(\theta)-k_{2}'cos(\theta)sin(\theta)+k_{g}k_{2}-2\widetilde{\tau}_{g}l_{1}+k_{2}\theta')v\\
		&+(2\widetilde{\tau}_{g}k_{2}cos(\theta)sin(\theta)-(k_{2})^2cos^2(\theta)-k_{g}l_{2}+(\widetilde{\tau}_{g})^2+(k_{2})^2-(l_{1})^2+l_{1}')w\\
		&+\bigg((k_{2})^3cos^4(\theta)-(k_{2})^3cos^2(\theta)-2\widetilde{\tau}_{g}(k_{2})^2cos^3(\theta)sin(\theta)+2\widetilde{\tau}_{g}(k_{2})^2cos(\theta)sin(\theta)\\
		&-(k_{g})^2k_{2}\theta'cos^2(\theta)+2k_{g}k_{2}cos^2(\theta)-k_{g}k_{2}l_{1}cos(\theta)sin(\theta)-k_{g}k_{2}\theta'\\
		&+2(\widetilde{\tau}_{g})^2k_{2}cos^2(\theta(u))-(\widetilde{\tau}_{g})^2k_{2}+k_{g}k_{2}'cos(\theta)sin(\theta)+3k_{g}\widetilde{\tau}_{g}l_{1}-2\widetilde{\tau}_{g}m_{21}-k_{g}\widetilde{\tau}_{g}'\\
		&+\frac{1}{2}[k_{2}m_{13}cos^2(\theta)-k_{g}m_{12}-k_{2}m_{13}-l_{1}m_{11}+m_{11}']\bigg)v^2\\
		&+\bigg(-\widetilde{\tau}_{g}k_{2}cos^2(\theta(u))l_{1}+k_{2}'\widetilde{\tau}_{g}^2cos(\theta)-k_{g}(\widetilde{\tau}_{g})^2+(\widetilde{\tau}_{g})^2l_{2}
		-2m_{31}\widetilde{\tau}_{g}-k_{g}l_{1}'\\
		&+(k_{g})^2l_{2}+2k_{g}(l_{1})^2-2l_{1}m_{21}+\widetilde{\tau}_{g}k_{2}l_{1}-2\widetilde{\tau}_{g}k_{2}\theta'cos(\theta)sin(\theta)\\
		&+2\widetilde{\tau}_{g}k_{2}l_{2}cos(\theta)sin(\theta)-(k_{2})^2l_{2}cos^4(\theta)+(k_{2})^2l_{2}cos^2(\theta)+(k_{2})^2\theta'cos^2(\theta)\\
		&+k_{2}m_{23}cos^2(\theta)+k_{2}\widetilde{\tau}_{g}'cos^2(\theta)-(k_{2})^2l_{1}cos^3(\theta)sin(\theta)-k_{g}m_{22}+m_{21}'\\
		&+k_{2}k_{2}'cos(\theta)sin(\theta)
		+(k_{2})^2l_{1}cos(\theta)sin(\theta)-k_{2}\widetilde{\tau}_{g}'-k_{2}m_{23}-(k_{2})^2\theta'\bigg)vw\\
		&+\bigg(k_{g}k_{2}l_{2}sin^2(\theta)-[\widetilde{\tau}_{g}+k_{2}cos(\theta)sin(\theta)]l_{2}'-k_{2}l_{1}'sin^2(\theta)+k_{2}(l_{1})^2sin^2(\theta)\\
		&-[\widetilde{\tau}_{g}+k_{2}cos(\theta)sin(\theta)]k_{g}l_{1}+[\widetilde{\tau}_{g}+k_{2}cos(\theta)sin(\theta)]l_{1}l_{2}
		+\frac{1}{2}[k_{2}m_{33}cos^2(\theta)\\
		&-k_{2}m_{33}-k_{g}m_{32}-3l_{1}m_{31}+m_{31}'\bigg)w^2+\mathcal{O}_{u}^3(v,w),
		\end{split}
		\end{equation}
		\begin{equation}
		\begin{split}
		&f=-k_{2}cos(\theta)sin(\theta)+\frac{1}{2}(k_{2}l_{1}cos^2(\theta)+2k_{2}\theta'cos(\theta)sin(\theta)-k_{2}'cos^2(\theta)-2\widetilde{\tau}_{g}l_{2}\\
		&+m_{11})v+\frac{1}{2}(2k_{2}\widetilde{\tau}_{g}cos^2(\theta)+(k_{2})^2cos(\theta)sin(\theta)+k_{g}l_{1}-k_{2}\widetilde{\tau}_{g}-2l_{1}l_{2}+l_{2}'+m_{21})w\\
		&+\frac{1}{4}\Big(2\widetilde{\tau}_{g}(k_{2})^2cos^2(\theta)-4\widetilde{\tau}_{g}(k_{2})^2cos^4(\theta)-2(k_{2})^3cos^3(\theta)sin(\theta)-2k_{2}k_{g}l_{1}cos^2(\theta)\\
		&-2k_{g}k_{2}l_{2}cos(\theta)sin(\theta)
		+k_{2}m_{13}cos(\theta)sin(\theta)+2k_{2}m_{21}cos^2(\theta)
		+2k_{g}\widetilde{\tau}_{g}l_{2}-l_{1}m_{12}\\
		&-4\widetilde{\tau}_{g}m_{22}-3\widetilde{\tau}_{g}m_{13}-k_{g}m_{11}-l_{2}m_{11}+n_{11}+m_{12}'\Big)v^2
		+\frac{1}{2}\Big(k_{2}m_{11}cos^2(\theta)\\
		&-k_{2}m_{12}cos(\theta)sin(\theta)+k_{2}m_{31}cos^2(\theta)+2k_{g}l_{1}l_{2}
		-2\widetilde{\tau}_{g}m_{23}-\widetilde{\tau}_{g}m_{12}-2l_{1}m_{22}\\
		&-l_{1}m_{13}-2\widetilde{\tau}_{g}m_{32}-k_{2}m_{11}-2l_{2}m_{21}+m_{22}'+n_{21}\Big)vw
		+\frac{1}{4}\Big(2k_{2}(l_{2})^2cos(\theta)sin(\theta)\\
		&-2k_{2}l_{1}l_{2}cos^2(\theta)-k_{2}m_{33}cos(\theta)sin(\theta)
		-2k_{2}m_{22}cos(\theta)sin(\theta)+2k_{2}m_{21}cos^2(\theta)\\
		&+2\widetilde{\tau}_{g}(l_{2})^2+2k_{2}l_{1}l_{2}-2l_{1}m_{23}
		-\widetilde{\tau}_{g}m_{33}-2\widetilde{\tau}_{g}m_{22}-3l_{1}m_{32}-2k_{2}m_{21}+k_{g}m_{31}\\
		&-3l_{2}m_{31}+n_{31}+m_{32}'\Big)w^2+\mathcal{O}_{u}^3(v,w)
		\end{split}
		\end{equation}
		\begin{equation}
		\begin{split}
		&g=-k_{2}cos^2(\theta)+(k_{2}l_{2}cos^2(\theta)+m_{12})v+(m_{22}-(l_{2})^2)w+\frac{1}{2}[2k_{2}m_{13}cos^2(\theta)\\
		&+n_{12}+2k_{2}m_{22}cos^2(\theta)-l_{2}m_{12}]v^2
		+[k_{2}m_{23}cos^2(\theta)+k_{2}m_{32}cos^2(\theta)-2l_{2}m_{22}\\
		&-l_{2}m_{13}+n_{22}]vw+\frac{1}{2}[n_{32}-2l_{2}m_{23}-3l_{2}m_{32}]w^2+\mathcal{O}_{u}^3(v,w).
		\end{split}
		\end{equation}
		
		The derivatives $(\mathcal{K})_{v}$ and $(\mathcal{K})_{w}$ evaluated at $(u,0,0)$ are given by
		\begin{equation}\label{Kv0}
		\begin{split}
		&(\mathcal{K})_{v}=(k_{1})_{v}k_{2}=[2\widetilde{\tau}_{g}l_{1}cos^2(\theta)-k_{g}k_{2}cos^2(\theta)sin^2(\theta)\\
		&-k_{g}k_{2}cos^4(\theta)-\widetilde{\tau}_{g}'cos^2(\theta) +k_{2}\theta'cos^2(\theta)sin^2(\theta)+k_{2}\theta'cos^4(\theta)\\
		&-m_{12}sin^2(\theta)-k_{2}l_{2}cos^2(\theta)sin^2(\theta)+m_{11}cos(\theta)sin(\theta)\\
		&-2\widetilde{\tau}_{g}l_{2}cos(\theta)sin(\theta)+k_{2}l_{1}cos^3(\theta)sin(\theta)]k_{2},
		\end{split}
		\end{equation}
		\noindent and
		\begin{equation}\label{Kw0}
		\begin{split}
		&(\mathcal{K})_{w}=(k_{1})_{w}k_{2}=[(l_{1})^2cos^2(\theta)-(\widetilde{\tau}_{g})^2cos^2(\theta)\\
		&-k_{2}\widetilde{\tau}_{g}cos^3(\theta)sin(\theta)
		+k_{g}l_{2}cos^2(\theta)-l_{1}'cos^2(\theta)-m_{22}sin^2(\theta)\\
		&+(l_{2})^2sin^2(\theta)+m_{21}cos(\theta)sin(\theta)-2l_{1}l_{2}cos(\theta)sin(\theta)\\
		&-k_{2}\widetilde{\tau}_{g}cos(\theta)sin^3(\theta)+k_{g}l_{1}cos(\theta)sin(\theta)+l_{2}'cos(\theta)sin(\theta)]k_{2}.
		\end{split}
		\end{equation}
		Set $F=e+2fp+gp^2$ and $q=-\left(\frac{a}{c}\right)-\left(\frac{b}{c}\right)p$.
		Let $\mathcal{X}$ be the Lie-Cartan vector field $\mathcal{X}=(F_{p},pF_{p},qF_{p},-(F_{u}+pF_{v}+qF_{w}))$.
		Then
		\begin{equation}
		\mathcal{X}\left(u,0,0,-\frac{sin(\theta(u))}{cos(\theta(u))}\right)=\left(0,0,0,-\frac{(k_{1})_{v}(u)sin(\theta(u))}{cos^3(\theta(u))}\right).
		\end{equation}
		
		It follows that $\left(u_{0},0,0,-\frac{sin(\theta(u_{0}))}{cos(\theta(u_{0}))}\right)$ is a singular point of $\mathcal{X}$ in $F^{-1}(0)$ if, and only if,
		$\theta(u_{0})=0$ or $(k_{1})_{v}(u_{0})=0$.
		
		Let $\lambda_{1}(u)$ and $\lambda_{2}(u)$ be the not necessarily zero eigenvalues of
		\begin{equation}
		D\mathcal{X}\left(u,0,0,-\frac{sin(\theta(u))}{cos(\theta(u))}\right).
		\end{equation}
		Then $(k_{1})_{v}(u)=(\lambda_{1}(u)+\lambda_{2}(u))cos^2(\theta(u))$.
		If $\theta(u_{0})=0$, then $\varphi'(s_{0})$ is a principal direction and $\lambda_{1}(u_{0})\lambda_{2}(u_{0})=0$. It follows that $\rho=0$ at $(u_{0},0,0)$ and that $(u_{0},0,0,0)$ is a singular point which is a transition of type saddle-node.
		
		
		If $(k_{1})_{v}(u_{0})=0$, then $\lambda_{1}(u_{0})+\lambda_{2}(u_{0})=0$, and so $\left(u_{0},0,0,-\frac{sin(\theta(u_{0}))}{cos(\theta(u_{0}))}\right)$
		is a singular point of type saddle, node or a transition that occurs when the pair of complex eigenvalues crosses the imaginary axis.
		
		Since $(k_{1})_{v}$ and $\rho$ has the property that each one do not vanishes simultaneously with the respectively derivative, then
		the transitions above occurs at isolated points of $\varphi$.
		
		Since $\xi\in P_{4}$, the parabolic points are of type saddle, focus, node, saddle-node, node-focus and Hopf. We conclude that
		$\bigcap\limits_{i=1}^{4} P_{i}\subset \mathcal{G}_{1}$. Since $\mathcal{G}_{1}$ contains a residual set, then $\mathcal{G}_{1}$ is a residual subset of $C^{k}(\mathbb{R}^3,\mathbb{R}^3)$.
	\end{proof}
	
	\section{Examples}
	
		\begin{prop}\cite[p. 23]{MR1749926}\label{propEx1}
		Let $\xi(x,y,z)=(f(x,y,z),g(x,y,z),1)$. Then
		\begin{equation}
			\mathcal{K}=(f_{x}-ff_{z})(g_{y}-gg_{z})-\frac{(f_{y}-fg_{z}+g_{x}-gf_{z})^2}{4}=\langle\wp,\xi\rangle-\frac{\langle curl(\xi),\xi\rangle}{4}^2,
		\end{equation}
		\noindent where $\wp=(det(\xi_{y},\xi_{z},\xi),det(\xi_{z},\xi_{x},\xi),det(\xi_{x},\xi_{y},\xi))$.
	\end{prop}
	
	\subsection{Generic parabolic points}
	
	\begin{ex}
		Cuspidal parabolic point. Consider the vector field
		$\xi(x,y,z)=(x^2-y,x+y,1)$. The equation of the plane field $\langle \xi,dr \rangle=0$ is given by $(x^2-y)dx+(x+y)dy+dz=0$,
		see the Figure \ref{cusp_cp}. The parabolic surface of the plane field is given by $x=0$. All points are of the cuspidal type.
		\begin{figure}[H]
			\captionsetup[subfigure]{width=.3\linewidth}
			\centering
			\subfloat{\includegraphics[width=.5\textwidth]{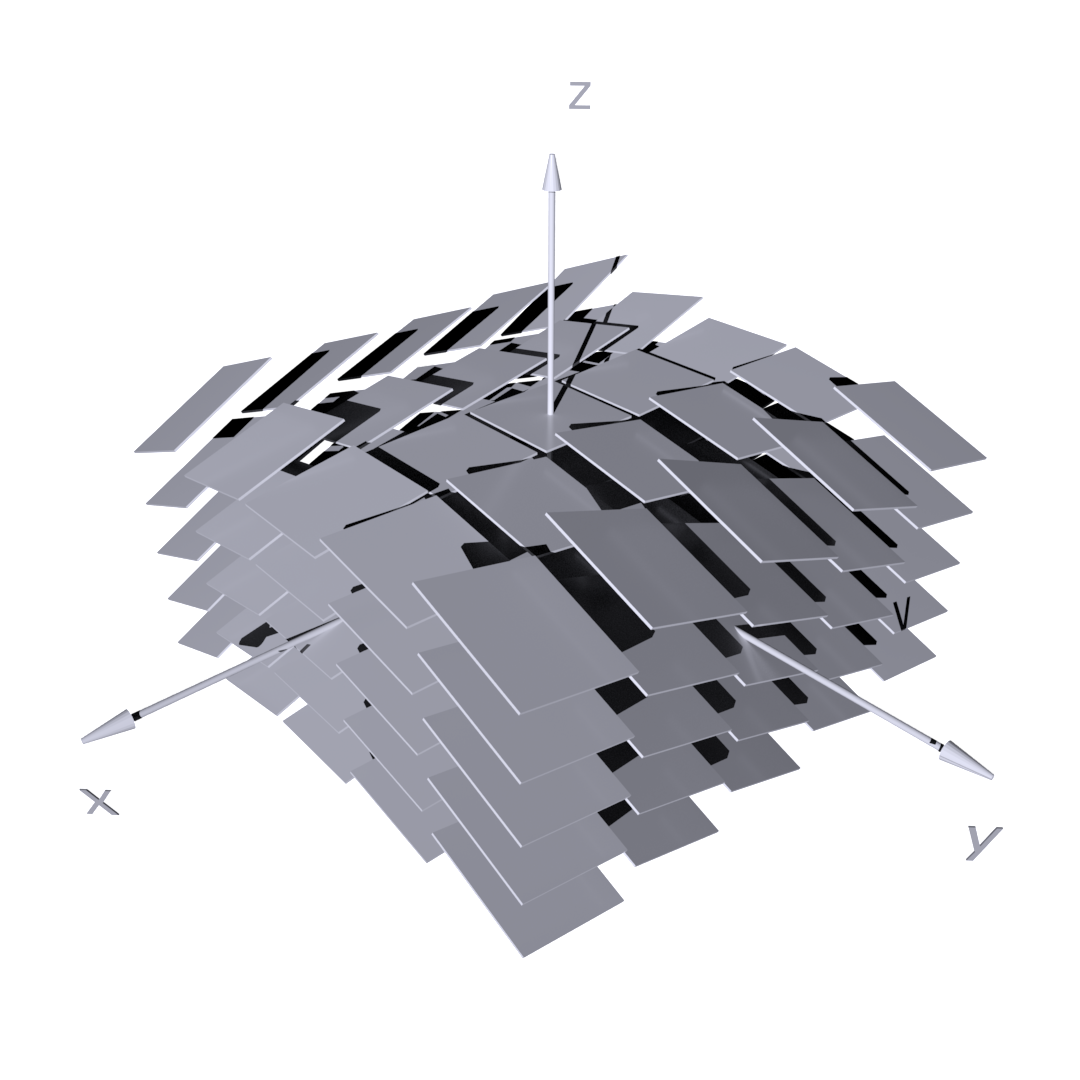}\label{cusp_cp1}}
			\subfloat{\includegraphics[width=.5\textwidth]{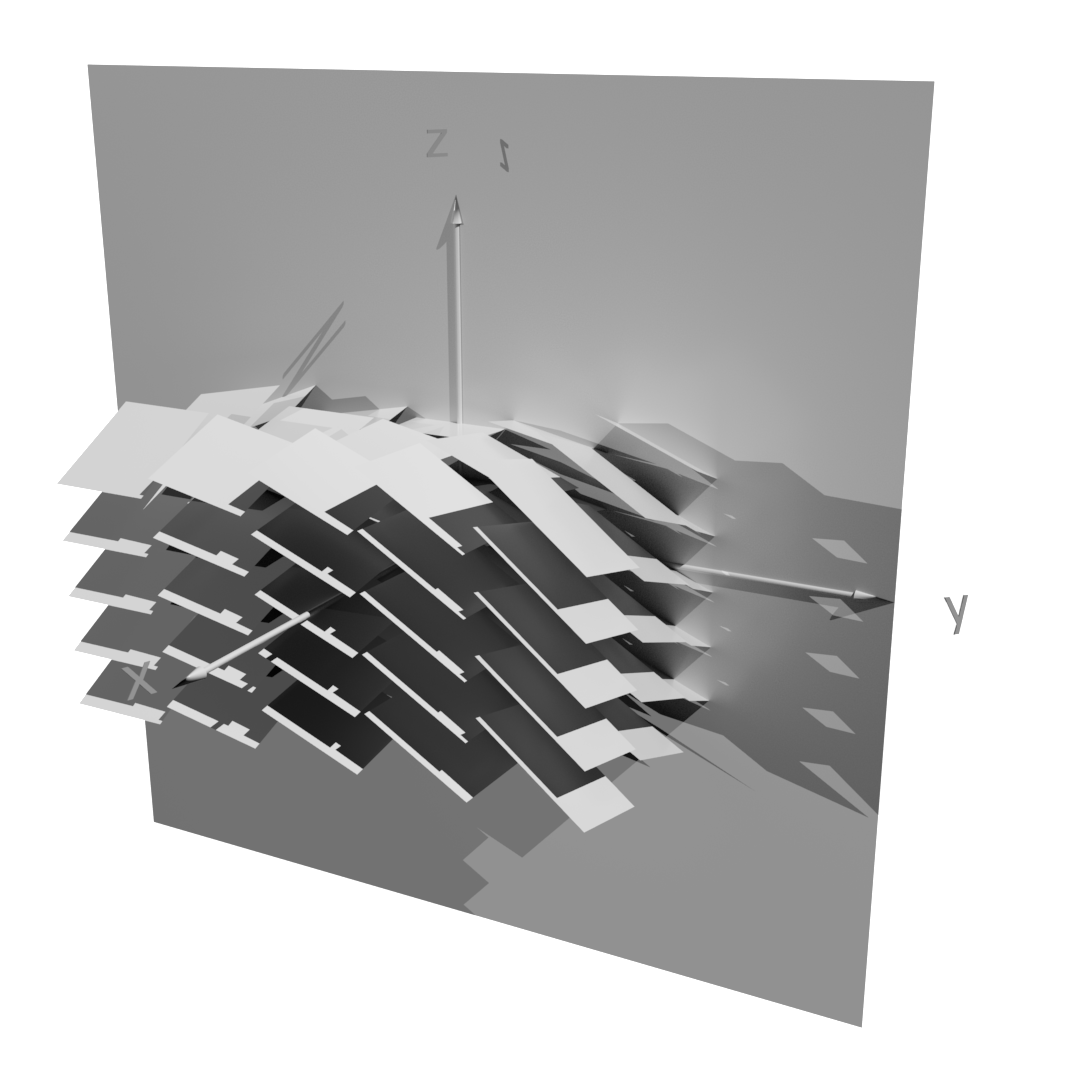}\label{cusp_cp2}}
			\caption{Plane field $(x^2-y)dx+ydy+dz=0$. Parabolic surface $x=0$.}
			\label{cusp_cp}
		\end{figure}
	\end{ex}
	
	\begin{ex} Parabolic point of saddle type, node type and focus type. Consider the vector field
		 $\xi(x,y,z)=(xy-y,x+y,1)$. The equation of the plane field $\langle \xi,dr \rangle=0$ is given by $(xy-y)dx+(x+y)dy+dz=0$,
		 see the Figure \ref{sela_cp}. The parabolic surface of the plane field is given by $4y-x^2=0$. The curve $x=y=0$, i.e, the $z$ axis, is 
		 a curve of parabolic points of saddle type. 
		\begin{figure}[H]
			\captionsetup[subfigure]{width=.3\linewidth}
			\centering
			\subfloat{\includegraphics[width=.5\textwidth]{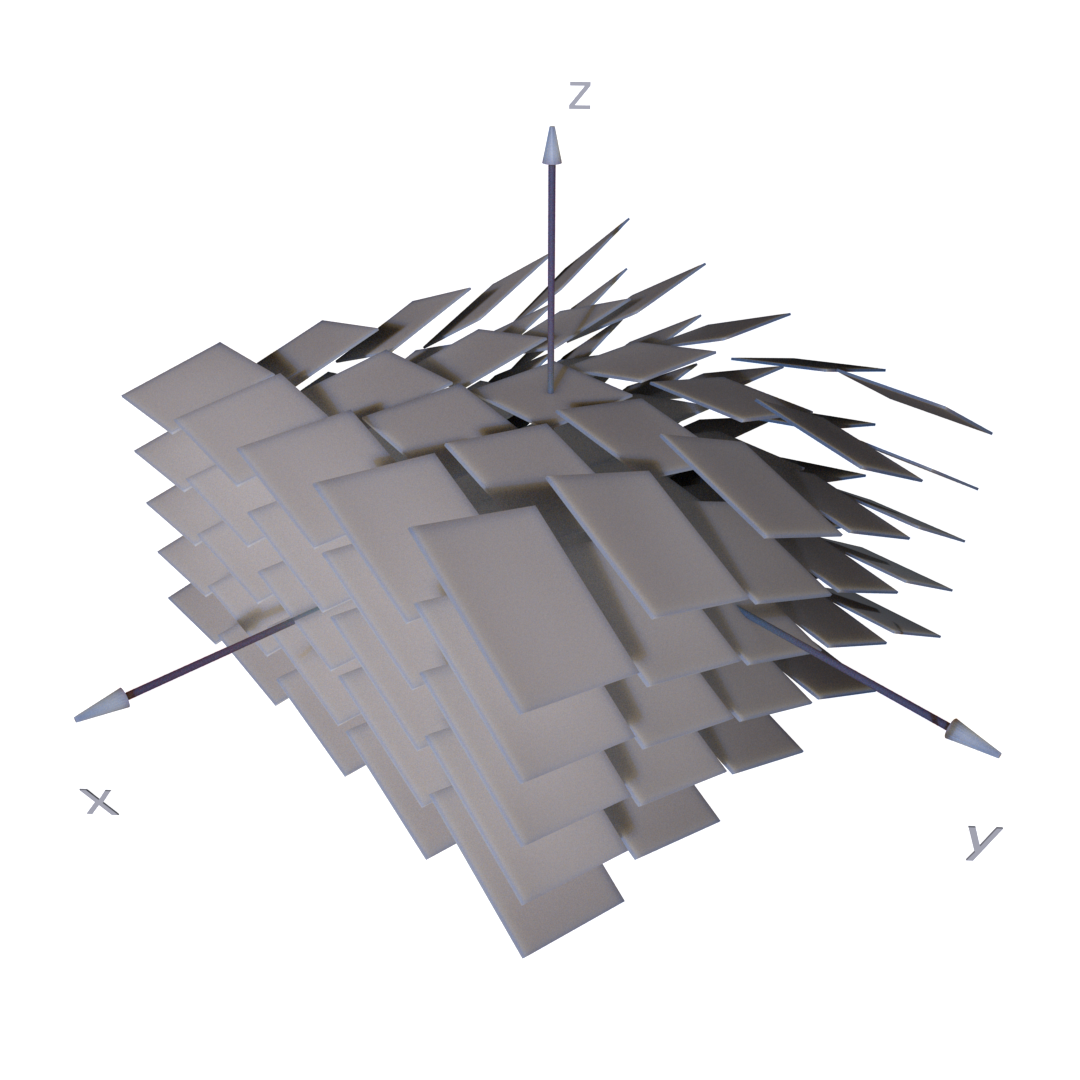}\label{sela_cp1}}
			\subfloat{\includegraphics[width=.5\textwidth]{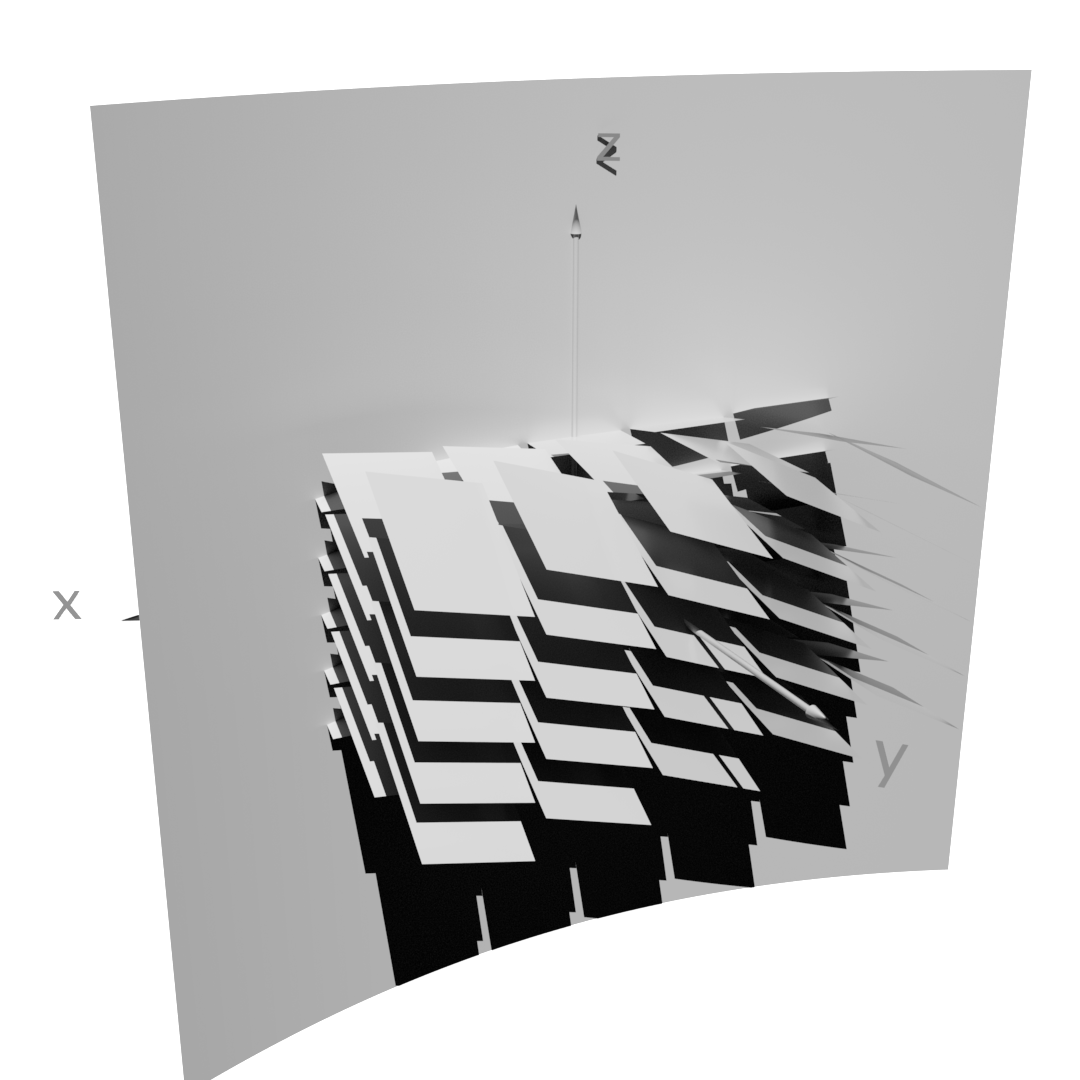}\label{sela_cp2}}
			\\
			\subfloat{\includegraphics[width=.5\textwidth]{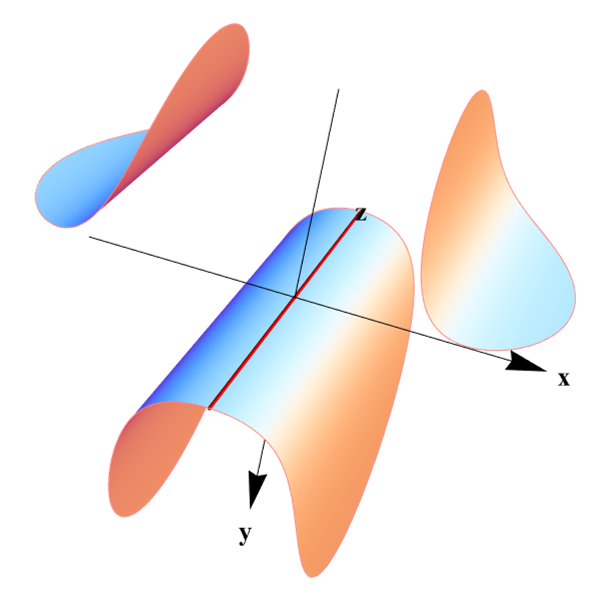}\label{sela_cp3}}
			\caption{Plane field $(xy-y)dx+(x+y)dy+dz=0$, parabolic surface and the curve $\varphi$ (coloured as red) of parabolic points of saddle type.}
			\label{sela_cp}
		\end{figure}
	
	Consider the vector field
	$\xi(x,y,z)=(xy-y+5z,x^2+x+y,1)$. The equation of the plane field $\langle \xi,dr \rangle=0$ is given by $(xy-y+5z)dx+(x^2+x+y)dy+dz=0$,
	see the Figure \ref{cpnode}. The parabolic surface of the plane field is given by $-5xy + 6y - 25z - (-(5/2)x^2 - x - (5/2)y)^2=0$. The curve 
	$-5xy + 6y - 25z - (-(5/2)x^2 - x - (5/2)y)^2=0$, $5x^2 + 2x + 5y=0$ is 
	a curve $\varphi$ of parabolic points of node type, $\varphi(x)=
	\left(x,-x^2 - (2/5x),((1/5)x^3 - (4/25)x^2 - (12/125x)\right)$. 
	\begin{figure}[H]
			\captionsetup[subfigure]{width=.6\linewidth}
			\centering
			\subfloat{\includegraphics[width=.5\textwidth]{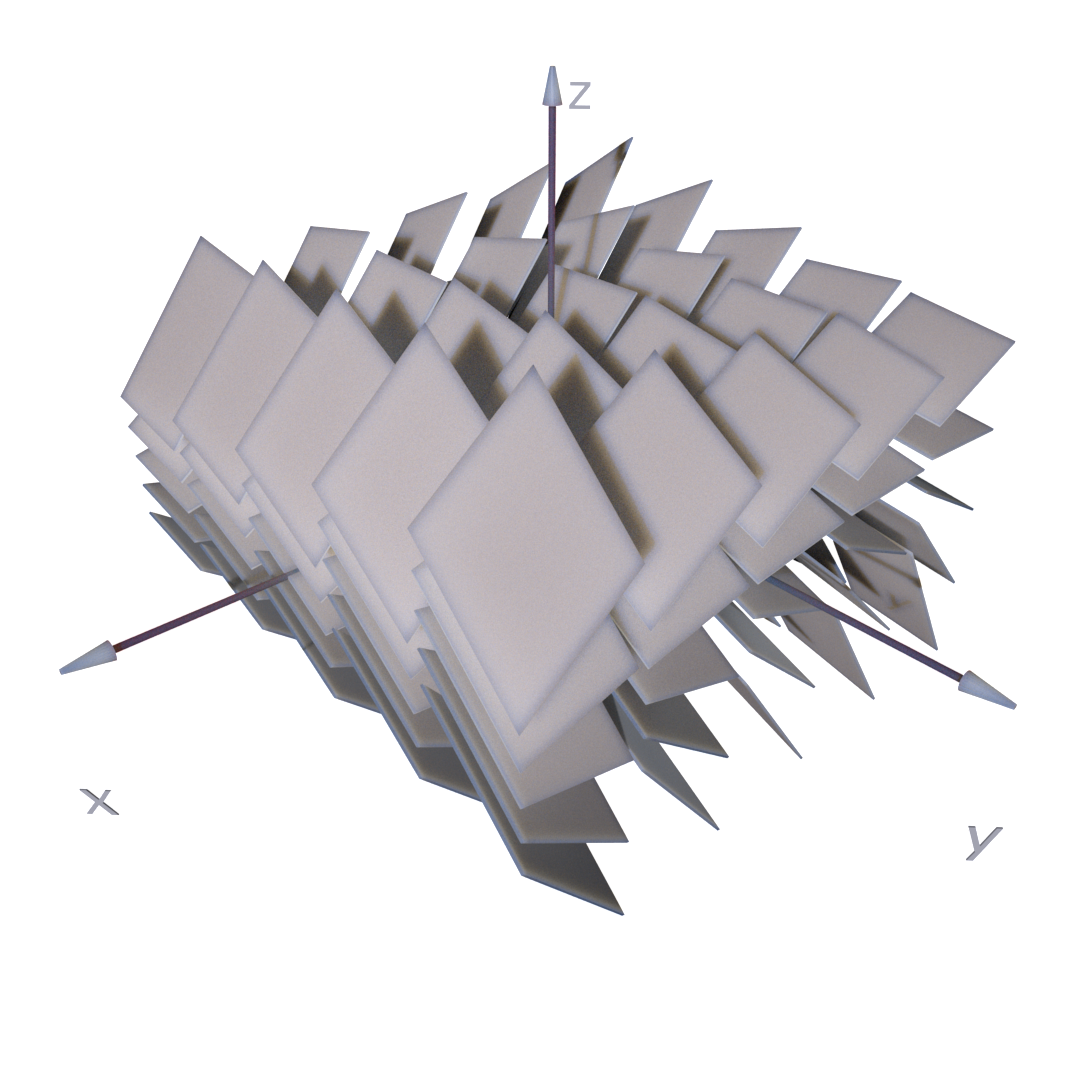}\label{cpnode1}}
			\subfloat{\includegraphics[width=.5\textwidth]{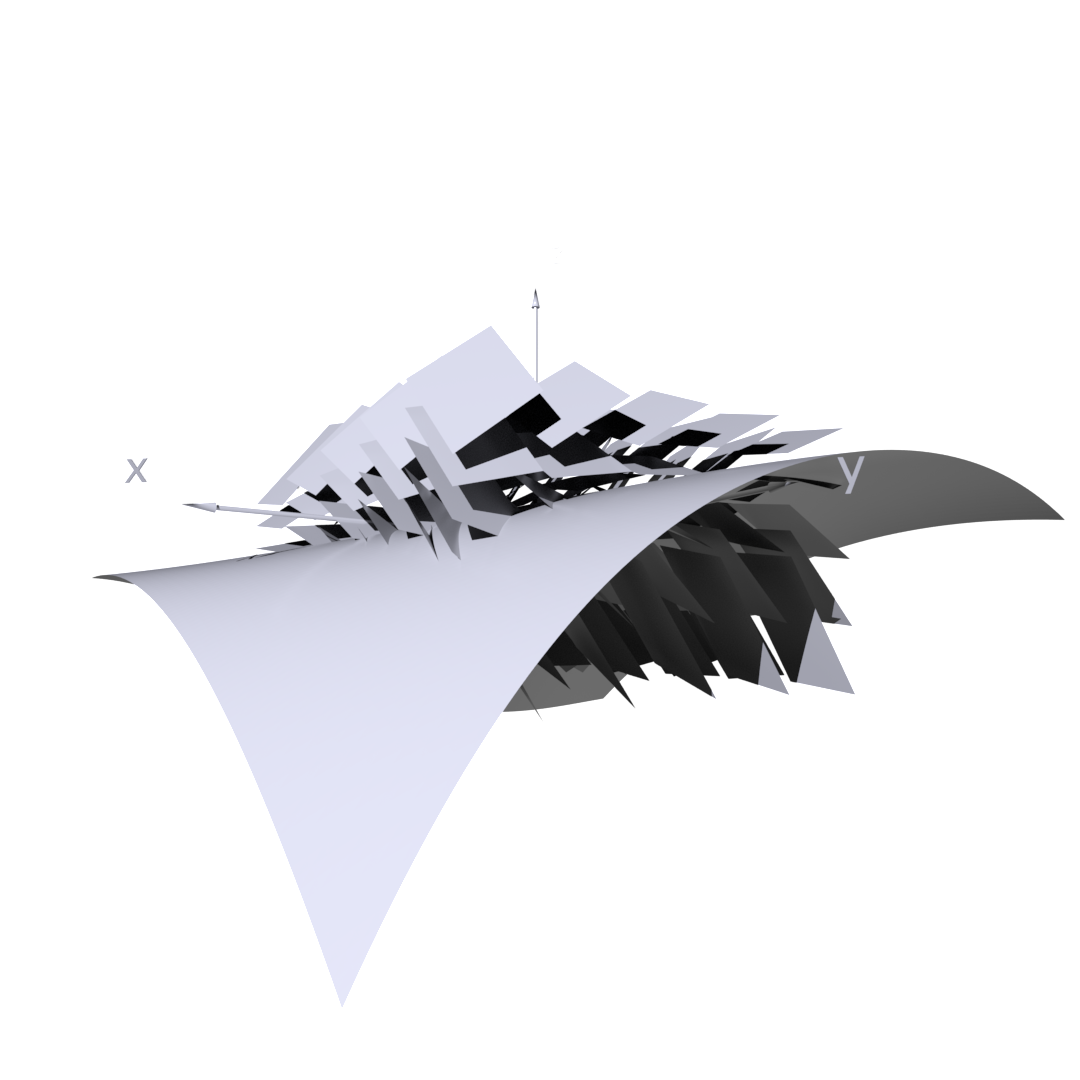}\label{spnode2}}
			\\
			\subfloat{\includegraphics[width=.5\textwidth]{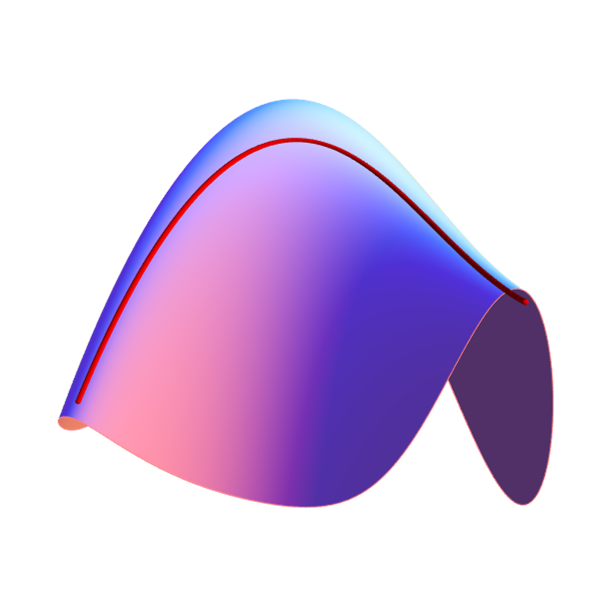}\label{spnode1}}
				
			\caption{Plane field $(xy-y+5z)dx+(x^2+x+y)dy+dz=0$, parabolic surface and the curve $\varphi$ (coloured as red) of parabolic points of node type.}
			\label{cpnode}
		\end{figure}
	
	Consider the vector field
	$\xi(x,y,z)=(x^3-y+z,x+y,1)$. The equation of the plane field $\langle \xi,dr \rangle=0$ is given by $(x^3-y+z)dx+(x+y)dy+dz=0$,
	see the Figure \ref{cpfoco}. 
	The parabolic surface of the plane field is given by $x^3 - (11/4)x^2 - y + z + (1/2)xy + (1/4)y^2=0$. The curve 
	$x=- y + z + (1/4)y^2=0$  is 
	a curve $\varphi$ of parabolic points of focus type, $\varphi(y)=
	\left(0,y,y-(1/4)y^2\right)$. 
	\begin{figure}[H]
			\captionsetup[subfigure]{width=.3\linewidth}
			\centering
			\subfloat{\includegraphics[width=.5\textwidth]{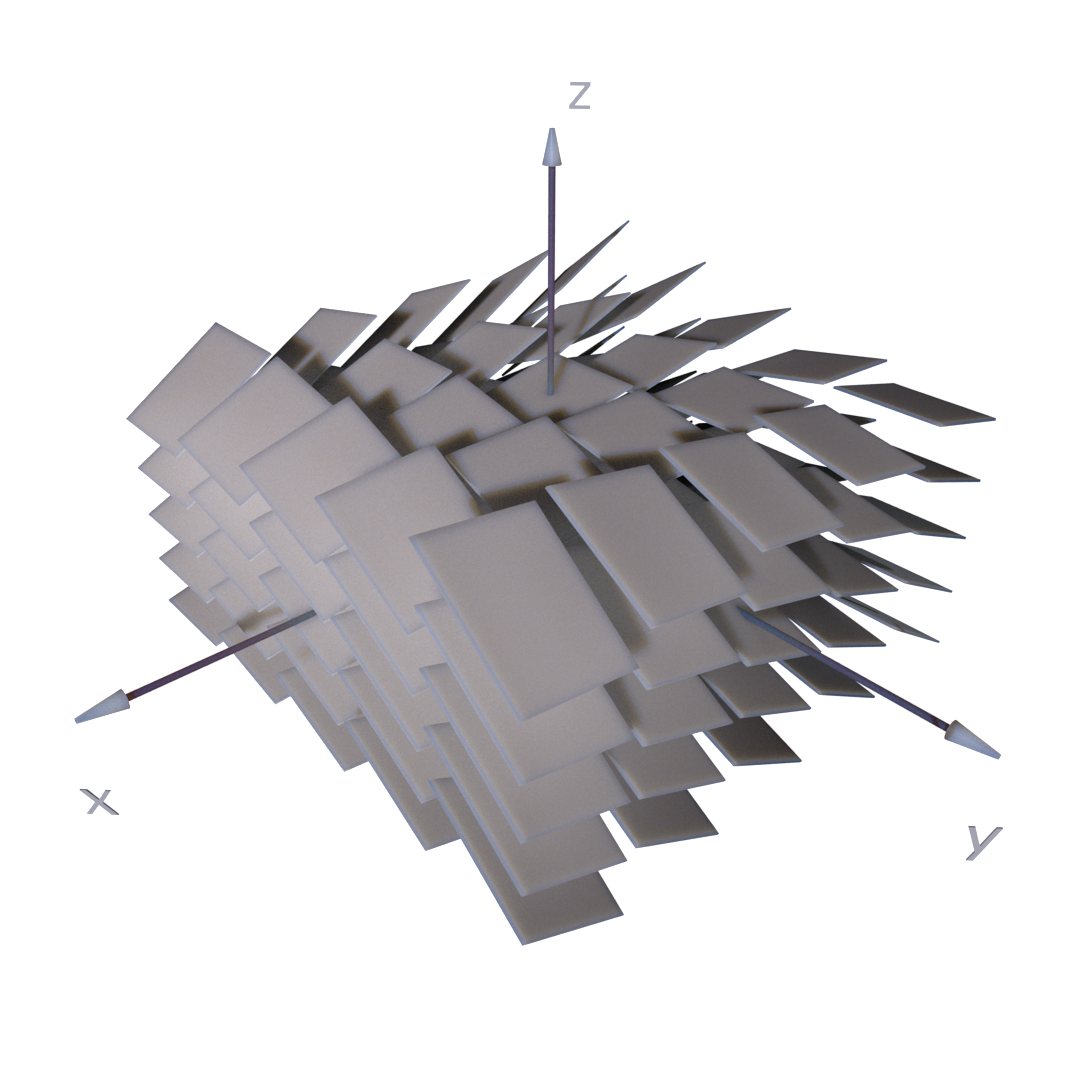}
			\label{cpfoco1}}
			\subfloat{\includegraphics[width=.5\textwidth]{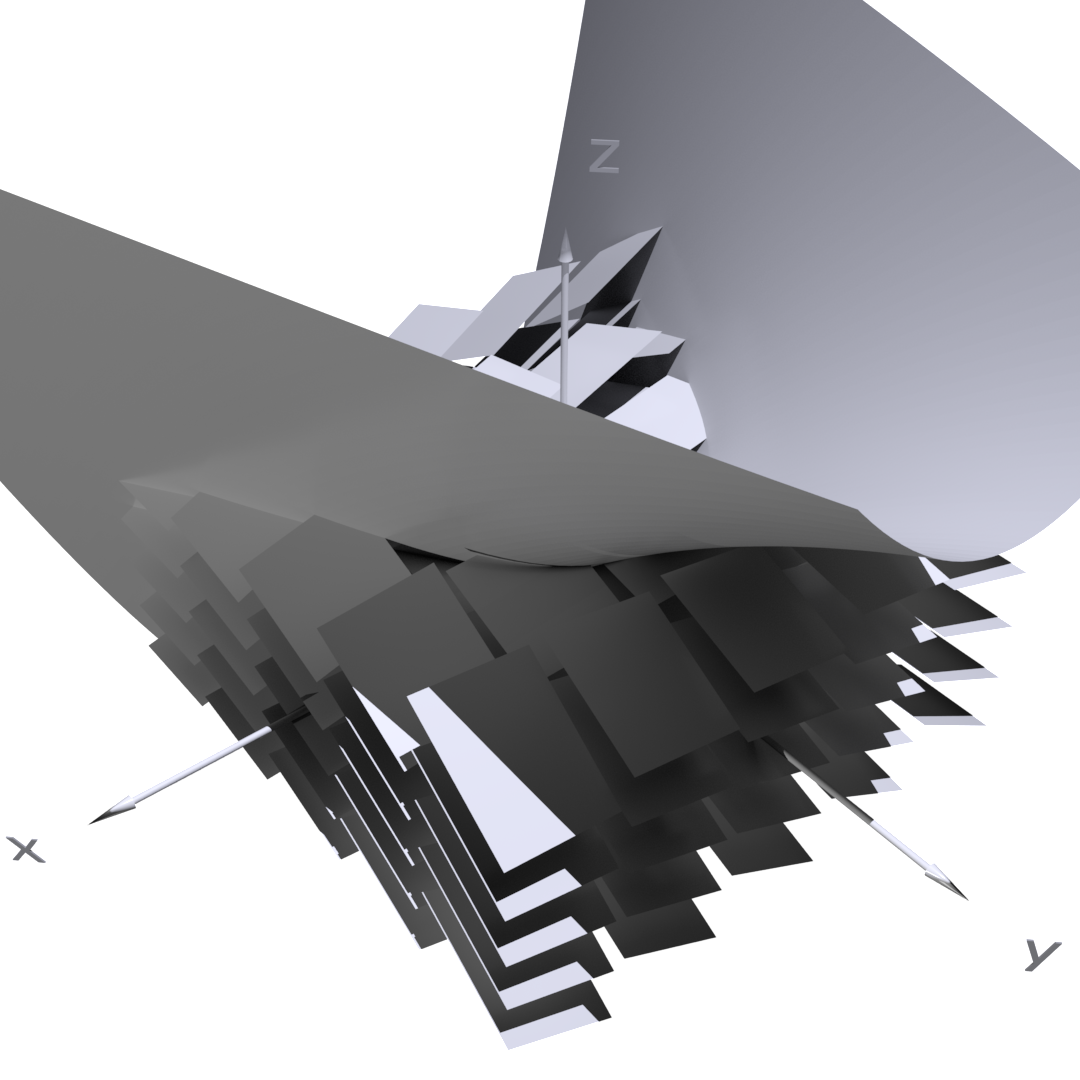}
			\label{spfocoP1}}
			\\
			\subfloat{\includegraphics[width=.5\textwidth]{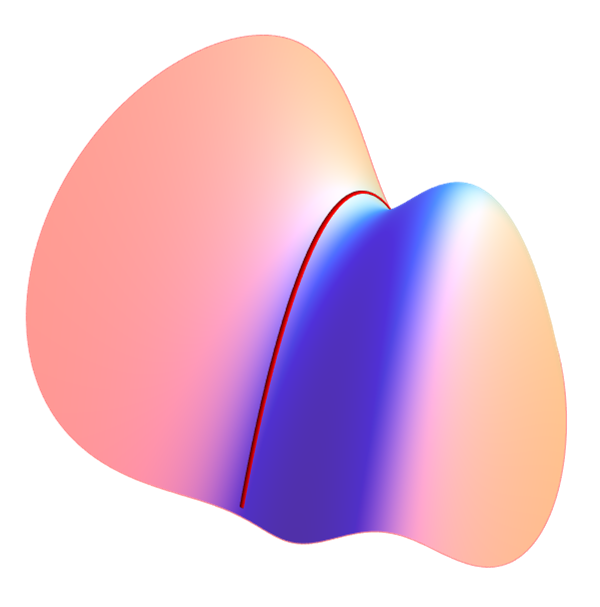}
			\label{spfoco1}}
			\caption{Plane field $(x^3-y+z)dx+(x+y)dy+dz=0$, parabolic surface and the curve $\varphi$ (coloured as red) of parabolic points of focus type.}
			\label{cpfoco}
		\end{figure}
	\end{ex}
	
	\begin{ex} Parabolic point of saddle-node transition type. Consider the vector field
		$\xi(x,y,z)=((2/3)x^3-y+xy+z,x^2+x+y,1)$. The equation of the plane field $\langle \xi,dr \rangle=0$ is given by $((2/3)x^3-y+xy+z)dx+(x^2+x+y)dy+dz=0$,
		see the Figure \ref{cpsaddlenode}. The parabolic surface of the plane field is given by $x^2 + 2y + (1/3)x^3 - z - (1/4)x^4 - (1/2)x^2y - (1/4)y^2=0$. The curve 
		$ x=2y-(1/4)y^2 -z=0$
		is the curve $\varphi(y)=\left(0,y,2 y - (1/4) y^2\right)$, of parabolic points of saddle type $(y<0)$ and node type $(y>0)$ with a saddle-node transition at $y=0$, i.e, the point $\varphi(0)=(0,0,0)$ is a parabolic point of saddle-node transition type. 
		\begin{figure}[H]
			\captionsetup[subfigure]{width=.3\linewidth}
			\centering
			\subfloat{\includegraphics[width=.5\textwidth]{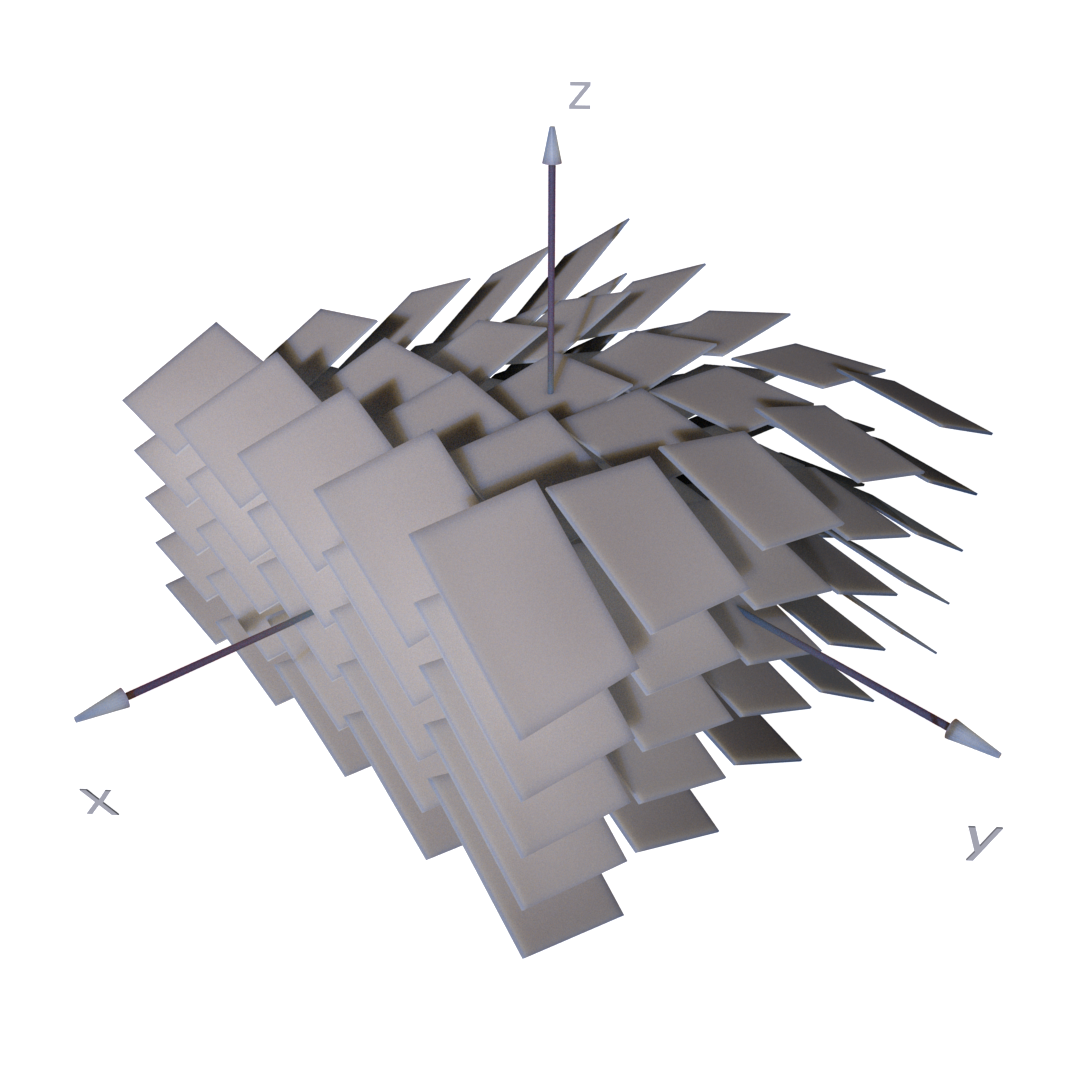}
			\label{cpsaddlenode1}}
			\subfloat{\includegraphics[width=.5\textwidth]{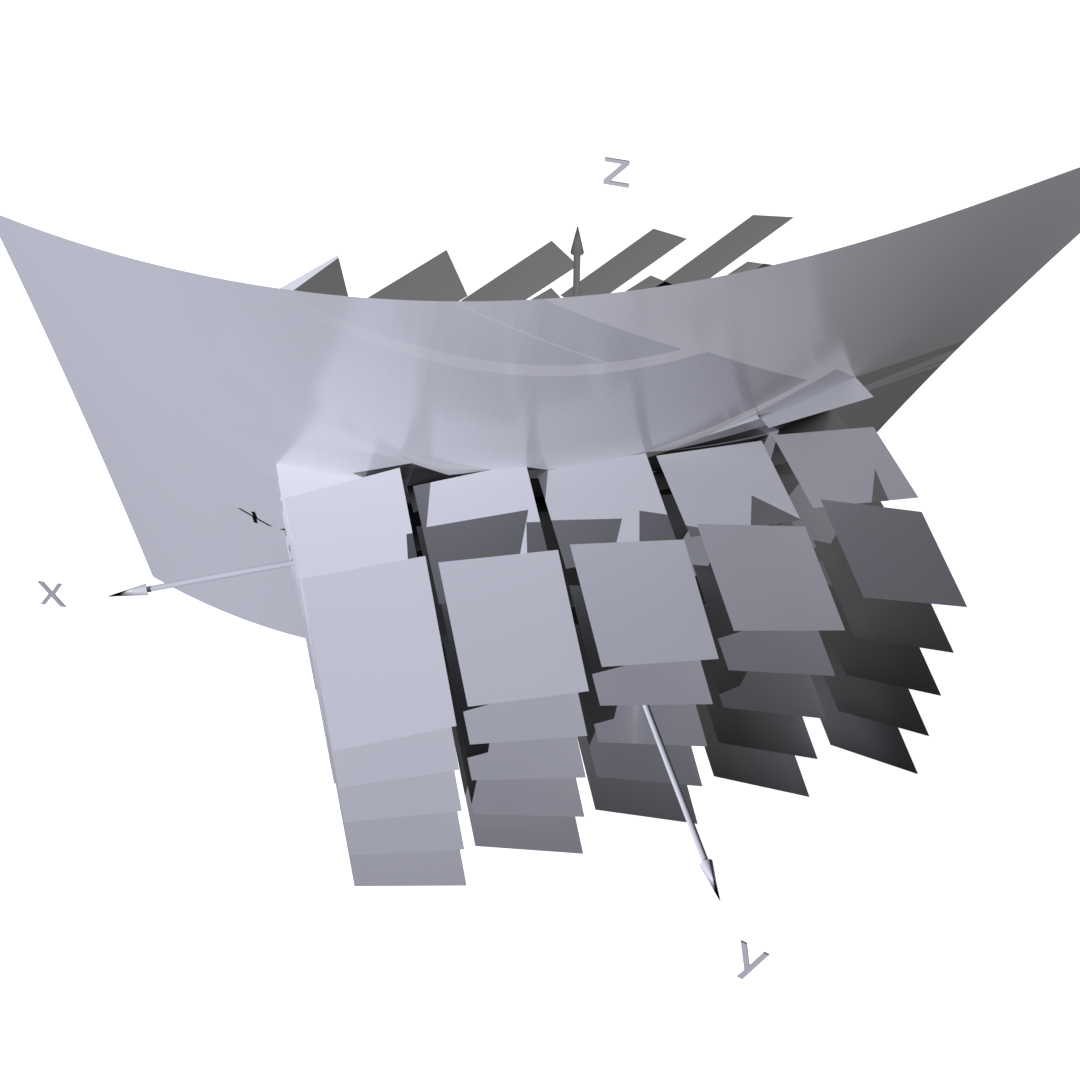}
			\label{cpsaddlenode2}}
			\\
			\subfloat{\includegraphics[width=.5\textwidth]{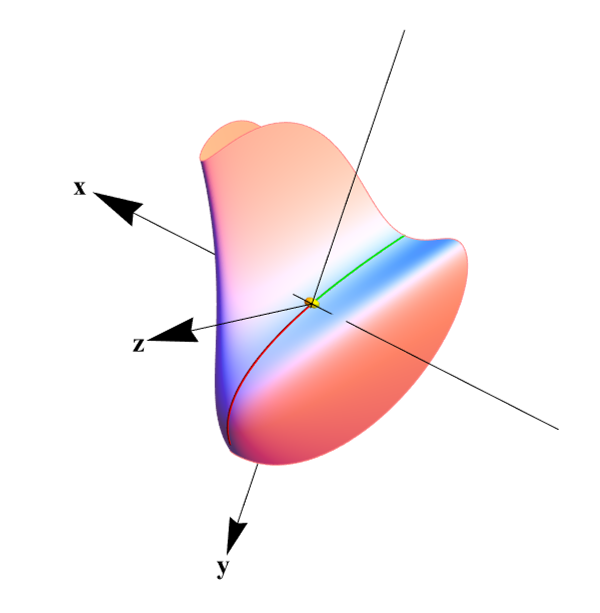}
			\label{spsaddlenode1}}
			\caption{Plane field $((2/3)x^3-y+xy+z)dx+(x^2+x+y)dy+dz=0$, parabolic surface and the curve $\varphi$ of parabolic points of saddle type (coloured as green), node   type (coloured as red) and the parabolic point $\varphi(0)=(0,0,0)$ of saddle-node transition type  (coloured as yellow).}
			\label{cpsaddlenode}
		\end{figure}
		
	\end{ex}
	\begin{ex} Parabolic point of node-focus transition type. Consider the vector field
	$\xi(x,y,z)=((3/4)x^3-y+xy+z,x^2+x+y,1)$. The equation of the plane field $\langle \xi,dr \rangle=0$ is given by $((3/4)x^3-y+xy+z)dx+(x^2+x+y)dy+dz=0$,
		see the Figure \ref{cpnodefocus}. The parabolic surface of the plane field is given by $(5/4)x^2 + 2y + (1/4)x^3 - z - (1/4)x^4 - (1/2)x^2y - (1/4)y^2=0$. The curve $(5/4)x^2 + 2y + (1/4)x^3 - z - (1/4)x^4 - (1/2)x^2y - (1/4)y^2=0$,
		$(1/4)x + x^2 + y=0$  
		is 
		the curve $\varphi(x)=\left(x,-(1/4)x - x^2,-(49/64)x^2 - (1/2)x + (1/4)x^3\right)$, of parabolic points of node type $(x<0)$ and focus type $(x>0)$ with a node-focus transition at $x=0$, i.e, the point $\varphi(0)=(0,0,0)$ is a parabolic point of node-focus transition type.  
		\begin{figure}[H]
			\captionsetup[subfigure]{width=.3\linewidth}
			\subfloat{\includegraphics[width=.5\textwidth]{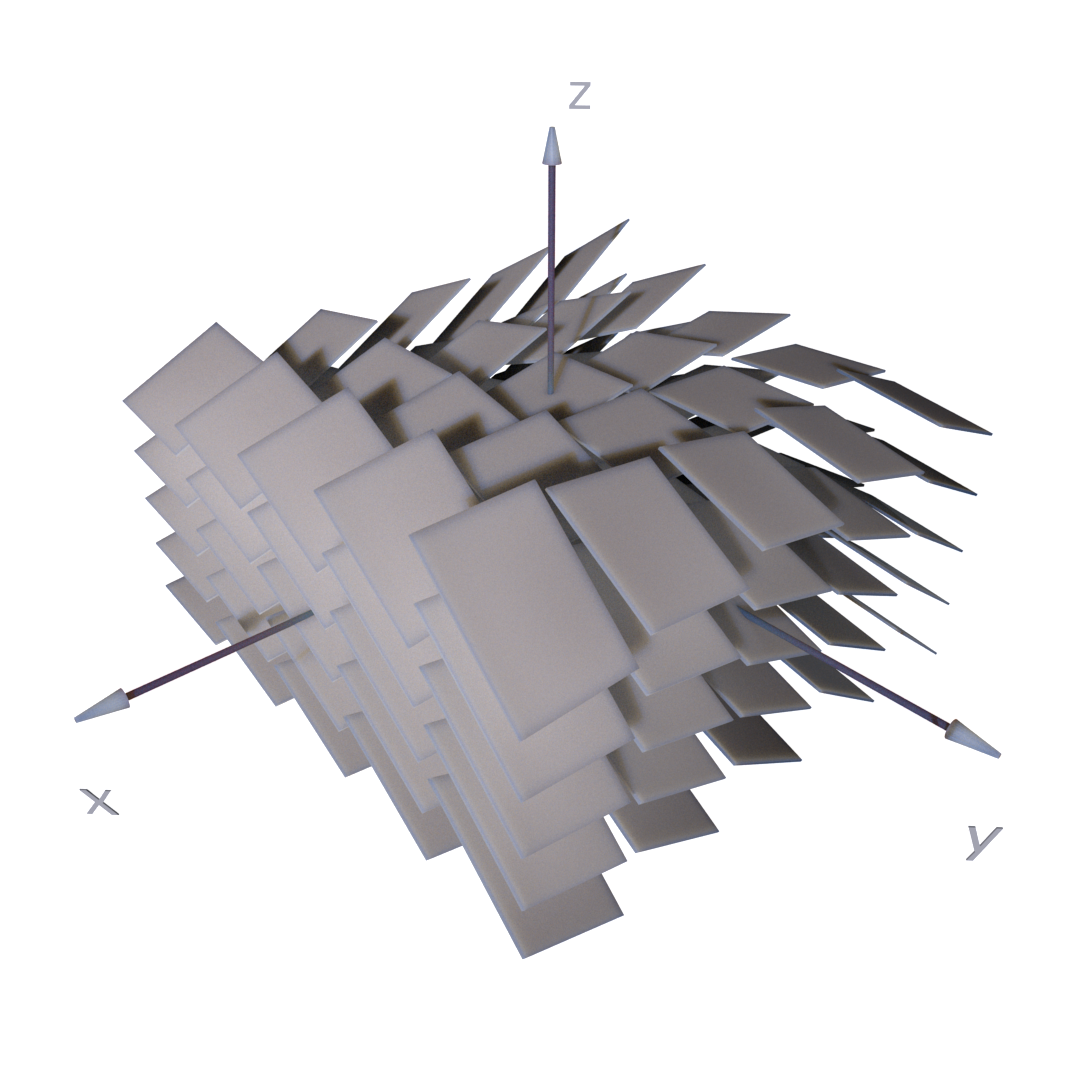}
			\label{cpnodefocus1}}
			\subfloat{\includegraphics[width=.5\textwidth]{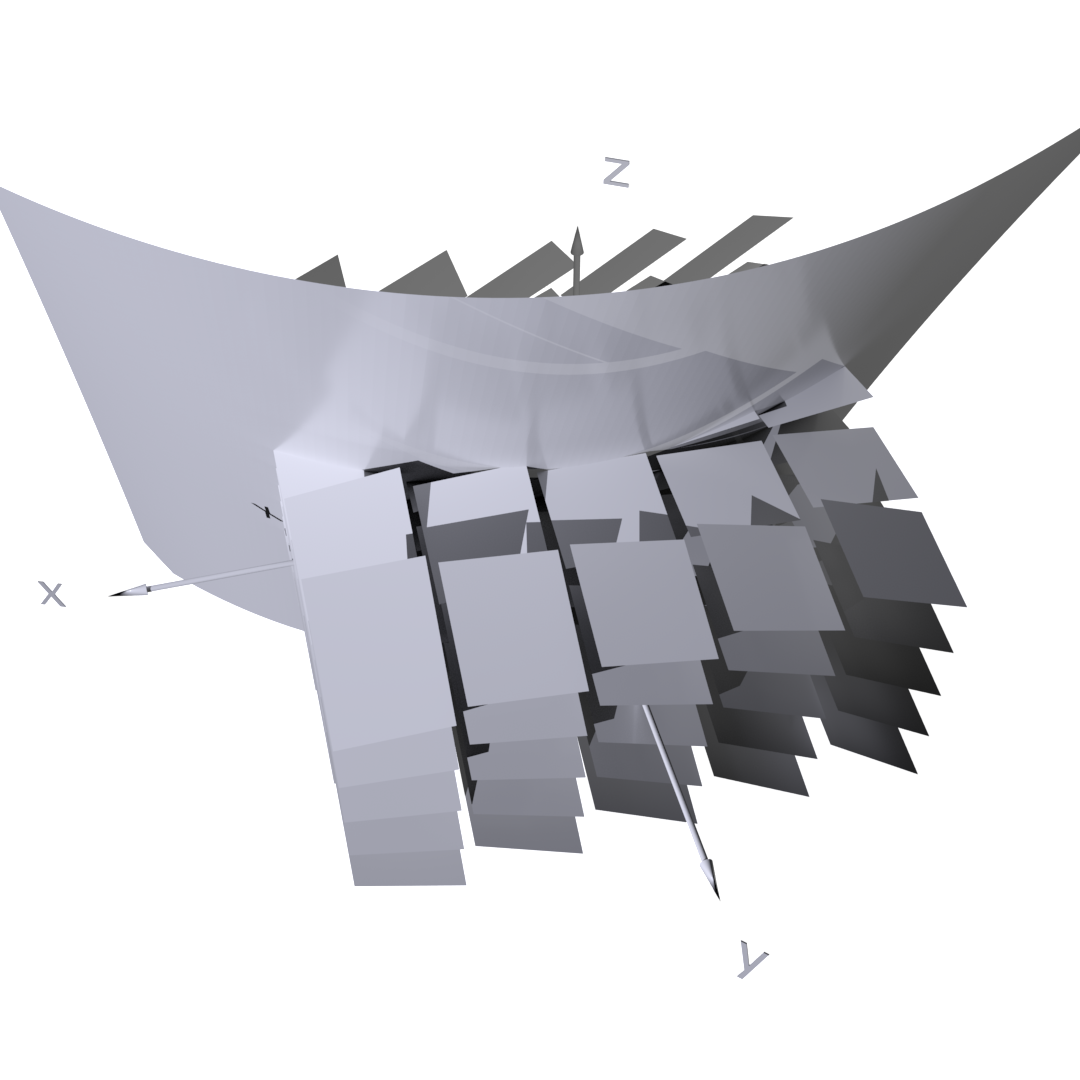}
			\label{cpnodefocus2}}
			\\
			\subfloat{\includegraphics[width=.5\textwidth]{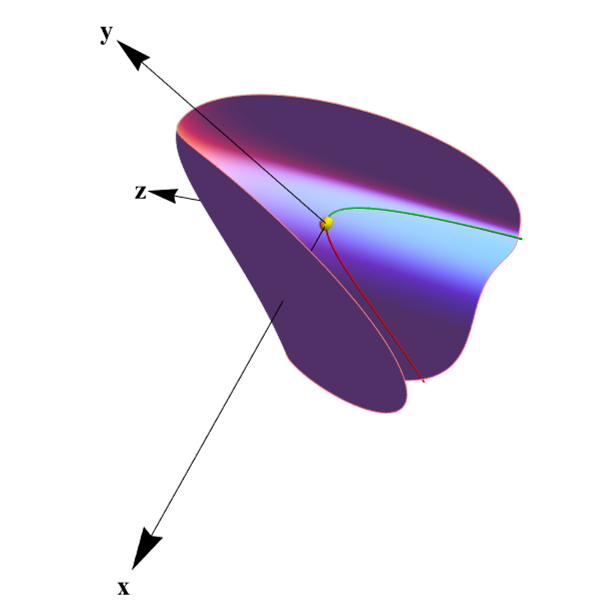}
			\label{cpnodefocus3}}
			\caption{Plane field $((3/4)x^3-y+xy+z)dx+(x^2+x+y)dy+dz=0$, parabolic surface and the curve $\varphi$ of parabolic points of node type (coloured as green), focus   type (coloured as red) and the parabolic point $\varphi(0)=(0,0,0)$ of node-focus transition type  (coloured as yellow).}
			\label{cpnodefocus}
		\end{figure}
		
	\end{ex}
	
	\begin{ex} Parabolic point with a pair of complex eigenvalues crossing the imaginary axis. Consider the vector field
		$\xi(x,y,z)=(x^3 - xy - y + z,x + y,1)$. The equation of the plane field $\langle \xi,dr \rangle=0$ is given by $(x^3 - xy - y + z)dx+(x + y)dy+dz=0$,
		see the Figure \ref{cphopfhyperbolic}. The parabolic surface of the plane field is given by $x^3 - 2x^2 + (1/4)y^2 + z=0$. The curve 
		$x^3 - 2x^2 + (1/4)y^2 + z=0$, $4x - y=0$,
		given by $\varphi(x)=\left(x,-4x, -x^3 - 2x^2\right)$, is 
		a curve of parabolic points of focus type. The point $(0,0,0)$ is a parabolic point  with a pair of complex eigenvalues crossing the imaginary axis. Furthermore, $(0,0,0)$ is an hyperbolic Hopf parabolic point
		\begin{figure}[H]
			\captionsetup[subfigure]{width=.3\linewidth}
			\centering
			\subfloat{\includegraphics[width=.5\textwidth]{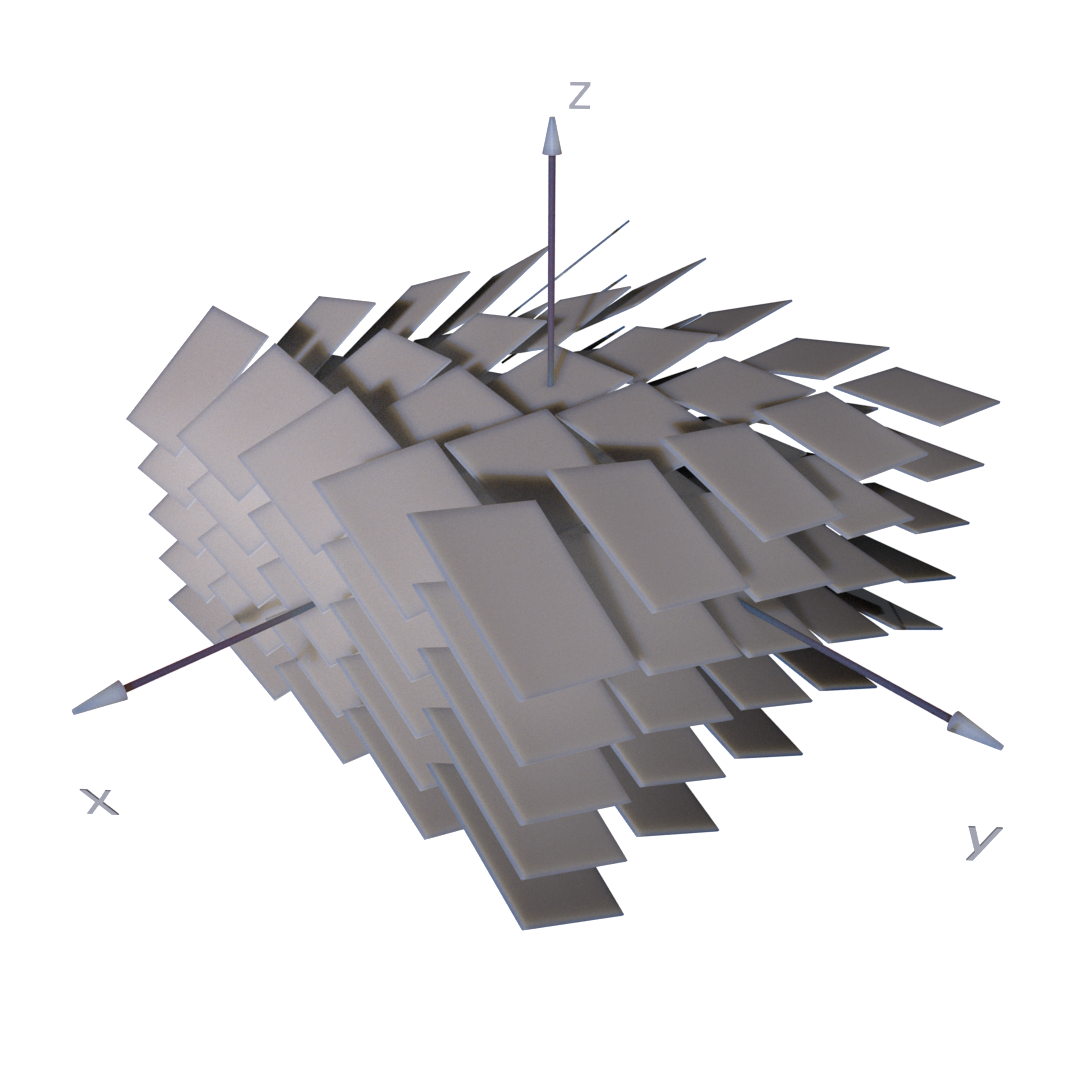}
			\label{cphopfhyperbolic1}}
			\subfloat{\includegraphics[width=.5\textwidth]{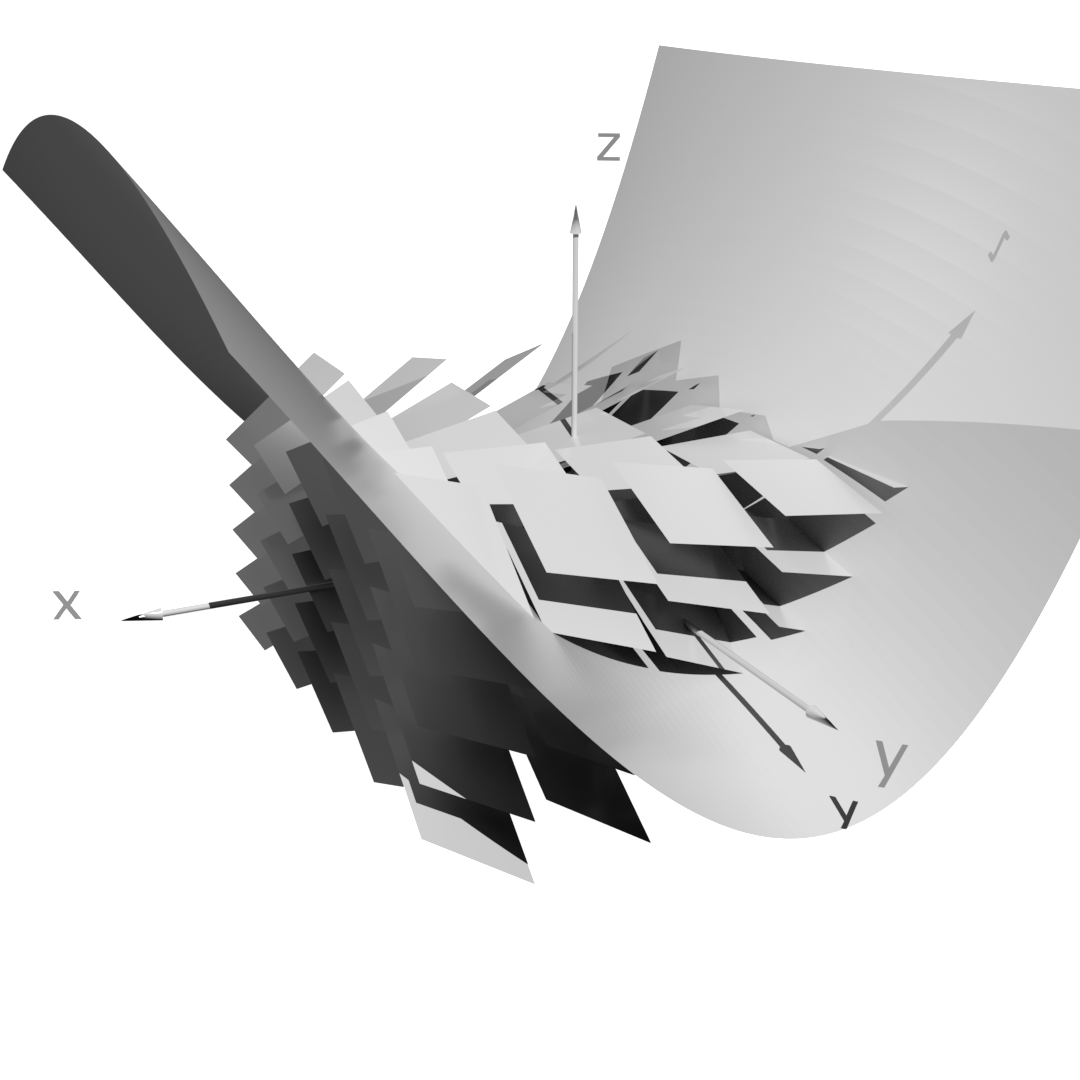}
			\label{cphopfhyperbolic2}}
			\\
			\subfloat{\includegraphics[width=.5\textwidth]{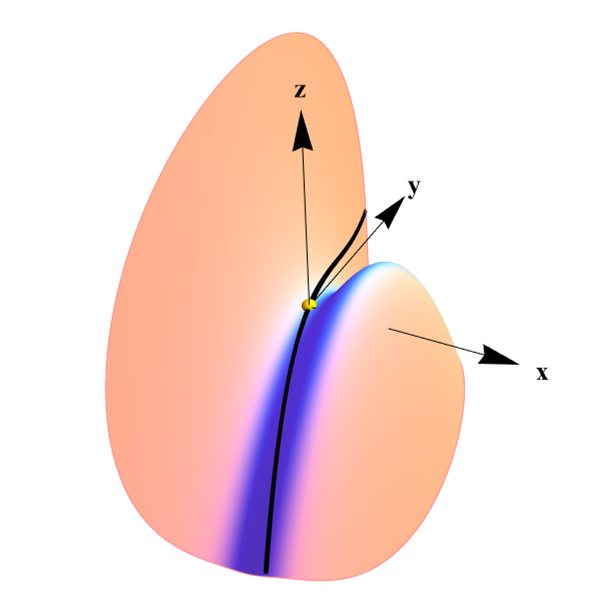}
			\label{cphopfhyperbolic3}}
			\caption{Plane field $(x^3 - xy - y + z)dx+(x + y)dy+dz=0$,
			parabolic surface and the curve $\varphi$ of parabolic points of focus type (coloured as black). The point $(0,0,0)$ (coloured as yellow) is a parabolic point  with a pair of complex eigenvalues crossing the imaginary axis. Furthermore, $(0,0,0)$ is an hyperbolic Hopf parabolic point.}
			\label{cphopfhyperbolic}
		\end{figure}
		
		 Consider the vector field
		$\xi(x,y,z)=(x^3 - 3xy - 3y + z,3x + y,1)$. The equation of the plane field $\langle \xi,dr \rangle=0$ is given by $(x^3 - 3xy - 3y + z)dx+(3x + y)dy+dz=0$,
		see the Figure \ref{cphopfelliptic}. The parabolic surface of the plane field is given by $x^3 + 6x^2 + z + (1/4)y^2=0$. The curve 
		$y=x^3 + 6x^2 + z + (1/4)y^2=0$, given by $\varphi(x)=\left(x, -4x, -x^3 - 10x^2\right)$, is 
		a curve of parabolic points of focus type. The point $(0,0,0)$ is a parabolic point  with a pair of complex eigenvalues crossing the imag
		inary axis. Furthermore, $(0,0,0)$ is an elliptic Hopf parabolic point. 
		\begin{figure}[H]
			\captionsetup[subfigure]{width=.3\linewidth}
			\centering
			\subfloat{\includegraphics[width=.5\textwidth]{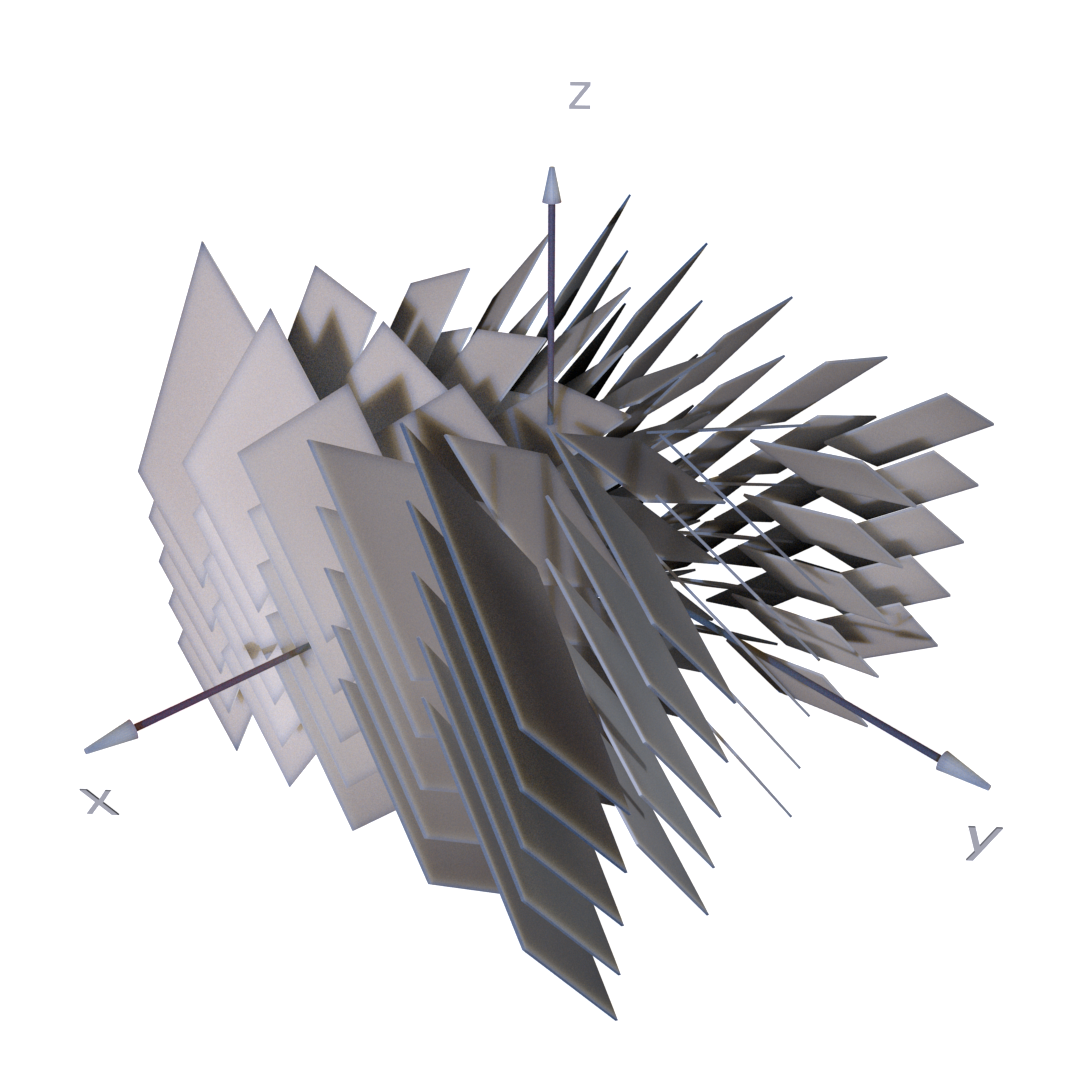}
			\label{cphopfelliptic1}}
			\subfloat{\includegraphics[width=.5\textwidth]{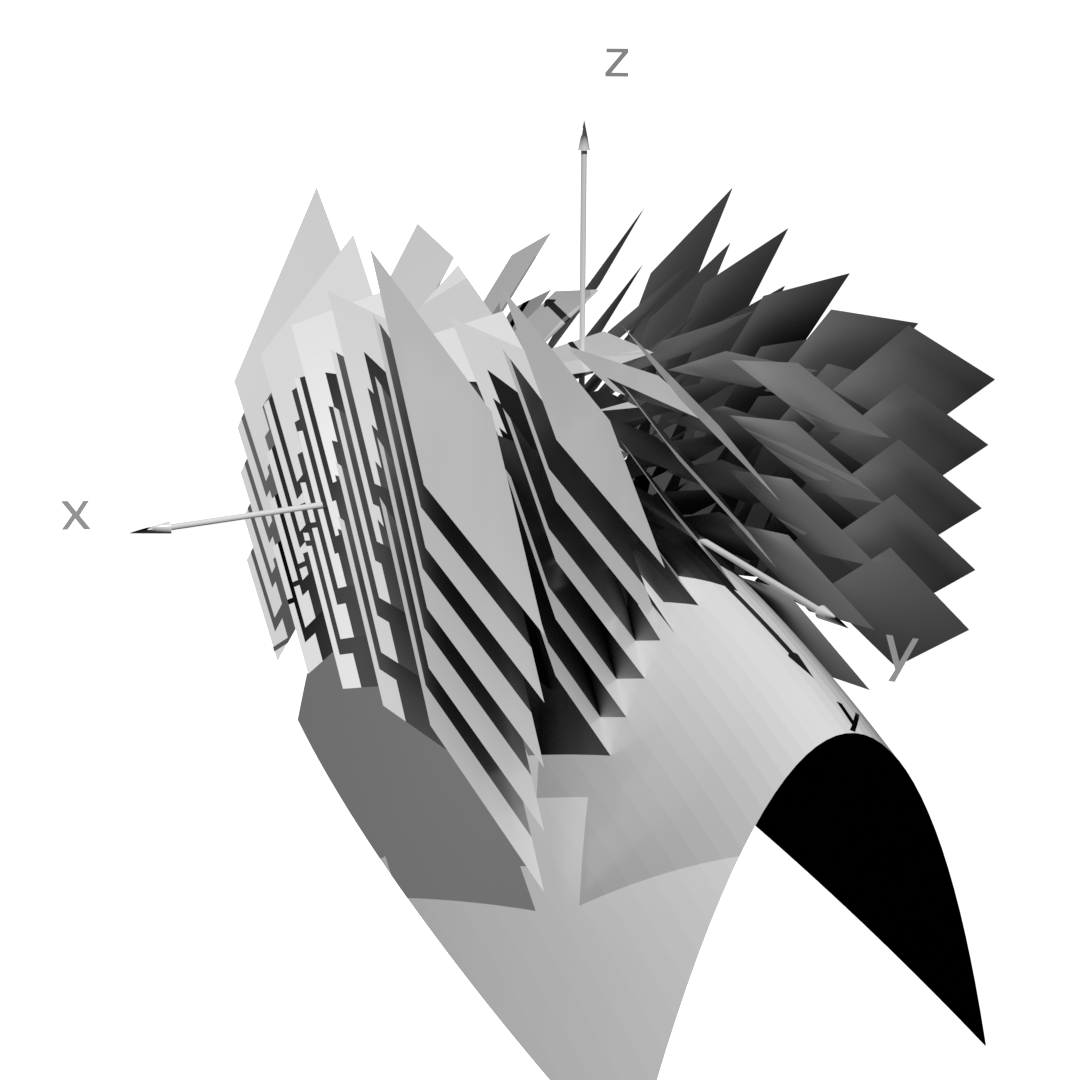}
			\label{cphopfelliptic2}}
			\\
			\subfloat{\includegraphics[width=.5\textwidth]{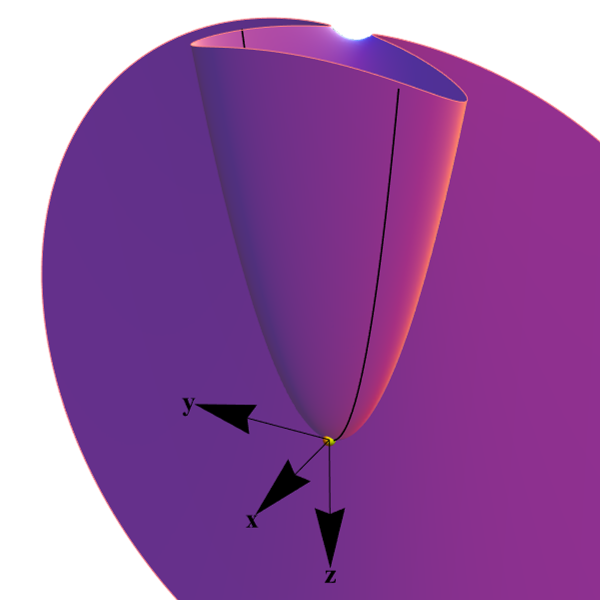}
			\label{cphopfelliptic3}}
			\caption{Plane field $(x^3 - 3xy - 3y + z)dx+(3x + y)dy+dz=0$,
			parabolic surface and the curve $\varphi$ of parabolic points of focus type (coloured as black). The point $(0,0,0)$ (coloured as yellow) is a parabolic point  with a pair of complex eigenvalues crossing the imaginary axis. Furthermore, $(0,0,0)$ is an elliptic Hopf parabolic point. }
			\label{cphopfelliptic}
		\end{figure}
	\end{ex}

	\subsection{Parabolic set}
	
	\begin{ex} Let $\xi(x,y,z)=(f(x,y,z),g(x,y,z),1)$,
		where $f(x,y,z)=-x$ and
		$g(x,y,z)=\frac{2x^3}{3}+2y^2x+2z^2x-x -y$.
		Then the equation of the plane field is given by
		\begin{equation*}
		\begin{split}
		-xdx+\left(\frac{2}{3}x^3+2x(y^2+z^2)-x-y\right)dy+dz=0
		\end{split}
		\end{equation*}
		By the Proposition \ref{propEx1}, the equation of the parabolic set $\mathcal{K}=0$ is given by
		\begin{equation*}
		\begin{split}
		\mathcal{K} &= 48x^2z^3-48x^4z^2-16x^4z+48x^2y^2z-12x^4-24x^2y^2-24x^2z^2
		-12y^4\\&-24y^2z^2
		-12z^4-24x^2z-48xyz+12x^2-48xy+12y^2
		+12z^2+9=0,
		\end{split}
		\end{equation*}
		\noindent and  a compact component   is topologically a sphere, see Figure \ref{exa}.
		
		\begin{figure}[H]
			\captionsetup[subfigure]{width=.3\linewidth}
			\centering
			\subfloat{\includegraphics[width=.6\textwidth]{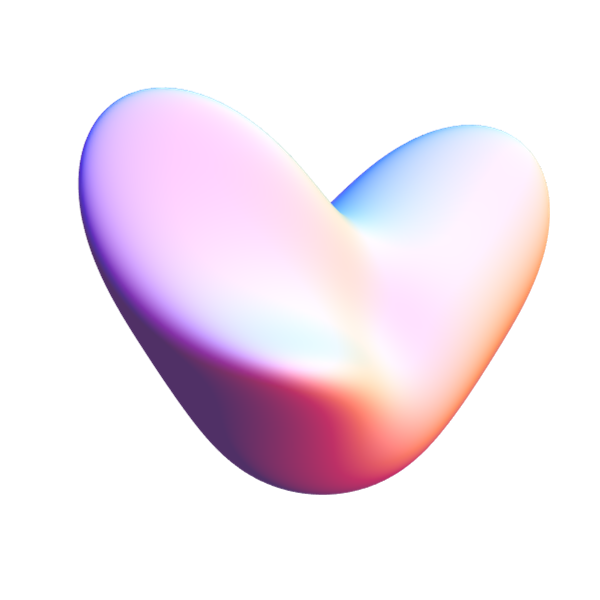}\label{exa1}}
			\caption{Parabolic set of the plane field generated by the vector field $\xi$ has a compact component which is  topologically a sphere.}
			\label{exa}
		\end{figure}
	\end{ex}
	
	\section*{Acknowledgments} The second author is fellow of CNPq and coordinator of Project PRONEX/ CNPq/FAPEG 2017 10 26 7000 508.
	
	\bibliographystyle{abbrv}
	\bibliography{referencias}

\end{document}